\documentclass{article}
\usepackage{amsthm,amsfonts,amsmath}
\usepackage{amscd,amssymb,latexsym,epsfig}
\usepackage[all]{xy}
\title{Morse homology for the heat flow
       \\
       }
\author{Joa Weber
        \\
        UC Berkeley
       }
\date{19 March 2010
     }
\newtheorem{theorem}{Theorem}[section]
\newtheorem{corollary}[theorem]{Corollary}

\newtheorem{prop}[theorem]{Proposition}
\newtheorem{lemma}[theorem]{Lemma}
\newtheorem{proposition}[theorem]{Proposition}

\theoremstyle{definition}

\newtheorem{hypothesis}[theorem]{Hypothesis}
\newtheorem{remark}[theorem]{Remark}

\theoremstyle{remark}
\newtheorem*{notation}{Notation}
\renewcommand{\1}{{{\mathchoice {\rm 1\mskip-4mu l} {\rm 1\mskip-4mu l}
{\rm 1\mskip-4.5mu l} {\rm 1\mskip-5mu l}}}}
\newcommand{\HH}{{\mathbb{H}}}

\newcommand{\N}{{\mathbb{N}}}

\newcommand{\R}{{\mathbb{R}}}

\newcommand{\Z}{{\mathbb{Z}}}
\newcommand{\Bb}{{\mathcal{B}}}
\newcommand{\Cc}{{\mathcal{C}}}   
\newcommand{\Dd}{{\mathcal{D}}}
\newcommand{\Ee}{{\mathcal{E}}}
\newcommand{\Ff}{{\mathcal{F}}}
\newcommand{\Hh}{{\mathcal{H}}}

\newcommand{\Ll}{{\mathcal{L}}}   
\newcommand{\Mm}{{\mathcal{M}}}   

\newcommand{\Oo}{{\mathcal{O}}}
\newcommand{\Pp}{{\mathcal{P}}}

\newcommand{\Ss}{{\mathcal{S}}}

\newcommand{\Vv}{{\mathcal{V}}}
\newcommand{\Ww}{{\mathcal{W}}}

\newcommand{\Zz}{{\mathcal{Z}}}
\newcommand{\coker}{{\rm coker\, }}  
\newcommand{\im}{{\rm im\, }}      
\newcommand{\dom}{{\rm dom\, }}      

\newcommand{\diag}{{\rm diag\,}}     
\newcommand{\cl}{{\rm cl}}         
\newcommand{\diam}{{\rm diam\,}}   
\newcommand{\INT}{{\rm int}}       
\newcommand{\supp}{{\rm supp}}     
\newcommand{\INDEX}{{\rm index}}   
\newcommand{\IND}{{\rm ind}}       

\newcommand{\grad}{{\rm grad }}    
\newcommand{\Vol}{{\rm Vol}}          
\newcommand{\CM}{{\rm CM}}            
\newcommand{\HM}{{\rm HM}}            

\newcommand{\norm}{{\rm norm}}

\newcommand{\eps}{{\varepsilon}}

\newcommand{\Om}{{\Omega}}

\renewcommand{\O}{{\rm O}}

\newcommand{\inner}[2]{\langle #1, #2\rangle}

\def\NABLA#1{{\mathop{\nabla\kern-.5ex\lower1ex\hbox{$#1$}}}}
\def\Nabla#1{\nabla\kern-.5ex{}_{#1}}
\def\Tabla#1{\Tilde\nabla\kern-.5ex{}_{#1}}
\def\abs#1{\mathopen|#1\mathclose|}   
\def\Abs#1{\left|#1\right|}            
\def\norm#1{\mathopen\|#1\mathclose\|}
\def\Norm#1{\left\|#1\right\|}

\renewcommand{\Tilde}{\widetilde}

\newcommand{\p}{{\partial}}

\begin{document}

\maketitle


\begin{abstract}
We use the heat flow on the loop space
of a closed Riemannian manifold to construct
an algebraic chain complex.
The chain groups are generated
by perturbed closed geodesics.
The boundary operator is defined
in the spirit of Floer theory
by counting, modulo time shift, 
heat flow trajectories that
converge asymptotically to
nondegenerate
closed geodesics 
of Morse index difference one.
\end{abstract}

\tableofcontents

\section{Introduction}
\label{sec:plan}

Let $M$ be a closed Riemannian manifold
and denote by $\nabla$
the Levi-Civita connection
and by $\Ll M$ the {\bf loop space},
that is the space of free loops
$C^\infty(S^1,M)$.
For $x:S^1\to M$
consider the action functional
$$
     \Ss_V(x) = \int_0^1 \left(
     \frac12\Abs{\dot x(t)}^2
     -V(t,x(t)) \right)
     dt.
$$
Here and throughout
we identify $S^1=\R/\Z$
and think of $x\in\Ll M$
as a smooth map $x:\R\to M$
which satisfies $x(t+1)=x(t)$.
\emph{Smooth} means $C^\infty$ smooth.
The potential is a smooth function
$V:S^1\times M\to \R$
and we set $V_t(q):=V(t,q)$.
The critical points of $\Ss_V$
are the 1-periodic solutions
of the ODE
\begin{equation}\label{eq:1-periodic}
     \Nabla{t}\dot x=-\nabla V_t(x),
\end{equation}
where $\nabla V_t$ denotes the gradient and 
$\Nabla{t}\dot x$ denotes the covariant derivative, 
with respect to the Levi-Civita connection,
of the vector field $\dot x:=\frac{d}{dt} x$ 
along the loop $x$ in direction $\dot x$.
By $\Pp=\Pp(V)$ we denote the set of 1-periodic
solutions of~(\ref{eq:1-periodic}).
These solutions are called 
{\bf perturbed closed geodesics},
since in the case $V=0$ they are closed geodesics.

From now on we assume that
$\Ss_V$ is a Morse function on the loop space,
i.e. the 1-periodic solutions
of~(\ref{eq:1-periodic}) are all nondegenerate.
We proved in~\cite{Joa-INDEX}
that $\Ss_V$ is Morse for a generic potential $V_t$
and that in this case the set
$$
     \Pp^a(V):=\{ x\in\Pp(V)\mid \Ss_V(x)\le a\}
$$
is finite for every real number $a$.
By $E_x^u$ we denote
the eigenspace corresponding
to negative eigenvalues of the
Hessian of $\Ss_V$ at $x\in\Pp^a(V)$.
The dimension of $E_x^u$ is finite
and called the {\bf Morse index of \boldmath$x$}. 
Choose an orientation $\langle x\rangle$
of the vector space $E_x^u$. By $\nu=\nu(V,a)$ we denote a 
choice of orientations
for all $x\in\Pp^a(V)$.
Now consider the $\Z$-module
$$
     \CM^a_*=\CM^a_*(V,\nu):=\bigoplus_{x\in\Pp^a(V)} \Z 
     \langle x\rangle.
$$
It is graded by the Morse index.

If in addition $\Ss_V$ is Morse--Smale, then
the module $\CM^a_*$ carries a boundary operator 
$\p=\p(V,a,\nu)$
defined as follows.
Consider the (negative) $L^2$ gradient flow lines
of $\Ss_V$ on the loop space. These are
solutions $u:\R\times S^1\to M$ of
the {\bf heat equation}
\begin{equation}\label{eq:heat-V}
     \p_su-\Nabla{t}\p_tu-\nabla V_t(u)=0
\end{equation}
satisfying
\begin{equation}\label{eq:limit-V}
     \lim_{s\to\pm\infty} u(s,t)
     =x^\pm(t),\qquad
     \lim_{s\to\pm\infty} \p_su(s,t)
     =0,
\end{equation}
where $x^\pm\in\Pp(V)$.
The limits are uniform in $t$
together with the first partial $t$-derivative,
i.e. in $C^1(S^1)$; see remark~\ref{rem:limit}.
By definition the {\bf moduli space}
$\Mm(x^-,x^+;V)$
is the space of solutions of~(\ref{eq:heat-V})
and~(\ref{eq:limit-V}).
The action functional $\Ss_V$ 
is called {\bf Morse--Smale below level \boldmath$a$}
if the operator $\Dd_u$
obtained by linearizing~(\ref{eq:heat-V})
is onto as a linear operator
between appropriate Banach spaces,
see~(\ref{eq:norms}) below,
and this is true
for all $u\in\Mm(x^-,x^+;V)$ and $x^\pm\in\Pp^a(V)$.
Note that Morse--Smale implies Morse
(use that $u:=x\in\Mm(x,x;V)$).
Under the Morse--Smale hypothesis
the space $\Mm(x^-,x^+;V)$
is a smooth manifold whose dimension
is equal to the difference
of the Morse indices of the
perturbed closed geodesics $x^\pm$.
In the case of index difference one
it follows that the quotient
$\Mm(x^-,x^+;V)/\R$
by the (free) time shift action is a finite set.
Counting these elements with
appropriate signs
defines the boundary operator $\p$ on $\CM^a_*$.
The Morse complex $\left(\CM^a_*,\p\right)$ 
is called the {\bf heat flow complex}
and the corresponding homology groups
$\HM_*^a(\Ll M,\Ss_V)$ are called
{\bf heat flow homology}.

In chapter~\ref{sec:transversality}
we explain how to perturb
the Morse function $\Ss_V$
by a regular perturbation $v\in\Oo^a_{reg}$
to achieve the Morse--Smale condition 
without changing the set of critical points.
By definition heat flow homology of $\Ss_V$
is then equal to heat flow homology
of the perturbed functional.
It is an open question
if $\Ss_V$ is Morse--Smale
for a generic potential $V_t$.
In section~\ref{sec:perturbations}
we introduce a class of 
{\bf abstract perturbations} 
$\Vv:\Ll M\to\R$
for which we can establish transversality.
In contrast we call the potentials $V_t$
{\bf geometric perturbations}.

\begin{theorem}
\label{thm:d^2=0}
Let $V\in C^\infty(S^1\times M)$ be a potential
such that $\Ss_V$ is Morse
and let $a$ be a regular value of $\Ss_V$.
Take a choice of orientations $\nu=\nu(V,a)$
and fix a regular perturbation $v\in\Oo^a_{reg}$.
Then $\p=\p(V,a,\nu,v)$ satisfies
$\p\circ\p=0$.
Furthermore, {\bf heat flow homology} defined by
$$
     \HM_*^a(\Ll M,\Ss_V)
     :=\frac{\ker \p(V,a,\nu,v)}{\im \p(V,a,\nu,v)}
$$
is independent of the choice of orientations $\nu$ 
and the regular perturbation $v$.
\end{theorem}

The construction of the Morse complex
in finite dimensions goes back
to Thom~\cite{THOM-49}, Smale~\cite{SMALE-60,SMALE-61}, 
and Milnor~\cite{MILNOR-h}.
It was rediscovered by Witten~\cite{Wi82}
and extended to infinite dimensions
by Floer~\cite{FLOER4,FLOER5}. We refer 
to~\cite{AM-LecsMCInfDimMfs}
for an extensive historical account.

\subsection{Perturbations}
\label{sec:perturbations}

We introduce
a class of abstract perturbations of 
equation~(\ref{eq:heat}) for which
the analysis works.
Later in 
section~\ref{subsec:perturbations-banach}
we extract a countable subset and
construct a separable Banach space
of perturbations for which transversality works.
The abstract perturbations take the 
form of smooth maps 
$
\Vv:\Ll M\to\R.
$
For $x\in\Ll M$ let $\grad\Vv(x)\in\Om^0(S^1,x^*TM)$
denote the $L^2$-gradient of $\Vv$; it is defined by 
$$
\int_0^1\inner{\grad\Vv(u)}{\p_su}\,dt = \frac{d}{ds}\Vv(u)
$$
for every smooth path $\R\to\Ll M:s\mapsto u(s,\cdot)$. 
The \emph{covariant Hessian of $\Vv$} at a loop $x:S^1\to M$
is the operator
$$
\Hh_\Vv(x):\Om^0(S^1,x^*TM)\to\Om^0(S^1,x^*TM)
$$
defined by 
\begin{equation}\label{eq:Hess-Vv}
\Hh_\Vv(u)\p_su := \Nabla{s}\grad\Vv(u)
\end{equation}
for every smooth map 
$\R\to\Ll M:s\mapsto u(s,\cdot)$.
The axiom~(V1) below asserts that this Hessian is a zeroth 
order operator. We impose the following conditions on $\Vv$;
here $\abs{\cdot}$ denotes the pointwise absolute value 
at $(s,t)\in\R\times S^1$ and $\norm{\cdot}_{L^p}$ denotes 
the $L^p$-norm over $S^1$ at time $s$. 
Although condition~(V1) and the first part of~(V2)
are special cases of~(V3) we state the axioms in 
the form below, because some of our results don't require
all the conditions to hold.

\begin{description}
\item[(V0)]
$\Vv$ is continuous
with respect to the $C^0$ topology on $\Ll M$.
Moreover, there is a constant $C=C(\Vv)$ such that
$$
\sup_{x\in\Ll M}\left|\Vv(x)\right|
+\sup_{x\in\Ll M}\left\|\grad\Vv(x)\right\|_{L^\infty(S^1)}
\le C.
$$
\item[(V1)]
There is a constant $C=C(\Vv)$ such that
\begin{align*}
\Abs{\Nabla{s}\grad\Vv(u)}
&\le C\bigl(\left|\p_su\right|+\Norm{\p_su}_{L^1}\bigr), \\
\left|\Nabla{t}\grad\Vv(u)\right|
&\le C\Bigl(1+\left|\p_tu\right|\Bigr)
\end{align*}
for every smooth map $\R\to\Ll M:s\mapsto u(s,\cdot)$ 
and every $(s,t)\in\R\times S^1$. 
\item[(V2)]
There is a constant $C=C(\Vv)$ such that
\begin{align*}
\Abs{\Nabla{s}\Nabla{s}\grad\Vv(u)} 
&\le C\Bigl(\Abs{\Nabla{s}\p_su} 
+ \Norm{\Nabla{s}\p_su}_{L^1} 
+ \bigl(\Abs{\p_su} + \Norm{\p_su}_{L^2}\bigr)^2
\Bigr), \\
\Abs{\Nabla{t}\Nabla{s}\grad\Vv(u)} 
&\le C\Bigl(
\Abs{\Nabla{t}\p_su} 
+ \bigl(1+\Abs{\p_tu}\bigr)
\bigl(\Abs{\p_su} + \Norm{\p_su}_{L^1}\bigr)
\Bigr), 
\end{align*}
and
$$
\Abs{\Nabla{s}\Nabla{s}\grad\Vv(u)
- \Hh_\Vv(u)\Nabla{s}\p_su}
\le C\bigl(\Abs{\p_su} + \Norm{\p_su}_{L^2}\bigr)^2
$$
for every smooth map $\R\to\Ll M:s\mapsto u(s,\cdot)$ 
and every $(s,t)\in\R\times S^1$.
\item[(V3)]
For any two integers $k>0$ and $\ell\ge 0$
there is a constant $C=C(k,\ell,\Vv)$ such that
$$
\Abs{\nabla_t^\ell\nabla_s^k\grad\Vv(u)}
\le C\sum_{k_j,\ell_j}
\left(
\prod_{\overset{j}{\scriptscriptstyle\ell_j>0}}
\Abs{\nabla_t^{\ell_j}\nabla_s^{k_j}u}
\right)
\prod_{\overset{j}{\scriptscriptstyle\ell_j=0}}
\Biggl(\Abs{\nabla_s^{k_j}u}
+\left\|\nabla_s^{k_j}u\right\|_{L^{p_j}}
\Biggr)
$$
for every smooth map $\R\to\Ll M:s\mapsto u(s,\cdot)$ 
and every $(s,t)\in\R\times S^1$;
here $p_j\ge 1$ and $\sum_{\ell_j=0}1/p_j=1$; 
the sum runs over all partitions $k_1+\cdots+k_m=k$
and $\ell_1+\cdots+\ell_m\le\ell$ such that $k_j+\ell_j\ge1$ 
for all $j$. For $k=0$ the same inequality holds
with an additional summand $C$ on the right.
\end{description}

\begin{remark}
In~(V0)
the $L^\infty$ bound for $\grad\,\Vv$ 
is imposed, since occasionally
we need $L^p$ bounds
for fixed but arbitrary $p$.
Continuity of $\Vv$ with respect to the 
$C^0$ topology
is used to prove~\cite[lem.~10.2]{SaJoa-LOOP} and
proposition~\ref{prop:unifconvergence-compactsets}.
\end{remark}

\begin{remark}\label{rmk:HV}
Each geometric potential $V$ 
provides an abstract perturbation $\Vv$
such that for smooth loops $x$ and smooth vector
fields $\xi$ along $x$ we have
$$
     \Vv(x):=\int_0^1V_t(x(t))\,dt
     ,\qquad
     \grad\Vv(x) = \nabla V_t(x),\qquad
     \Hh_\Vv(x)\xi = \Nabla{\xi}\nabla V_t(x).
$$
\end{remark}

\begin{remark}\label{rmk:pert}
To prove transversality
in section~\ref{sec:transversality} we use
perturbations\footnote{Here and throughout
the difference $x-x_0$ of two loops
denotes the difference in some 
ambient Euclidean space into which 
$M$ is (isometrically) embedded.
Note that cutting off with respect
to the $L^2$ norm -- as opposed
to the $L^\infty$ norm --
prevents us from expressing
the difference in terms of
the exponential map.
}
of the form
$$
\Vv(x):= \rho\left(\left\|x-x_0\right\|_{L^2}^2\right)
\int_0^1V_t(x(t))\,dt,
$$
where $\rho:\R\to[0,1]$ is a smooth cutoff function
and $x_0:S^1\to M$ is a smooth loop.
Any such perturbation
satisfies~{\rm (V0)--(V3)}.
Here compactness of $M$ is crucial,
in particular, finiteness
of the diameter of $M$.
\end{remark}

\subsection{Main results}
\label{subsec:results}

There are two main purposes
of this text.
One is to construct
the Morse chain complex
for the action functional on the loop space.
The other one is to provide proofs of the
results announced and used
in~\cite{SaJoa-LOOP}
to calculate the adiabatic limit of the Floer complex
of the cotangent bundle.
More precisely, in~\cite{SaJoa-LOOP} we proved
in joint work with
D.~Salamon that the connecting
orbits of the heat flow are
the adiabatic limit of Floer connecting orbits
in the cotangent bundle $T^*M$
with respect to the Hamiltonian given by
kinetic plus potential energy.
The key idea is to appropriately rescale the 
Riemannian metric on $M$.
Both purposes are achieved simultaneously by
theorems~\ref{thm:regularity}--\ref{thm:transversality}.

From now on
we replace the potential $V$
by an abstract perturbation
$\Vv$ satisfying~{\rm (V0)--(V3)}.
In this case the action is given by
\begin{equation}\label{eq:action}
     \Ss_\Vv(x) 
     = \frac12\int_0^1\Abs{\dot x(t)}^2\,dt - \Vv(x)
\end{equation}
for smooth loops $x:S^1\to  M$ and
the heat equation has the form
\begin{equation}\label{eq:heat}
   \p_su - \Nabla{t}\p_tu - \grad\Vv(u) = 0
\end{equation}
for smooth maps
$u:\R\times S^1\to M$, $(s,t)\mapsto u(s,t)$.
Here $\grad\Vv(u)$ denotes the value of $\grad\Vv$ on the
loop $u_s:t\mapsto u(s,t)$.  
The relevant set $\Pp(\Vv)$ of critical points of $\Ss_\Vv$
consists of the (smooth) loops $x:S^1\to M$ that satisfy the 
ODE
\begin{equation}\label{eq:crit}
     \Nabla{t}\dot x=-\grad\Vv(x).
\end{equation}
The subset $\Pp^a(\Vv)$ 
consists of all critical points $x$ with $\Ss_\Vv(x)\le a$. 
For two nondegenerate critical points 
$x^\pm\in\Pp(\Vv)$ we 
denote by $\Mm(x^-,x^+;\Vv)$ the set of 
all solutions $u$ of~(\ref{eq:heat}) 
such that
\begin{equation}\label{eq:limit}
     \lim_{s\to\pm\infty} u(s,t)
     =x^\pm(t),\qquad
     \lim_{s\to\pm\infty} \p_su(s,t)
     =0.
\end{equation}
The limits are uniform in $t$
together with the first partial $t$-derivative.
These solutions are
called \emph{connecting orbits}.
The energy of such a solution is given by 
\begin{equation}\label{eq:par-energy}
     E(u) 
     = \int_{-\infty}^\infty\int_0^1\Abs{\p_su}^2\,dtds
     = \Ss_\Vv(x^-)-\Ss_\Vv(x^+).
\end{equation}

\begin{remark}[Asymptotic limits]\label{rem:limit}
In~(\ref{eq:limit-V}) and~(\ref{eq:limit})
we require convergence in $C^1(S^1)$
as opposed to $C^0(S^1)$ which is standard
in elliptic Floer theory.
We need the stronger assumption
in theorem~\ref{thm:kerD-exp-decay}
to establish exponential decay.
Actually $W^{1,2}(S^1)$ convergence already works.
Compare~\cite{SaJoa-LOOP}
where the asymptotic $C^0$ limits of $(u,v)$ and
$(\p_su,\Nabla{s}v)$ are required to be
$(x^\pm,\p_tx^\pm)$ and zero, respectively.
Now $v$ corresponds to $\p_tu$ 
in the adiabatic limit
studied in~\cite{SaJoa-LOOP}.
\end{remark}

\begin{theorem}[Regularity]
\label{thm:regularity}
Fix a constant $p>2$ and a perturbation 
$\Vv:\Ll M\to\R$ that satisfies~{\rm (V0)--(V3)}.
Let $u:\R\times S^1\to M$ 
be a continuous function
of class $\Ww^{1,p}_{loc}$,
that is $u,\p_tu,\Nabla{t}\p_tu,\p_su$
are locally $L^p$ integrable.
Assume further
that $u$ solves the heat 
equation~(\ref{eq:heat})
almost everywhere.
Then $u$ is smooth.
\end{theorem}

\begin{remark}
It seems unlikely that
the assumption $u\in\Ww^{1,p}_{loc}$
can be weakened
to $u\in W^{1,p}_{loc}$, as announced in~\cite{SaJoa-LOOP},
unless we also weaken $p>2$ to $p>3$;
see~\cite[rmk.~5.2]{Joa-PARABOLI}.
However, the stronger assumption $u\in\Ww^{1,p}_{loc}$
is satisfied in our applications of 
theorem~\ref{thm:regularity}.
These are~\cite[proof of lemma~10.2]{SaJoa-LOOP},
the Banach bundle setup introduced in chapter~\ref{sec:IFT},
step~1 of the proof of theorem~\ref{thm:modified-IFT},
and the proof of proposition~\ref{prop:onto-universal}
on surjectivity of the universal section.
\end{remark}

\begin{theorem}[Apriori estimates]
\label{thm:apriori}
Fix a perturbation $\Vv:\Ll M\to\R$ that 
satisfies~{\rm (V0)--(V1)} and a constant $c_0$. 
Then there is a positive constant
$C=C(c_0,\Vv)$ such that the following holds.
If $u:\R\times S^1\to M$ is a smooth solution
of~(\ref{eq:heat}) such that
$\Ss_\Vv(u(s,\cdot))\le c_0$
for every $s\in\R$ then
$$
     \Norm{\p_tu}_\infty
     +\Norm{\Nabla{t}\p_tu}_\infty
     +\Norm{\p_su}_\infty
     +\Norm{\Nabla{t}\p_su}_\infty
     +\Norm{\Nabla{s}\p_su}_\infty
     \le C.
$$
\end{theorem}

\begin{theorem}[Exponential decay]
\label{thm:par-exp-decay}
Fix a perturbation $\Vv:\Ll M\to\R$ that 
satisfies~{\rm (V0)--(V3)} and assume $\Ss_\Vv$ is Morse. 
\begin{enumerate}
\item[\rm\bfseries(F)]
Let $u:[0,\infty)\times S^1\to M$ 
be a smooth solution 
of~(\ref{eq:heat}).  Then there are positive 
constants $\rho$ and $c_0,c_1,c_2,\dots$
such that 
$$
     \Norm{\p_su}_{C^k([T,\infty)\times S^1)}
     \le c_ke^{-\rho T}
$$
for every $T\ge1$. Moreover, there is 
a periodic orbit $x\in\Pp(\Vv)$ such that
$u(s,\cdot)$ converges to $x$
in $C^2(S^1)$ as $s\to\infty$. 
\item[\rm\bfseries(B)]
Let $u:(-\infty,0]\times S^1\to M$ 
be a smooth solution 
of~(\ref{eq:heat}) with finite energy. 
Then there are positive 
constants $\rho$ and $c_0,c_1,c_2,\dots$
such that 
$$
     \Norm{\p_su}_{C^k((-\infty,-T]\times S^1)}
     \le c_ke^{-\rho T}
$$
for every $T\ge1$. Moreover, there is 
a periodic orbit $x\in\Pp(\Vv)$ such that
$u(s,\cdot)$ converges to $x$
in $C^2(S^1)$ as $s\to-\infty$. 
\end{enumerate}
\end{theorem}

The \emph{covariant Hessian of
$\Ss_\Vv$} at a loop $x:S^1\to M$
is the linear operator 
$A_x:W^{2,2}(S^1,x^*TM)\to L^2(S^1,x^*TM)$
given by
\begin{equation}\label{eq:Hessian}
A_x\xi
:= - \Nabla{t}\Nabla{t}\xi 
- R(\xi,\dot x)\dot x - \Hh_\Vv(x)\xi
\end{equation}
where $R$ denotes the
Riemannian curvature tensor
and the Hessian $\Hh_\Vv$ is 
defined by~(\ref{eq:Hess-Vv}).
This operator is self-adjoint 
with respect to 
the standard $L^2$ inner product.
The number 
of negative eigenvalues is finite.
It is denoted by
$\IND_\Vv(A_x)$
and called the 
Morse index of $A_x$.
If $x$ is a
critical point of $\Ss_\Vv$
we define its \emph{Morse index}
by $\IND_\Vv(x):=\IND_\Vv(A_x)$
and we
call $x$ \emph{nondegenerate}
if $A_x$ is bijective.
In this notation the linearized operator 
$\Dd_u:\Ww_u^{1,p}\to\Ll_u^p$ is given by 
\begin{equation}\label{eq:Du}
     \Dd_u\xi := \Nabla{s}\xi + A_{u_s}\xi
\end{equation}
where $u_s(t):=u(s,t)$
and the spaces 
$\Ww_u=\Ww_u^{1,p}$ and $\Ll_u=\Ll_u^p$
are defined as the completions 
of the space of smooth 
compactly supported sections of the pullback
tangent bundle $u^*TM\to\R\times S^1$
with respect to the norms
\begin{equation}\label{eq:norms}
\begin{gathered}
     \left\|\xi\right\|_{\Ll}
     = \left(\int_{-\infty}^\infty\int_0^1 
     |\xi|^p\,dtds\right)^{1/p},
     \\
     \left\|\xi\right\|_{\Ww}
     = \left(\int_{-\infty}^\infty\int_0^1 
      |\xi|^p + |\Nabla{s}\xi|^p 
      + |\Nabla{t}\Nabla{t}\xi|^p
      \,dtds\right)^{1/p}.
\end{gathered}
\end{equation}

\begin{theorem}[Fredholm]
\label{thm:fredholm}
Fix a perturbation $\Vv:\Ll M\to\R$ that 
satisfies~{\rm (V0)--(V3)}, a constant $p>1$,
and two nondegenerate
critical points $x^\pm\in\Pp(\Vv)$.
Assume $u:\R\times S^1\to M$
is a smooth map such that
$\norm{\Nabla{t}\Nabla{t}\p_su_s}_2$ 
is bounded, uniformly in $s\in\R$,
and
$$
     u_s=\exp_{x^\pm}(\eta_s^\pm),
     \quad
     \Norm{\eta_s^\pm}_{W^{2,2}}\to 0,
     \quad
     \Norm{\p_su_s}_{W^{1,2}}\to 0,
     \quad
     \text{as $s\to\pm\infty$}.
$$
Then the operator 
$\Dd_u:\Ww_u^{1,p}\to\Ll_u^p$
is Fredholm and
$$
     \INDEX\,\Dd_u
     =\IND_\Vv(x^-)-\IND_\Vv(x^+).
$$
Moreover, the formal adjoint operator 
$\Dd_u^*=-\Nabla{s}+A_{u_s}:\Ww_u^{1,p}\to\Ll_u^p$
is Fredholm with
$
     \INDEX\,\Dd_u^*
     =-\INDEX\,\Dd_u
$.
\end{theorem}

Concerning the funny assumption
on $\Nabla{t}\Nabla{t}\p_su_s$
see the footnote in 
section~\ref{sec:linearized}.

\begin{theorem}[Implicit function theorem]
\label{thm:IFT}
Fix a perturbation $\Vv:\Ll M\to\R$ 
that satisfies~{\rm (V0)--(V3)}.
Assume $x^\pm$
are nondegenerate
critical points of $\Ss_\Vv$
and $\Dd_u$ is onto
for every $u\in\Mm(x^-,x^+;\Vv)$.
Then $\Mm(x^-,x^+;\Vv)$
is a smooth manifold
of dimension
$\IND_\Vv(x^-)-\IND_\Vv(x^+)$.
\end{theorem}

\begin{proposition}
[Finite set]
\label{prop:finite-set}
Fix a perturbation
$\Vv:\Ll M\to\R$ 
that satisfies {\rm (V0)--(V3)}
and assume $\Ss_\Vv$ is \emph{Morse--Smale
below level $a$}
in the sense that every $u\in\Mm(x^-,x^+;\Vv)$
is \emph{regular} (i.e. the Fredholm operator 
$\Dd_u$ is surjective),
for every pair $x^\pm\in\Pp^a(\Vv)$.
Then the quotient space
$$
     \widehat\Mm(x^-,x^+;\Vv)
     :=\Mm(x^-,x^+;\Vv)/\R
$$
is a finite set 
for every such pair
of Morse index difference one.
Here the (free) action of $\R$ is given by
time shift
$
     (\sigma,u)\mapsto 
     u(\sigma+\cdot,\cdot)
$.
\end{proposition}

\begin{theorem}
[Refined implicit function theorem]
\label{thm:modified-IFT}
Fix a perturbation $\Vv:\Ll M\to\R$ 
that satisfies~{\rm (V0)--(V3)}
and a pair of nondegenerate critical points 
$x^\pm\in\Pp(\Vv)$ with 
$\Ss_\Vv(x^+)<\Ss_\Vv(x^-)$
and Morse index difference one.
Then, for every $p>2$
and every large constant $c_0>1$,
there are positive constants 
$\delta_0$ and $c$
such that the following holds.
Assume $\Ss_\Vv$ is Morse--Smale
below level $2c_0^2$.
Assume further that $u:\R\times S^1\to M$ 
is a smooth map such that 
$u(s,\cdot)$ converges in $W^{1,2}(S^1)$
to $x^\pm$, as $s\to\pm\infty$, and such that
$$
      \Abs{\p_su(s,t)}
      \le\frac{c_0}{1+s^2},\qquad
      \Abs{\p_tu(s,t)}
      \le c_0,\qquad
      \Abs{\Nabla{t}\p_tu(s,t)}
      \le c_0,\qquad
$$
for all $(s,t)\in\R\times S^1$ and 
$$
     \Norm{\p_su-\Nabla{t}\p_tu
     -\grad\Vv(u)}_p\le\delta_0.
$$
Then there exist elements 
$u_*\in\Mm(x^-,x^+;\Vv)$
and $\xi\in\im\Dd_{u_*}^*\cap\Ww$
satisfying
$$
     u=\exp_{u_*}(\xi),\qquad
     \Norm{\xi}_\Ww
     \le c
     \Norm{\p_su-\Nabla{t}\p_tu-\grad\Vv(u)}_p.
$$
\end{theorem}

In the previous theorem 
``$c_0$ large''
means that the constant $c_0$
should be larger than
the constant $C_0$ in axiom~(V0).
Recall that a subset of a complete
metric space is called {\bf residual}
if it contains a countable intersection
of open and dense sets.
By Baire's category theorem
a residual subset is dense.

\begin{theorem}[Transversality]
\label{thm:transversality}
Fix a perturbation $\Vv:\Ll M\to\R$ 
that satisfies~{\rm (V0)--(V3)} and
assume $\Ss_{\Vv}$ is Morse.
Then for every regular value $a$
there is a complete metric space
$\Oo^a=\Oo^a(\Vv)$ of perturbations
supported away from $\Pp^a(\Vv)$ and
satisfying~{\rm (V0)--(V3)} such that
the following is true.
If $v\in\Oo^a$, then 
$$
     \Pp^a(\Vv)=\Pp^a(\Vv+v),
     \qquad
     {\rm H}_*\left(\left\{\Ss_\Vv\le a\right\}\right)
     \cong
     {\rm H}_*\left(\left\{
     \Ss_{\Vv+v}\le a\right\}\right).
$$
Moreover, there is a residual subset
$\Oo^a_{reg}\subset\Oo^a$
such that for each $v\in\Oo^a_{reg}$
the perturbed functional
$\Ss_{\Vv+v}$ is Morse--Smale
below level $a$.
\end{theorem}

\subsection*{Outlook}

The obvious next step is to relate
heat flow homology defined in
theorem~\ref{thm:d^2=0}
to singular homology of the loop space.
In our forthcoming paper~\cite{Joa-LOOPSPACE}
we establish the following result.
Throughout singular homology ${\rm H}_*$ is
meant with integer coefficients.

\begin{theorem}\label{thm:isomorphism}
Assume $\Ss_V$ is Morse and $a$ is either
a regular value of $\Ss_V$ or equal to infinity.
Then there is a natural isomorphism
$$
     \HM_*^a(\Ll M,\Ss_V)
     \cong
     {\rm H}_*(\Ll^a M),
     \qquad
     \Ll^a M:=\{\gamma\in\Ll M\mid \Ss_V(\gamma)\le a\}.
$$
If $M$ is not simply connected,
then there is a separate isomorphism for
each component of the loop space.
The isomorphism commutes with the
homomorphisms
$
     \HM_*^a(\Ll M,\Ss_V)
     \to
     \HM_*^b(\Ll M,\Ss_V)
$
and
$
     {\rm H}_*(\Ll^a M)
     \to
     {\rm H}_*(\Ll^b M)
$
for $a<b$.
\end{theorem}

For a $C^1$ gradient flow on a Banach manifold,
where the Morse functional is bounded below
and its critical points are of finite Morse index,
Abbondandolo and 
Majer~\cite{AM-LecsMCInfDimMfs}
proved the existence of a natural
isomorphism between singular and
Morse homology.
The geometric idea is that
the unstable manifolds
carry the homologically relevant information.
A major point is to construct a cellular filtration 
of $\Ll^a M$
by open forward flow invariant subsets
$F_0\subset F_1\subset\ldots\subset F_N\subset\Ll^a M$
such that $F_k$ contains all critical points
up to Morse index $k$ and relative singular
homology $\mathrm{H}_*(F_k,F_{k-1})$ is 
isomorphic to the free abelian group
generated over $\Z$ by
the critical points of index $k$.
Let $F_0$ be the union of disjoint, open, and forward flow
invariant neighborhoods of the critical points of index zero.
Then they fix small neighborhoods of the index one critical points
and consider the set exhausted by the forward flow.
Now they take the union of this set with $F_0$ to obtain $F_1$.
Clearly $F_1$ is forward flow invariant.
Moreover, it is open, because the time-$t$-map of the flow
is an open map. Next continue with the index two critical points
and so on. 

Unfortunately
the time-$t$-map for the semiflow
generated by the heat equation does
not take open sets to open sets
due to the extremely strong regularizing
nature of the heat flow.
Hence new ideas are required.
Firstly, find the right notion of
Conley index pairs of isolated invariant
sets in the infinite dimensional situation at hand.
Secondly, solve
the forward time Cauchy problem for 
the heat equation~(\ref{eq:heat})
for initial values
in the Hilbert manifold $\Lambda M=W^{1,2}(S^1,M)$
to establish existence of a continuous semiflow
$
     \varphi:[0,\infty)\times\Lambda^a M \to\Lambda^a M 
$.
Now use continuity of the time-$t$-map
to conclude that the preimage
${\varphi_T}^{-1}(F_0)$ is an open subset of $\Lambda^a M$.
Here $F_0$ consists locally of (strict) sublevel sets
near the local minimima
and $T>0$ is chosen sufficiently large
such that the time-$T$-map $\varphi_T$ maps
the exit set $L_1$ of the Conley index pair $(N_1,L_1)$
associated to the index one critical points
into $F_0$. Then the set
$F_1:=N_1\cup {\varphi_T}^{-1}(F_0)$
is open and semiflow invariant.
Next include the index two points and so on.
Full details will be provided
in~\cite{Joa-LOOPSPACE}.

\subsection{Overview}

In appendix~\ref{sec:REG}
we recall for convenience of the reader
facts proved in~\cite{Joa-PARABOLI}
concerning parabolic regularity.
These are used extensively in the present text.
More precisely, we recall
the fundamental $L^p$ estimate
and local regularity
for the linear heat operator
$\p_su-\p_t\p_tu$ 
acting on real-valued maps $u$ defined
on the lower half plane $\HH^-$
or on cylindrical sets.
Moreover, we introduce relevant parabolic spaces
$\Ww^{k,p}$ and $\Cc^{k,p}$ and recall the
product estimate lemma~\ref{le:product-derivatives}
crucial to prove the quadratic
estimate in proposition~\ref{prop:Quadest0}.

In chapter~\ref{sec:kerD}
we study the solutions to the linearized
version of the heat equation~(\ref{eq:heat}),
in other words, the kernel of 
the operator $\Dd_u$ given by~(\ref{eq:Du}).
In theorem~\ref{thm:REG-L}
we show that these solutions are smooth.
In fact, even weak solutions are smooth.
In section~\ref{subsec:kerD-apriori}
we derive pointwise bounds
in terms of the $L^2$ norm.
Section~\ref{subsec:kerD-exp-decay}
then establishes exponential decay
of these $L^2$ norms.
The combination of these results
is used in section~\ref{sec:linearized}
to prove that the operator
$\Dd_u$ is Fredholm for a rather general
class of smooth cylinders $u$ in $M$
with nondegenerate asymptotic
limits $x^\pm\in\Pp(\Vv)$.
The main result is theorem~\ref{thm:fredholm}.

In chapter~\ref{sec:nonlinear-heat}
we study the solutions $u$ to
the (nonlinear) heat equation~(\ref{eq:heat}).
Since $\p_su$ solves the linearized
equation the results of 
chapter~\ref{sec:kerD} apply.
In section~\ref{subsec:regularity-compactness}
we prove smoothness of
$\Ww^{1,p}_{loc}$ solutions
and a compactness result for
sequences of uniformly bounded gradient
with respect to appropriate norms.
In sections~\ref{subsec:apriori}--\ref{subsec:exp-decay}
boundedness of the action is a crucial assumption.
Fix a positive constant $c_0$.
Then all solutions $u$ of~(\ref{eq:heat})
with
$$
     \sup_{s\in\R} \Ss_\Vv(u_s) 
     \le c_0
$$
admit a uniform apriori estimate
for $\norm{\p_tu}_\infty$
(theorem~\ref{thm:apriori-t}), 
uniform energy bounds 
(lemma~\ref{le:energy-bound}),
uniform gradient bounds 
(theorem~\ref{thm:gradient}), 
and uniform $L^2$ exponential decay 
(theorem~\ref{thm:exp-decay}).
In section~\ref{subsec:compactness}
we study compactness of the
moduli spaces $\Mm(x^-,x^+;\Vv)$
in the case that $\Ss_\Vv:\Ll M\to\R$ 
is a Morse function.

Chapter~\ref{sec:IFT}
deals with implicit function
type theorems. Here, in addition to
the Morse condition, the Morse--Smale
condition enters: To prove that
the moduli spaces are smooth manifolds
we not only need nondegeneracy
of the asymptotic boundary data,
that is the critical points $x^\pm$,
but in addition surjectivity 
of the linearized operators.
Under these assumptions
proposition~\ref{prop:finite-set} asserts
that modulo time shift
there are only finitely many 
heat flow lines from $x^-$ to $x^+$
whenever the Morse index difference is one.
Here the compactness results of
section~\ref{subsec:compactness} enter.
Furthermore, we prove the
refined implicit function 
theorem~\ref{thm:modified-IFT},
a major technical tool in~\cite{SaJoa-LOOP}.
Here the product estimate provided by
lemma~\ref{le:product-derivatives}
is the crucial ingredient to obtain the required
quadratic estimates.
Furthermore, the choice of the sublevel set on 
which $\Ss_\Vv$ needs to be Morse--Smale
requires care. The reason is that one
starts out only with an \emph{approximate} 
solution $u$
along which the action is not necessarily
decreasing. However, the assumptions
guarantee that all loops
$u_s$ are contained in the sublevel set
$\{\Ss_\Vv\le 2c_0^2\}$.

In chapter~\ref{sec:UC}
we prove unique continuation
for the heat equation~(\ref{eq:heat})
and its linearization.
The proof is based on 
an extension of 
a result by Agmon and Nirenberg.
In contrast to forward unique continuation
the result on backward unique continuation
is surprising at first sight.
Of course, there is an assumption.
Namely, the action along the 
two semi-infinite backward
trajectories $u,v$ which 
coincide at time $s=0$
must be bounded. In this case
we obtain that $u=v$.

In chapter~\ref{sec:transversality}
we construct a separable Banach space $Y$
of abstract perturbations that
satisfy axioms~{\rm (V0)--(V3)}.
Assume $\Ss_\Vv$ is Morse and 
$a$ is a regular value.
Then we define a Banach submanifold
$\Oo^a(\Vv)$ of admissible perturbations $v$.
These have the property
that $\Ss_\Vv$ and $\Ss_{\Vv+v}$
do have the same critical points
on their respective sublevel set
with respect to $a$
and, moreover, both sublevel sets
are homologically equivalent.
The proof that there is a residual
subset $\Oo^a_{reg}(\Vv)$ 
of regular perturbations
for which $\Ss_{\Vv+v}$ is Morse-Smale below level $a$
requires unique continuation
for the linearized heat equation
and the fact that the action is strictly
decreasing along nonconstant
heat flow trajectories.

In chapter~\ref{sec:Morse-homology}
we define Morse homology for
the heat flow.
In section~\ref{subsec:unstable}
we define the unstable manifold
of a critical point $x$
of the action functional
$\Ss_\Vv:\Ll M\to\R$
as the set of endpoints at time zero
of all backward halfcylinders
solving the heat equation~(\ref{eq:heat})
and emanating from $x$ at $-\infty$.
The main result is 
theorem~\ref{thm:unstable-mf} saying that
if the critical point $x$ is nondegenerate,
then this is a contractible submanifold
of the loop space and its dimension
equals the Morse index of $x$.
Here we use unique continuation
for the linear and the nonlinear heat equation.
In section~\ref{subsec:Morse-complex}
we put together everything
to define the Morse complex
for the negative $L^2$ gradient
of the action functional on the loop space.

Despite the title of this text
the fact that the heat equation
gives rise to a forward semiflow
is nowhere used.
Existence of
this semiflow will be proved 
and used in our forthcoming
paper~\cite{Joa-LOOPSPACE} to construct a natural
isomorphism to singular homology
of the loop space.

\begin{notation}
If $f=f(s,t)$ is a map,
then $f_s$ abbreviates 
the map $f(s,\cdot):t\mapsto f(s,t)$.
In contrast partial derivatives
are denoted by $\p_sf$ and $\p_tf$.
\end{notation}

\noindent
{\sl Acknowledgements.} 
{\small
For useful discussions 
and pleasant conversations
the author would like to thank
K.~Cieliebak,
K.~Mohnke, and
D.~Salamon.
For partial financial support and hospitality
we are grateful to 
MSRI~Berkeley
and 
SFB~647 at HU~Berlin.
}
%
%

\section{The linearized heat equation}
\label{sec:kerD}

Fix a smooth function
$\Vv:\Ll M\to\R$ that
satisfies~{\rm (V0)--(V3)}
and a smooth map $u:\R\times S^1\to M$.
In this chapter
we study the linear
parabolic PDE
\begin{equation}\label{eq:kerD-heat}
     \Nabla{s}\xi
     -\Nabla{t}\Nabla{t}\xi
     -R(\xi,\p_tu)\p_tu
     -\Hh_\Vv(u)\xi
     =0
\end{equation}
for vector
fields $\xi$ along $u$.
Throughout $R$ denotes the Riemannian
curvature tensor associated to
the closed Riemannian manifold $M$
and the covariant Hessian 
$\Hh_\Vv$ of $\Vv$
at a loop $u(s,\cdot)$ is defined
by~(\ref{eq:Hess-Vv}).

In section~\ref{subsec:kerD-regularity}
we show that \emph{strong solutions},
that is solutions of class $\Ww^{1,p}_u$,
are automatically smooth.
More generally, for $\xi\in \Ll^p_u$
we define the notion of
\emph{weak solution} 
and show that even weak solutions are smooth.
In section~\ref{subsec:kerD-apriori}
we derive pointwise estimates
of $\xi$ and certain partial derivatives
in terms of the $L^2$ norm
of $\xi$ over small backward cylinders.
In section~\ref{subsec:kerD-exp-decay}
we establish asymptotic exponential decay
of the slicewise $L^2$ norm
$\norm{\xi_s}_{L^2(S^1)}$
of a solution $\xi$
whenever the covariant Hessian
$A_{u_s}$ given by~(\ref{eq:Hessian})
is asymptotically injective.
Still assuming asymptotic injectivity
we prove in section~\ref{sec:linearized}
that the linear operator
$$
     \Dd_u:\Ww^{1,p}_u\to\Ll^p_u
$$
given by the left hand side 
of~(\ref{eq:kerD-heat})
is Fredholm.

Observe that if $u$ solves the
(nonlinear) heat equation~(\ref{eq:heat})
then $\xi:=\p_su$ solves
the linear equation~(\ref{eq:kerD-heat}).
Hence the results of this chapter
will be useful in chapter~\ref{sec:nonlinear-heat}
on solutions of the nonlinear heat equation.

\subsection{Regularity}
\label{subsec:kerD-regularity}

Define the operator
$\Dd_u^*$ by the left hand side 
of~(\ref{eq:kerD-heat})
with $\Nabla{s}$ replaced by $-\Nabla{s}$.

\begin{theorem}[Local regularity of weak solutions]
\label{thm:REG-L}
Fix a perturbation 
$\Vv:\Ll M\to\R$ that
satisfies~{\rm (V0)--(V3)}
and constants $q>1$
and $a<b$. Let
$u:(a,b]\times S^1\to M$
be a smooth map with bounded derivatives
of all orders.
Then the following is true.
If $\eta$ is a vector field
along $u$ of class $L^q_{loc}$
such that
$$
     \langle\eta,\Dd_u^*\xi\rangle=0
$$
for every smooth
vector field $\xi$ along $u$
of compact support in $(a,b)\times S^1$,
then $\eta$ is smooth.
Here $\langle\cdot,\cdot\rangle$ 
denotes integration over 
the pointwise inner products.
\end{theorem}

\begin{remark}
Theorem~\ref{thm:REG-L} remains true
if we replace $\Dd_u^*$ by $\Dd_u$ and 
define $u$ on $[a,b)\times S^1$.
This follows by the variable substitution
$s\mapsto-s$.
\end{remark}

\begin{proof}
It suffices to prove
the conclusion
in a neighborhood of
any point $z\in(a,b]\times S^1$.
Shifting the $s$ and $t$
variables, if necessary,
we may assume that
$z\in\Omega_r=(-r^2,0]\times(-r,r)$
for some sufficiently small $r>0$.
Now choose local coordinates
on the manifold $M$ around
the point $u(z)$ and fix $r>0$
sufficiently small such that
$u(\overline{\Omega_r})$
is contained in the
local coordinate patch.
In these local coordinates
the vector field $\eta$ is
represented by the map
$(\eta^1,\ldots,\eta^n):\Omega_r\to\R^n$ 
of class $L^q_{loc}$
and the Riemannian metric $g$
by the matrix with
components $g_{ij}$.
Throughout we use Einstein's sum convention.
By induction we will prove that
$$
     v_\mu:=g_{\mu j}\eta^j
     \in\bigcap_{m=1}^\infty
     \Ww^{m,q}_{loc}(\Omega_r),\qquad
     \mu=1,\ldots,n.
$$
Note that the intersection of spaces
equals $C^\infty(\Omega_r)$;
see e.g.~\cite[app.~B.1]{MS}.
Now apply the inverse metric matrix
to obtain that
$\eta^j=g^{j\mu}v_\mu\in C^\infty(\Omega_r)$
and this proves the theorem.

\vspace{.1cm}
\noindent
{\bf\boldmath Step~$m=1$.}
Fix $\mu\in\{1,\ldots,n\}$
and consider vector fields
of the form
$$
     \xi^{(\mu,\phi)}=(0,\ldots,0,\phi,0,\ldots,0)
     :\Omega_r\to\R^n
$$
where a function
$\phi\in C_0^\infty(\INT\, \Omega_r)$
occupies slot $\mu$.
Via extension by zero
we view $\xi^{(\mu,\phi)}$
as a compactly supported smooth
vector field along $u$.
Now our assumption implies that
$\langle\eta,\Dd_u^*\xi^{(\mu,\phi)}\rangle=0$
for every $\phi\in C_0^\infty(\INT\, \Omega_r)$.
By straightforward calculation
this is equivalent to
$$
     \int_{\Omega_r}
     v_\mu\left(-\p_s\phi-\p_t\p_t\phi\right)
     =\int_{\Omega_r} f_\mu\phi
     -\int_{\Omega_r} h_\mu\,\p_t\phi
$$
for every $\phi\in C_0^\infty(\INT\, \Omega_r)$,
where $h_\mu=-2v_k\Gamma_{i\mu}^k \,\p_tu^i$ and
\begin{equation*}
\begin{split}
     f_\mu
    &=v_k\Bigl(
     \Gamma_{i\mu}^k \,\p_su^i
     +\frac{\p\Gamma_{i\mu}^k}{\p u^r}
     \,\p_tu^r\,\p_tu^i
     +\Gamma_{i\mu}^k \,\p_t\p_tu^i\\
    &\quad
     +\Gamma_{ij}^k \,\p_tu^i
     \Gamma_{r\mu}^j \,\p_tu^r
     +R_{\mu ij}^k \,\p_tu^i \,\p_tu^j
     +H_\mu^k \Bigr).
\end{split}
\end{equation*}
Here $R_{\ell ij}^k$ represents the Riemann
curvature operator and $H_\ell^k$
the Hessian $\Hh_\Vv(u)$ in local coordinates.
The Christoffel symbols
associated to the 
Levi Civita connection $\nabla$
are denoted by $\Gamma_{ij}^k$.

From now on the domain of
all spaces will be $\Omega_r$,
unless specified differently.
Observe that
$v_\mu\in L^q_{loc}
\subset L^1_{loc}$
by smoothness of the metric,
compactness of $M$,
and the fact that
$\eta^\ell\in L^q_{loc}$
by assumption. It follows that
$h_\mu$ and $f_\mu$ are in $L^q_{loc}$.
Here we used in addition boundedness
of the derivatives of $u$
and axiom~(V1).
Hence $\p_tv_\mu\in L^q_{loc}$
by theorem~\ref{thm:local-regularity}~b)
and this implies that
$\p_th_\mu\in L^q_{loc}$.
Now integration by parts
shows that
$$
     \int_{\Omega_r}
     v_\mu\left(-\p_s\phi-\p_t\p_t\phi\right)
     =\int_{\Omega_r} 
     \left(f_\mu+\p_th_\mu\right)\phi
$$
for every $\phi\in C_0^\infty(\INT\, \Omega_r)$
and therefore
$v_\mu\in\Ww^{1,q}_{loc}$ by
theorem~\ref{thm:local-regularity}~a).

\vspace{.1cm}
\noindent
{\bf\boldmath Induction step~$m\Rightarrow m+1$.}
Assume that
$v_\mu\in\Ww^{m,q}_{loc}$. Then
$f_\mu,h_\mu\in\Ww^{m,q}_{loc}$
by compactness of $M$,
boundedness of the derivatives of $u$,
and axiom~(V3). Hence
$\p_tv_\mu\in\Ww^{m,q}_{loc}$ by
theorem~\ref{thm:local-regularity}~b).
But this implies that
$\p_th_\mu$ is in $\Ww^{m,q}_{loc}$
and so is $f_\mu+\p_th_\mu$.
Therefore
$v_\mu\in\Ww^{m+1,q}_{loc}$ by
theorem~\ref{thm:local-regularity}~a).
\end{proof}

\subsection{Apriori estimates}
\label{subsec:kerD-apriori}

\begin{theorem}
\label{thm:kerD-apriori}
Fix a perturbation 
$\Vv:\Ll M\to\R$ that 
satisfies~{\rm (V0)--(V2)} and a 
constant $C_0>0$.
Then there is a constant
$C=C(C_0,\Vv)>0$ such that the following is true.
Assume $u:\R\times S^1\to M$ 
is a smooth map with
$\norm{\p_tu}_\infty\le C_0$
and $\xi$ is a smooth vector field
along $u$ satisfying the
linear heat
equation~(\ref{eq:kerD-heat}).
Then
$$
     \Abs{\xi(s,t)}
     \le C \Norm{\xi}_{L^2([s-\frac{1}{2},s]\times S^1)}
$$
for every $(s,t)\in\R\times S^1$.
If in addition
$\norm{\p_su}_\infty +
\norm{\Nabla{t}\p_tu}_\infty\le C_0$, 
then
$$
     \Abs{\Nabla{t}\xi(s,t)}
     \le C \Norm{\xi}_{L^2([s-1,s]\times S^1)}
$$
for every $(s,t)\in\R\times S^1$.
\end{theorem}

\begin{theorem}
\label{thm:kerD-apriori-II}
Fix a perturbation 
$\Vv:\Ll M\to\R$ that 
satisfies~{\rm (V0)--(V2)} and a 
constant $C_0>0$.
Then there is a constant
$C=C(C_0,\Vv)>0$ such that the following is true.
Assume $u:\R\times S^1\to M$ 
is a smooth map with
$$
     \norm{\p_tu}_\infty
     +\norm{\p_su}_\infty
     +\norm{\Nabla{t}\p_tu}_\infty
     +\norm{\Nabla{t}\p_su}_\infty
     +\norm{\Nabla{t}\Nabla{t}\p_tu}_\infty
     \le C_0
$$
and $\xi$ is a smooth vector field
along $u$ satisfying the
linear heat
equation~(\ref{eq:kerD-heat}).
Then
$$
     \Abs{\Nabla{t}\Nabla{t}\xi(s,t)}
     +\Abs{\Nabla{s}\xi(s,t)}
     \le C \Norm{\xi}_{L^2([s-2,s]\times S^1)}
$$
for every $(s,t)\in\R\times S^1$.
\end{theorem}

\begin{remark}
\label{rmk:apriori-adjoint}
If in theorem~\ref{thm:kerD-apriori}
or theorem~\ref{thm:kerD-apriori-II}
the vector field
$\xi$ solves $\Dd_u^*\xi=0$,
then $\eta(s,t):=\xi(-s,t)$
solves~(\ref{eq:kerD-heat}).
The apriori estimates for
$\eta$ then translate
into apriori estimates for $\xi$.
For example, it follows that
$$
     \Abs{\xi(s,t)}
     \le C \Norm{\xi}_{L^2([s,s+\frac{1}{2}]\times S^1)}
$$
for every $(s,t)\in\R\times S^1$
and similarly for the
higher order derivatives.
\end{remark}

The proof of theorem~\ref{thm:kerD-apriori}
and theorem~\ref{thm:kerD-apriori-II}
is based on the following mean value
inequalities.
Consider the
{\bf parabolic domain}
defined for $r>0$ by 
$$
     P_r:=(-r^2,0)\times (-r,r).
$$

\begin{lemma}[{\cite[lemma~B.1]{SaJoa-LOOP}}]
\label{le:kerD-apriori-basic}
There is a constant $c_1>0$
such that the following holds 
for all $r\in(0,1]$ and $a\ge 0$. If
$w:P_r\to\R$, 
$(s,t)\mapsto w(s,t)$,
is $C^1$ in the $s$-variable 
and $C^2$ in the $t$-variable
such that
$$
     (\p_t\p_t-\p_s)w\ge -aw, \qquad w\ge0,
$$
then
$$
     w(0)\le\frac{c_1e^{ar^2}}{r^3}
     \int_{P_r} w.
$$
\end{lemma}

\begin{corollary}
\label{co:kerD-apriori-basic}
Let $c_1$ be the constant
of lemma~\ref{le:kerD-apriori-basic}
and fix two constants
$r\in(0,1]$ and $\mu\ge 0$.
Then the following is true.
If $F:[-r^2,0]\to\R$
is a $C^1$ function such that
$$
     -F^\prime+\mu F\ge 0, \qquad F\ge0,
$$
then
$$
     F(0)\le\frac{2c_1e^{\mu r^2}}{r^2}
     \int_{-r^2}^0 F(s)\: ds.
$$
\end{corollary}

\begin{proof}
Lemma~\ref{le:kerD-apriori-basic}
with $w(s,t):=F(s)$.
\end{proof}

\begin{lemma}[{\cite[lemma~B.4]{SaJoa-LOOP}}]
\label{le:kerD-apriori-L2}
Let $R,r>0$ and $u:P_{R+r}\to\R$, 
$(s,t)\mapsto u(s,t)$, be $C^1$
in the $s$-variable 
and $C^2$ in the $t$-variable 
and $f,g:P_{R+r}\to\R$ 
be continuous functions such that
$$
     \left(\p_t\p_t-\p_s\right)u\ge g-f, \qquad
     u\ge 0, \qquad f\ge 0, \qquad g\ge 0.
$$
Then
$$
     \int_{P_R}g
     \le\int_{P_{R+r}}f+\left(\frac{4}{r^2}
     +\frac{1}{Rr}\right)
     \int_{P_{R+r}\setminus P_R}u.
$$
\end{lemma}

\begin{corollary}\label{co:kerD-apriori-L2}
Fix two positive constants $r,R$
and three functions
$U,F,G:[-(R+r)^2,0]\to\R$ 
such that $U$ is $C^1$ and 
$F,G$ are continuous. If
$$
     -U'\ge G-F,
     \qquad U\ge 0,\qquad
     F\ge 0,\qquad G\ge 0,
$$
then
$$
     \int_{-R^2}^0G(s)\,ds 
     \le\frac{R+r}{R}\left(
     \int_{-(R+r)^2}^0F(s)\,ds
     +
     \left(\frac{4}{r^2}+\frac{1}{Rr}\right)
     \int_{-(R+r)^2}^0U(s)\,ds\right).
$$
\end{corollary}

\begin{proof}
Lemma~\ref{le:kerD-apriori-L2}
with $u(s,t)=U(s)$, $f(s,t)=F(s)$, 
and $g(s,t)=G(s)$.
\end{proof}

\begin{proof}
[Proof of theorem~\ref{thm:kerD-apriori}]
We prove the theorem
in three steps.
The idea is to
prove in step~1
the desired pointwise estimate
in its integrated form
(slicewise estimate).
In steps~2 and 3
this is then used to prove
the pointwise estimates.
Note that in step~3
we provide an estimate
which is not used in the
current proof, but 
later on in the proof of 
theorem~\ref{thm:kerD-apriori-II}.
Occasionaly
we denote $\xi(s,t)$ by $\xi_s(t)$
and in this case
$\norm{\xi_s}$ abbreviates
$\norm{\xi_s}_{L^2(S^1)}$.

\vspace{.1cm}
\noindent
{\bf Step 1.}
{\it There is a constant $C_1=C_1(C_0,\Vv)>0$
such that
$$
     \int_0^1\Abs{\xi(s,t)}^2\: dt
     +\int_{s-\frac{1}{16}}^s\int_0^1
     \Abs{\Nabla{t}\xi(s,t)}^2\: dtds
     \le C_1 \Norm{\xi}_{L^2([s-\frac{1}{4},s]\times S^1)}^2
$$
for every $s\in\R$.
}

\vspace{.1cm}
\noindent
Define the functions
$f,g:\R\times S^1\to\R$
and $F,G:\R\to\R$ by
$$
     2f:=\abs{\xi}^2,\quad
     2g:=\abs{\Nabla{t}\xi}^2,\quad
     F(s):=\int_0^1f(s,t)\: dt,\quad
     G(s):=\int_0^1g(s,t)\: dt,
$$
and abbreviate
$$
     L:=\p_t\p_t-\p_s,\qquad
     \Ll:=\Nabla{t}\Nabla{t}-\Nabla{s}.
$$
Then
\begin{equation}\label{eq:ker-Lf}
     Lf=2g+U,\qquad
     U:=\langle\xi,\Ll\xi\rangle.
\end{equation}
Assume that $U$
satisfies the pointwise inequality
\begin{equation}\label{eq:ker-U}
    \abs{U}\le
    \mu f+\frac{1}{2}
    \norm{\xi_s}^2
\end{equation}
for a suitable constant $\mu=\mu(C_0,\Vv)>0$.
Hence
$
     Lf+\mu f+F\ge 2g
$
by~(\ref{eq:ker-Lf})
and integration over the interval
$0\le t\le1$ shows that
$$
     -F^\prime+(\mu+1)F\ge 2G.
$$
Step~1 follows by
Corollary~\ref{co:kerD-apriori-basic}
with $r=\frac{1}{2}$ and
corollary~\ref{co:kerD-apriori-L2}
with $R=r=\frac{1}{4}$.

It remains to prove~(\ref{eq:ker-U}).
Since $\xi$
solves the linear
heat equation~(\ref{eq:kerD-heat}),
it follows that
\begin{equation*}
\begin{split}
      \Abs{U}
     &=\Abs{\langle\xi,\Nabla{t}\Nabla{t}\xi
      -\Nabla{s}\xi\rangle}\\
     &=\Abs{\langle\xi,R(\xi,\p_t u)\p_t u
      +\Hh_\Vv(u)\xi\rangle}\\
     &\le \Norm{R}_\infty
     \Norm{\p_t u}_\infty^2\Abs{\xi}^2
     +c_1 \Abs{\xi}
     \bigl(\Abs{\xi}+
     \Norm{\xi_s}_{L^1(S^1)}\bigr) \\
    &\le \left(2C_0^2\Norm{R}_\infty
     +2c_1+{c_1}^2\right)
     \frac{1}{2}\Abs{\xi}^2
     +\frac{1}{2}\Norm{\xi_s}^2.
\end{split}
\end{equation*}
Here we used
the assumption on $\p_t u$,
axiom~(V1) with constant $c_1$,
and the fact that
$\norm{\cdot}_{L^1(S^1)}
\le\norm{\cdot}_{L^2(S^1)}$
by H\"older's inequality.
This proves~(\ref{eq:ker-U}).

\vspace{.1cm}
\noindent
{\bf Step 2.}
{\it We prove the estimate
for $\abs{\xi}$ in
theorem~\ref{thm:kerD-apriori}.}

\vspace{.1cm}
\noindent
Note that
$Lf\ge-\abs{U}$ by~(\ref{eq:ker-U}).
Hence the estimate~(\ref{eq:ker-U})
for $\abs{U}$
and the slicewise estimate
for $\xi_s$
provided by step~1
prove the pointwise
inequality
$$
     Lf
     \ge-\mu f-2C_1 
     \Norm{\xi}_{L^2([s-\frac{1}{4},s]\times S^1)}^2
$$
for all $s$ and $t$.
Fix $(s_0,t_0)$ and set
$
     a=a(s_0):=\frac{2C_1}{\mu}
     \Norm{\xi}_{L^2([s_0-\frac{1}{2},s_0]\times S^1)}^2
$.
Then
$$
     L\left(f+a\right)
     \ge-\mu\left(f+a\right)
$$
for all $t$ and $s\in[s_0-\frac{1}{4},s_0]$.
Hence
lemma~\ref{le:kerD-apriori-basic} 
with $r=\frac{1}{2}$
applies to the
function
$w(s,t):=f(s_0+s,t_0+t)+a$
and we obtain that
\begin{equation*}
\begin{split}
     f(s_0,t_0)
    &\le 8c_1e^{\mu/4}
     \int_{-\frac{1}{4}}^0\int_0^1
     \left( f(s_0+s,t_0+t)+a \right)\: dtds\\
    &\le 8c_1e^{\mu/4}\left(\frac{1}{2}
     +\frac{C_1}{2\mu}\right) 
     \Norm{\xi}_{L^2([s_0-\frac{1}{2},s_0]\times S^1)}^2.
\end{split}
\end{equation*}
Since $s_0\in\R$ 
and $t_0\in S^1$
were chosen arbitrarily,
this proves step~2.

\vspace{.1cm}
\noindent
{\bf Step 3.}
{\it There is a constant $C_3=C_3(C_0,\Vv)>0$
such that
$$
     \int_{s-\frac{1}{4}}^s\int_0^1
     \Abs{\Nabla{t}\Nabla{t}\xi(s,t)}^2\: dtds
     \le C_3 \Norm{\xi}_{L^2([s-\frac{5}{4},s]\times S^1)}^2
$$
for every $s\in\R$.
Moreover, the estimate
for $\abs{\Nabla{t}\xi}$ in
theorem~\ref{thm:kerD-apriori}
holds true.}

\vspace{.1cm}
\noindent
Define the functions
$f_1,g_1:\R\times S^1\to\R$ by
$$
     2f_1:=\abs{\Nabla{t}\xi}^2,\quad
     2g_1:=\abs{\Nabla{t}\Nabla{t}\xi}^2
$$
and the functions
$F_1,G_1:\R\to\R$ by
$$
     F_1(s):=\int_0^1f_1(s,t)\: dt,\quad
     G_1(s):=\int_0^1g_1(s,t)\: dt.
$$
Then
\begin{equation}\label{eq:kerD-Lf1}
     Lf_1=2g_1+U_t,\qquad
     U_t:=\langle\Nabla{t}\xi,
     \Ll\Nabla{t}\xi\rangle.
\end{equation}
Since $\xi$ solves 
the linear heat 
equation~(\ref{eq:kerD-heat}),
it follows that
\begin{equation*}
\begin{split}
     \Ll \Nabla{t}\xi
    &=\Nabla{t}\left(\Nabla{t}\Nabla{t}\xi
     -\Nabla{s}\xi\right)
     -[\Nabla{s},\Nabla{t}]\xi\\
    &=\Nabla{t}\left(-R(\xi,\p_tu)\p_tu
     -\Hh_\Vv(u)\xi\right)
     -R(\p_su,\p_tu)\xi\\
    &=-\left(\Nabla{t} R\right)
     (\xi,\p_tu)\p_tu
     -R(\Nabla{t}\xi,\p_tu)\p_tu
     -R(\xi,\Nabla{t}\p_tu)\p_tu\\
    &\quad
     -R(\xi,\p_tu)\Nabla{t}\p_tu
     -\Nabla{t}\Hh_\Vv(u)\xi
     -R(\p_su,\p_tu)\xi.
\end{split}
\end{equation*}
Now take the pointwise
inner product of
this identity and
$\Nabla{t}\xi$ and estimate
the resulting six terms
separately using the 
$L^\infty$ boundedness
assumption of the various 
derivatives of $u$.
For instance, term five
satisfies the estimate
$$
     \Abs{\langle\Nabla{t}\xi,
     \Nabla{t}\Hh_\Vv(u)\xi\rangle}
     \le
     c_2\Abs{\Nabla{t}\xi}
     \left(
     \Abs{\Nabla{t}\xi}
     +(1+\abs{\p_tu})\left(
     \Abs{\xi}+\Norm{\xi_s}_{L^1(S^1)}\right)
     \right)
$$
by the second inequality
of axiom~(V2) with constant $c_2$.
It follows that $U_t$
satisfies the pointwise inequality
\begin{equation*}
    \abs{U_t}\le
    \mu f_1+\mu\Abs{\xi}^2
    +\mu\Norm{\xi_s}_{L^2(S^1)}^2
\end{equation*}
for a suitable constant
$\mu=\mu(C_0,\Vv)>0$.
Hence
\begin{equation}\label{eq:Lf_1}
     L f_1
     \ge 2g_1-\mu f_1
     -\mu \Abs{\xi}^2
     -\mu \Norm{\xi_s}_{L^2(S^1)}^2
\end{equation}
pointwise for all $s$ and $t$.
Integrate this
inequality over $t\in[0,1]$
to obtain that
$$
     -F_1^\prime
     \ge 2G_1-\mu F_1
     -2\mu F
$$
pointwise for every $s\in\R$.
Then corollary~\ref{co:kerD-apriori-L2}
with $R=r=\frac{1}{2}$
shows that
$$
     \int_{s_0-\frac{1}{4}}^{s_0}
     \Norm{\Nabla{t}\Nabla{t}\xi_s}^2 ds
     \le (\mu+20) \int_{s_0-1}^{s_0}
     \Norm{\Nabla{t}\xi_s}^2 ds
     +2\mu \int_{s_0-1}^{s_0}
     \Norm{\xi_s}^2 ds
$$
for every $s_0\in\R$.
Now
$$
     \int_{s_0-1}^{s_0}
     \Norm{\Nabla{t}\xi_s}^2 ds
     \le 16C_1 \int_{s_0-\frac{5}{4}}^{s_0}
     \Norm{\xi_s}^2 ds
$$
by step~1
and this proves
the first assertion of step~3.
(We need this result
only in the proof of
theorem~\ref{thm:kerD-apriori-II} below.)

To prove the second assertion
of step~3, that is the 
estimate for
$\abs{\Nabla{t}\xi}$,
note that estimate~(\ref{eq:Lf_1}),
step~1, and step~2
imply the pointwise estimate
$$
     Lf_1\ge -\mu f_1
     -\mu
     \Norm{\xi}_{L^2([s-\frac{1}{2},s]\times S^1)}^2
$$
for all $s$ and $t$.
Here we have chosen a
larger value for
the constant $\mu$.
Fix $(s_0,t_0)\in\R\times S^1$
and set
$
     a=a(s_0):=
     \Norm{\xi}_{L^2([s_0-1,s_0]\times S^1)}^2
$.
Then
$$
     L\left(f_1+a\right)
     \ge-\mu\left(f_1+a\right)
$$
for all $t$ and $s\in[s_0-\frac{1}{2},s_0]$.
Hence
lemma~\ref{le:kerD-apriori-basic} 
with $r=\frac{1}{2}$ applies to the
function
$w(s,t):=f_1(s_0+s,t_0+t)+a$
and proves the desired estimate, namely
\begin{equation*}
\begin{split}
     f_1(s_0,t_0)
    &\le 8c_1e^{\mu/4}
     \int_{-\frac{1}{4}}^0\int_0^1
     \left( f_1(s_0+s,t_0+t)+a \right)\: dtds\\
    &=8c_1e^{\mu/4}\left(\frac{1}{2}
     \int_{s_0-\frac{1}{4}}^{s_0}\int_0^1
     \Abs{\Nabla{t}\xi(s,t)}^2 dtds
     +\frac{a}{4}\right)\\
    &\le 8c_1e^{\mu/4}\left(
     2\Norm{\xi}_{L^2([s_0-\frac{1}{2},s_0]\times S^1)}^2
     +\frac{1}{4}
     \Norm{\xi}_{L^2([s_0-1,s_0]\times S^1)}^2
     \right)
\end{split}
\end{equation*}
for all $s_0\in\R$ and $t_0\in S^1$.
The final inequality
uses the estimate of step~1.
This concludes the proof 
of step~3 and
theorem~\ref{thm:kerD-apriori}.
\end{proof}

\begin{proof}
[Proof of theorem~\ref{thm:kerD-apriori-II}]
Occasionaly
we denote $\xi(s,t)$ by $\xi_s(t)$.
Define the functions
$f_2,g_2:\R\times S^1\to\R$
by
$$
     f_2:=\frac{1}{2}\abs{\Nabla{t}\Nabla{t}\xi}^2,\qquad
     g_2:=\frac{1}{2}\abs{\Nabla{t}\Nabla{t}\Nabla{t}\xi}^2
$$
and abbreviate
$
     L:=\p_t\p_t-\p_s
$
and
$
     \Ll:=\Nabla{t}\Nabla{t}-\Nabla{s}
$.
Then
\begin{equation}\label{eq:ker-Lf_tt}
     Lf_2=2g_2+U_{tt},\qquad
     U_{tt}:=\langle\Nabla{t}\Nabla{t}\xi,
     \Ll\Nabla{t}\Nabla{t}\xi\rangle.
\end{equation}
We estimate $\abs{U_{tt}}$.
Since $\xi$ solves 
the linear heat 
equation~(\ref{eq:kerD-heat}),
it follows that
\begin{equation*}
\begin{split}
     \Ll \Nabla{t}\Nabla{t}\xi
    &=\Nabla{t}\Nabla{t}\left(\Nabla{t}\Nabla{t}\xi
     -\Nabla{s}\xi\right)
     +[\Nabla{t}\Nabla{t},\Nabla{s}]\xi\\
    &=\Nabla{t}\Nabla{t}\left(-R(\xi,\p_tu)\p_tu
     -\Hh_\Vv(u)\xi\right)
     +\Nabla{t}[\Nabla{t},\Nabla{s}]\xi
     +[\Nabla{t},\Nabla{s}]\Nabla{t}\xi\\
    &=\Nabla{t}\Bigl(
     -\left(\Nabla{t} R\right)
     (\xi,\p_tu)\p_tu
     -R(\Nabla{t}\xi,\p_tu)\p_tu
     -R(\xi,\Nabla{t}\p_tu)\p_tu\\
    &\quad
     -R(\xi,\p_tu)\Nabla{t}\p_tu\Bigr)
     -\Nabla{t}\Nabla{t}\Hh_\Vv(u)\xi
     +\left(\Nabla{t} R\right)
     (\p_tu,\p_su)\xi\\
    &\quad
     +R(\Nabla{t}\p_tu,\p_su)\xi
     +R(\p_tu,\Nabla{t}\p_su)\xi
     +2R(\p_tu,\p_su)\Nabla{t}\xi.
\end{split}
\end{equation*}
Now take the pointwise
inner product of this identity
and $\Nabla{t}\Nabla{t}\xi$. Estimate
the resulting sum term by term and
use the assumption that
various derivatives of $u$
are bounded in $L^\infty$.
It follows that
$$
     \Abs{U_{tt}}
     \le \mu_1\Abs{\Nabla{t}\Nabla{t}\xi}
     \left(\Abs{\xi}+\Abs{\Nabla{t}\xi}
     +\Abs{\Nabla{t}\Nabla{t}\xi}\right)
     +\Abs{\Nabla{t}\Nabla{t}\xi}\cdot
     \Abs{\Nabla{t}\Nabla{t}\Hh_\Vv(u)\xi}
$$
for some positive constant
$\mu_1$ which depends only
on the $L^\infty$ bound $C_0$.
Note that by axiom~(V3)
there is a positive constant
$c_3=c_3(\Vv)$ such that
\begin{equation*}
\begin{split}
     \Abs{\Nabla{t}\Nabla{t}\Hh_\Vv(u)\xi}
    &\le c_3
     \Abs{\Nabla{t}\Nabla{t}\xi}
     +c_3\left(1+\Abs{\p_tu}\right)\Abs{\Nabla{t}\xi}\\
    &\quad
     +c_3\left(1+\Abs{\p_tu}^2
     +\Abs{\Nabla{t}\p_tu}\right)
     \left(\Abs{\xi}+\Norm{\xi_s}_{L^1(S^1)}\right).
\end{split}
\end{equation*}
Hence there is a
positive constant
$\mu_2=\mu_2(C_0,\Vv)$ such that
$$
     \Abs{U_{tt}}
     \le \mu_2
     \left(f_2
     +\Abs{\Nabla{t}\xi}^2
     +\Abs{\xi}^2
     +\Norm{\xi_s}_{L^2(S^1)}^2
     \right).
$$
Theorem~\ref{thm:kerD-apriori}
applied to the last three 
terms of this sum
implies that
\begin{equation*}
    \abs{U_{tt}}\le
    \mu f_2+\mu
    \Norm{\xi}_{L^2([s-1,s]\times S^1)}^2
\end{equation*}
pointwise for all $s$ and $t$
and with a suitable constant
$\mu=\mu(C_0,\Vv)>0$.
Now $Lf_2\ge-\abs{U_{tt}}$
by~(\ref{eq:ker-Lf_tt})
and therefore
$$
     Lf_2
     \ge-\mu f_2-\mu
     \Norm{\xi}_{L^2([s-1,s]\times S^1)}^2
$$
pointwise
for all $s$ and $t$.
Fix $s_0\in\R$ and set
$
     a:=
     \Norm{\xi}_{L^2([s_0-2,s_0]\times S^1)}^2
$,
then
$$
     L\left(f_2+a\right)
     \ge-\mu\left(f_2+a\right)
$$
for all $t\in S^1$ and $s\in[s_0-1,s_0]$.
Fix $t_0\in S^1$
and apply
lemma~\ref{le:kerD-apriori-basic} 
with $r=1$
to the function
$w(s,t):=f_2(s_0+s,t_0+t)+a$
to obtain that
\begin{equation*}
\begin{split}
     f_2(s_0,t_0)
    &\le c_1e^\mu
     \int_{-1}^0\int_{-1}^{+1}
     \left( f_2(s_0+s,t_0+t)+a \right)
     \: dtds\\
    &=c_1e^\mu\left(
     \int_{s_0-1}^{s_0}\int_{0}^{1}
     \Abs{\Nabla{t}\Nabla{t}\xi(s,t)}^2\: dtds
     +2a\right)\\
    &\le c_1e^\mu\left(4C_3+2\right)
     \Norm{\xi}_{L^2([s_0-2,s_0]
     \times S^1)}^2.
\end{split}
\end{equation*}
Here the last inequality
follows by the estimate of step~3
in the proof of
theorem~\ref{thm:kerD-apriori}
with constant $C_3=C_3(C_0,\Vv)>0$.
Since $s_0\in\R$ 
and $t_0\in S^1$
were chosen arbitrarily,
the proof of the first
estimate of
theorem~\ref{thm:kerD-apriori-II}
is complete.

The second estimate,
that is the one for $\abs{\Nabla{s}\xi}$,
follows easily from the
fact that $\xi$
solves the linear heat
equation~(\ref{eq:kerD-heat}),
the estimate for 
$\abs{\Nabla{t}\Nabla{t}\xi}$
which we just proved,
the estimate for $\abs{\xi}$
of
theorem~\ref{thm:kerD-apriori},
and the estimate for
$\abs{\Hh_\Vv(u)\xi}$ 
provided by axiom~(V1).
This concludes the
proof of 
theorem~\ref{thm:kerD-apriori-II}.
\end{proof}

\subsection{Exponential decay}
\label{subsec:kerD-exp-decay}

Given a smooth loop $x:S^1\to M$
consider the linear operator
defined by
\begin{equation}\label{eq:A_x}
     A_x\xi
     =-\Nabla{t}\Nabla{t}\xi
     -R(\xi,\p_tx)\p_tx
     -\Hh_\Vv(x)\xi
\end{equation}
on $L^2(S^1,x^*TM)$ 
with dense domain 
$W^{2,2}(S^1,x^*TM)$.
With respect to the $L^2$ inner product 
$\langle\cdot,\cdot\rangle$
this operator is self-adjoint;
see e.g.~\cite{Joa-INDEX}
for the case of geometric
perturbations $V_t$
and use lemma~\ref{le:symmetry}
in the general case.

\begin{theorem}[Backward exponential decay]
\label{thm:kerD-exp-decay}
Fix a perturbation $\Vv:\Ll M\to\R$ 
that satisfies~{\rm (V0)--(V2)}
and a constant $c_0>0$.
Then there exist positive constants
$\delta,\rho,C$
such that the following holds.
Let $x:S^1\to M$ be a smooth loop
such that $A_x$ given
by~(\ref{eq:A_x}) is injective and
$
     \Norm{\p_tx}_2+\Norm{\Nabla{t}\p_tx}_2
     \le c_0
$.
Assume
$u:(-\infty,0]\times S^1\to M$
is a smooth map and $T_0>0$ 
is a constant such that
$$
     u_s=\exp_x \eta_s,
     \quad
     \Norm{\eta_s}_{W^{2,2}}\le \delta,
     \quad
     \Norm{\p_su_s}_2+\Norm{\Nabla{s}\p_tu_s}_2
     \le\delta,
$$
whenever $s\le-T_0$.
Assume further that $\xi$ is a smooth vector
field along $u$ such that the function
$s\mapsto\norm{\xi_s}_2$
is bounded by a constant $c=c(\xi)$
and $\xi$ solves one of two equations
\begin{equation}\label{eq:pm-heat}
     \pm\Nabla{s}\xi
     -\Nabla{t}\Nabla{t}\xi
     -R(\xi,\p_tu)\p_tu
     -\Hh_\Vv(u)\xi=0.
\end{equation}
Then
$$
     \Norm{\xi_s}_2^2
     \le e^{\rho(s+T_0)}
     \Norm{\xi_{-T_0}}_2^2
     \le c^2 e^{\rho(s+T_0)}
$$
and
$$
     \Norm{\xi}_{L^2((-\infty,s]\times S^1)}^2
     \le {\textstyle\frac{C^2}{\rho}} e^{\rho(s+T_0)}
     \Norm{\xi}_{L^2([-T_0-1,-T_0]\times S^1)}^2
$$
for every $s\le -T_0$.
\end{theorem}

Note the weak assumption
($L^2$ versus $L^\infty$) 
on the $s$-derivatives of $\p_tu_s$
and its base component $u_s$.
To prove theorem~\ref{thm:kerD-exp-decay}
we need two lemmas.

\begin{remark}[Forward exponential decay]
\label{rmk:decay-forward}
If the
domain of $u$
is the forward half cylinder
$[0,\infty)\times S^1$
and the vector field $\xi$
along $u$ solves~$\pm(\ref{eq:pm-heat})$,
then theorem~\ref{thm:kerD-exp-decay}
applies to
$v(\sigma,t):=u(-\sigma,t)$
and $\eta(\sigma,t):=\xi(-\sigma,t)$,
since $\eta$
solves~$\mp(\ref{eq:pm-heat})$.
The estimates obtained
for $\eta$ provide
estimates for $\xi$, for instance
$$
     \Norm{\xi}_{L^2([\sigma,\infty)\times S^1)}^2
     \le {\textstyle\frac{C^2}{\rho}} e^{\rho(-\sigma+T_0)}
     \Norm{\xi}_{L^2([T_0,T_0+1]\times S^1)}^2
$$
for every $\sigma\ge T_0$.
\end{remark}

\begin{lemma}[Stability of injectivity]\label{le:hessian}
Fix a perturbation $\Vv:\Ll M\to\R$ 
that satisfies~{\rm (V0)--(V2)}
and a constant $c_0>1$.
Then there are
constants $\mu,\delta_0>0$ 
such that the following holds.
If $x$ and $\gamma$ are
smooth loops in $M$
such that the operator
$A_x$ is injective and
$$
    \gamma=\exp_x(\eta),
    \qquad
    \Norm{\eta}_{W^{2,2}}
    \le\delta_0,
    \qquad
    \norm{\p_t x}_2
    +\norm{\Nabla{t}\p_t x}_2
    \le c_0,
$$
then
$$
    \Norm{\xi}_2+\Norm{\Nabla{t}\xi}_2
    +\Norm{\Nabla{t}\Nabla{t}\xi}_2
    \le \mu\Norm{A_\gamma \xi}_2
$$
for every $\xi\in\Om^0(S^1,\gamma^*TM)$.
\end{lemma}

\begin{proof}
By self-adjointness and injectivity
the operator $A_x$
is bijective. Hence it
admits a bounded inverse
by the open mapping theorem.
This proves the estimate
in the case $\gamma=x$
for some positive constant, 
say $\mu_0=\mu_0(\Vv,c_0)>1$.
Since bijectivity is preserved
under small perturbations
(with respect to the operator norm),
the result for general $x$
follows from continuous dependence
of the operator family on $\eta$
with respect to the
$W^{2,2}$ topology.
More precisely, 
given a smooth vector
field $\xi$ along $\gamma$,
define $X=\Phi^{-1}\xi$ where $\Phi=\Phi(x,\eta)$
denotes parallel transport
along the geodesic
$[0,1]\ni\tau\mapsto\exp_x(\tau\eta)$.
Recall that $\Phi$ is pointwise
an isometry, then
straightforward calculation shows that
$$
    \Norm{\xi}_2+\Norm{\Nabla{t}\xi}_2
    +\Norm{\Nabla{t}\Nabla{t}\xi}_2
    \le cc_0^2\mu_0\Norm{\Phi A_x\Phi^{-1}\xi}_2
$$
where the constant
$c>1$ depends only on the
closed Riemannian manifold $M$
and the constant $c_1$
associated to the Sobolev
embedding $W^{1,2}\hookrightarrow C^0$.
Now
$$
     \Norm{\Phi A_x\Phi^{-1}\xi-A_\gamma\xi}_2
    \le C \Norm{\eta}_{W^{2,2}}\Norm{\xi}_{W^{1,2}}
    \le \delta_0 C\Norm{\xi}_{W^{1,2}}
$$
by straightforward calculation,
where the constant $C>1$
depends on $\norm{R}_\infty$, $c_0$, $c_1$,
$\delta_0$, and the constant in axiom~(V2)
and where we estimated 
the term quadratic in $\Nabla{t}\eta$ by
$
     \norm{\Nabla{t}\eta}_\infty^2
     \le c_1^2\norm{\eta}_{W^{2,2}}^2
$.
The second inequality
uses the assumption on $\eta$.
Now combine both estimates
and choose $\delta_0>0$
sufficiently small to obtain
the assertion of the lemma
with $\mu=2cc_0^2\mu_0$.
\end{proof}

\begin{lemma}\label{le:inequality}
Let $f\ge 0$ be a $C^2$
function on the interval $(-\infty,-T_0]$.
If $f$ is bounded by a constant $c$
and satisfies
the differential inequality
$f^{\prime\prime}\ge \rho^2 f$
for some constant $\rho\ge 0$,
then
$$
     f(s)\le e^{\rho(s+T_0)} f(-T_0)
$$
for every $s\le-T_0$.
\end{lemma}

\begin{proof}
Although the 
argument is standard,
see e.g.~\cite{DOSA},
we provide the details
for the sake of completeness.
The main point is to
observe that
$
     f^\prime(s)-\rho f(s)\ge 0
$
for every $s\le-T_0$.
To see this assume by 
contradiction that 
$
     f^\prime(s_0)-\rho f(s_0)< 0
$
for some time
$s_0\le -T_0$.
Note that the function 
$ 
     g(s)
     =e^{\rho s}\left(
     f^\prime(s)-\rho f(s)\right)
$
satisfies $g^\prime\ge 0$
on $(-\infty,-T_0]$.
Hence
$g(s)\le g(s_0)$, or
equivalently
$$
     f^\prime(s)
     \le e^{\rho(s_0-s)}\left(
     f^\prime(s_0)-\rho f(s_0)\right)
     +\rho c
$$
for every $s\le s_0$.
It follows that
$f^\prime(s)\to-\infty$
as $s\to-\infty$
and therefore
$$
     \int_s^{s_0} f^\prime(\sigma)\; d\sigma
     \to
     -\infty,\quad
     \text{as $s\to-\infty$}.
$$
But this contradicts the fact that
by boundedness of $f$
$$
     \int_s^{s_0} f^\prime(\sigma)\; d\sigma
     =f(s_0)-f(s)
     \ge -c
$$
for every $s\le s_0$.
To conclude the proof
consider the function $h(s)=e^{-\rho s} f(s)$ 
on the interval $(-\infty,-T_0]$.
It follows from the
observation above that
$h^\prime\ge 0$.
Hence $h(s)\le h(-T_0)$ 
for every $s\le-T_0$
and this proves the lemma.
\end{proof}

To prove theorem~\ref{thm:kerD-exp-decay}
it is useful to 
denote $\exp_u(\xi)$ by $E(u,\xi)$
and define linear maps
$$
  E_i (u,\xi):T_uM\to T_{exp_u\xi}M ,\qquad
  E_{ij}(u,\xi):T_uM\times T_uM\to T_{exp_u\xi}M
$$
for $\xi\in T_xM$ and $i,j\in\{1,2\}$.
If $u:\R\to M$ is a smooth curve
and $\xi,\eta$ are smooth vector fields 
along $u$, then the maps $E_i$ and $E_{ij}$
are characterized by the identities
\begin{equation} \label{eq:exponential-identity}
\begin{split}
     \frac{d}{ds}\exp_u(\xi)
    &=E_1(u,\xi)\p_su
     +E_2(u,\xi)\Nabla{s}\xi
    \\
     \Nabla{s}\left( E_1(u,\xi)\eta\right)
    &=E_{11}(u,\xi)\left(\eta,\p_s u\right)
      +E_{12}(u,\xi)\left(\eta,\Nabla{s}\xi\right)
      +E_1(u,\xi)\Nabla{s}\eta
    \\
     \Nabla{s}\left( E_2(u,\xi)\eta\right)
    &=E_{21}(u,\xi)\left(\eta,\p_s u\right)
      +E_{22}(u,\xi)\left(\eta,\Nabla{s}\xi\right)
      +E_2(u,\xi)\Nabla{s}\eta.
\end{split}
\end{equation}
These maps satisfy the symmetry properties
\begin{equation} \label{eq:E12-symmetry}
     E_{12}(u,\xi)\left(\eta,\eta^\prime\right)
     =E_{21}(u,\xi)\left(\eta^\prime,\eta\right),\quad
     E_{22}(u,\xi)\left(\eta,\eta^\prime\right)
     =E_{22}(u,\xi)\left(\eta^\prime,\eta\right),
\end{equation}
and the identities
\begin{equation}\label{eq:E_ij(0)}
     E_{11}(u,0)=E_{12}(u,0)=E_{22}(u,0)=0,\qquad
     E_1(u,0)=E_2(u,0)=\1.
\end{equation}
Alternatively $E_2$ can be defined by
$$
     E_2(u,\xi)\eta 
     :=\left.\frac{d}{d\tau}\right|_{\tau=0}
     \exp_u (\xi+\tau \eta)
$$
for $\xi,\eta\in T_uM$ and $\tau\in\R$.
An explicit definition of $E_1$ 
and the maps $E_{ij}$
can be given in local coordinates.

\begin{proof}[Proof of theorem~\ref{thm:kerD-exp-decay}]
Fix $c_0$ and $\Vv$ and let $\mu$ and 
$\delta_0$ be the constants of 
lemma~\ref{le:hessian} 
and $C$ be the constant of
theorem~\ref{thm:kerD-apriori}
with this choice. Set $\delta:=\delta_0$
and suppose $u,x,T_0,\xi$ satisfy
the assumptions of the theorem.
Then lemma~\ref{le:hessian}
for $\gamma=u_s$
and vector fields $\eta=\eta_s$
and $\xi=\xi_s$ asserts that
\begin{equation}\label{eq:kerD-injective}
     \Norm{\xi_s}_2^2
     +\Norm{\Nabla{t}\xi_s}_2^2
     +\Norm{\Nabla{t}\Nabla{t}\xi_s}_2^2
     \le \mu^2
     \Norm{A_{u_s}\xi_s}_2^2
     = \mu^2
     \Norm{\Nabla{s}\xi_s}_2^2
\end{equation}
whenever $s\le -T_0$.
The last step uses the consequence
$\Nabla{s}\xi_s=\mp A_{u_s}\xi_s$
of~(\ref{eq:A_x})
and~(\ref{eq:pm-heat}).
From now on we assume that $s\le -T_0$.
Observe that
\begin{equation*}
\begin{split}
     \p_tu_s
    &=E_1(x,\eta_s)\p_tx+E_2(x,\eta_s)\Nabla{t}\eta_s
     \\
     \Nabla{t}\p_tu_s
    &=E_{11}(x,\eta_s)\left(\p_tx,\p_tx\right)
     +2E_{12}(x,\eta_s)\left(\p_tx,\Nabla{t}\eta_s\right)
     +E_1(x,\eta_s)\Nabla{t}\p_tx\\
    &\qquad
     +E_{22}(x,\eta_s)\left(\Nabla{t}\eta_s,\Nabla{t}\eta_s\right)
     +E_2(x,\eta_s) \Nabla{t}\Nabla{t}\eta_s.
\end{split}
\end{equation*}
By the identities~(\ref{eq:E_ij(0)})
we can choose $\delta>0$
smaller, if necessary,
such that
$$
     \Norm{\p_tu_s}_2
     \le\Norm{E_1(x,\eta_s)}_\infty\Norm{\p_tx}_2
     +\Norm{E_2(x,\eta_s)}_\infty\Norm{\Nabla{t}\eta_s}_2
     \le 2c_0.
$$
and, similarly, that
$
     \Norm{\Nabla{t}\p_tu_s}_2
     \le 2c_0
$.

\vspace{.1cm}
\noindent
{\bf Claim.}
{\it Consider the function
$$
     F(s)
     :=\frac{1}{2}\Norm{\xi_s}_2^2
     =\frac{1}{2}\int_0^1\abs{\xi(s,t)}^2\; dt.
$$
Then there is a sufficiently small
constant $\delta>0$ such that
$$
     F^{\prime\prime}(s)
     \ge \frac{1}{\mu^2} F(s)
$$
whenever $s\le-T_0$.
}
\vspace{.1cm}

\noindent
Before proving the claim
we show how it implies the conclusions
of theorem~\ref{thm:kerD-exp-decay}.
Set $\rho=\rho(c_0,\Vv):=1/\mu$,
then $F^{\prime\prime}\ge \rho^2 F$
on $(-\infty,T_0]$.
Hence lemma~\ref{le:inequality} 
proves the first conclusion of
theorem~\ref{thm:kerD-exp-decay}.
Use this conclusion, the fact that
$\norm{\cdot}_2\le\norm{\cdot}_\infty$
on the domain $S^1$,
and theorem~\ref{thm:kerD-apriori}
with constant $C=C(c_0,\Vv)$
to obtain that
$$
     \Norm{\xi_s}_2^2
     \le e^{\rho(s+T_0)}
     \Norm{\xi_{-T_0}}_\infty^2
     \le C^2 e^{\rho(s+T_0)}
     \Norm{\xi}_{L^2([-T_0-1,-T_0]\times S^1)}^2
$$
whenever $s\le -T_0$.
Fix $\sigma\le-T_0$
and integrate this estimate
over $s\in(-\infty,\sigma]$.
This proves the final
conclusion of
theorem~\ref{thm:kerD-exp-decay}.

It remains to
prove the claim.
In the following calculation
we drop the subindex $s$ for simplicity
and denote the $L^2(S^1)$ inner product
by $\langle\cdot,\cdot\rangle$.
By straightforward computation
it follows that
$$
     F^{\prime\prime}(s)
     =\Norm{\Nabla{s}\xi_s}_2^2
     +\bigl\langle\xi,
     \Nabla{s}\Nabla{s}\xi\bigr\rangle
$$
and
\begin{equation*}
\begin{split}
     \bigl\langle\xi,
     \Nabla{s}\Nabla{s}\xi\bigr\rangle
    &=\pm
     \bigl\langle\xi,
     \Nabla{s}\left(\Nabla{t}\Nabla{t}\xi
     +R(\xi,\p_tu)\p_tu+\Hh_\Vv(u)\xi\right)
     \bigr\rangle \\
    &=\pm
     \bigl\langle\xi,
     [\Nabla{s},\Nabla{t}\Nabla{t}]\xi
     +\Nabla{t}\Nabla{t}\Nabla{s}\xi
     +\Nabla{s}\left( R(\xi,\p_tu)\p_tu
     +\Hh_\Vv(u)\xi\right)
     \bigr\rangle \\
    &=\pm
     \bigl\langle\xi,
     \Nabla{t}[\Nabla{s},\Nabla{t}]\xi
     +[\Nabla{s},\Nabla{t}]\Nabla{t}\xi
     +\Nabla{s}\left( R(\xi,\p_tu)\p_tu
     +\Hh_\Vv(u)\xi\right)
     \bigr\rangle\\
    &\quad\pm\bigl\langle\Nabla{t}\Nabla{t}\xi,
     \Nabla{s}\xi\bigr\rangle \\
    &=\pm\bigl\langle
     \pm\Nabla{s}\xi-R(\xi,\p_tu)\p_tu-\Hh_\Vv(u)\xi,
     \Nabla{s}\xi\bigr\rangle\\
    &\quad\pm\bigl\langle\xi,
     \bigl(\Nabla{t}R\bigr)(\p_su,\p_tu)\xi
     +R(\Nabla{t}\p_su,\p_tu)\xi
     +R(\p_su,\Nabla{t}\p_tu)\xi\\
    &\qquad\quad\;
     +2R(\p_su,\p_tu)\Nabla{t}\xi
     +\bigl(\Nabla{s}R\bigr)(\xi,\p_tu)\p_tu
     +R(\Nabla{s}\xi,\p_tu)\p_tu\\
    &\qquad\quad\;
     +R(\xi,\Nabla{s}\p_tu)\p_tu
     +R(\xi,\p_tu)\Nabla{s}\p_tu
     +\Nabla{s}\Hh_\Vv(u)\xi
     \bigr\rangle\\
    &=\Norm{\Nabla{s}\xi}_2^2
     \pm\bigl\langle\xi,
     \Nabla{s}\Hh_\Vv(u)\xi
     -\Hh_\Vv(u)\Nabla{s}\xi
     \bigr\rangle\\
    &\quad
     \pm\bigl\langle\xi,
     \bigl(\Nabla{t}R\bigr)(\p_su,\p_tu)\xi
     +2R(\xi,\p_tu)\Nabla{t}\p_su
     +R(\p_su,\Nabla{t}\p_tu)\xi\\
    &\qquad\quad\;
     +2R(\p_su,\p_tu)\Nabla{t}\xi
     +\bigl(\Nabla{s}R\bigr)(\xi,\p_tu)\p_tu
     \bigr\rangle.\\
\end{split}
\end{equation*}
To obtain the first
and the fourth step
we replaced $\xi$
according to~(\ref{eq:pm-heat}).
The third step
is by integration by parts.
In the final step
we used twice the first 
Bianchi identity
and lemma~\ref{le:symmetry}
on symmetry of the Hessian.
Note that the term $\Nabla{t}\p_su$
forces us to assume $W^{1,2}$ 
and not only $L^\infty$ smallness
of $\p_su_s$.

\noindent
Abbreviate $\norm{\cdot}_{1,2}:=\norm{\cdot}_{W^{1,2}(S^1)}$
and assume from now on that $s\le-T_0$.
Recall that 
$\norm{\p_tu_s}_\infty
\le c_1\norm{\p_tu_s}_{1,2}
\le 4c_0c_1$ where $c_1$
is the Sobolev constant of the embedding
$W^{1,2}(S^1)\hookrightarrow C^0(S^1)$.
Then the former two identities
imply that
\begin{equation*}
\begin{split}
     F^{\prime\prime}(s)
    &\ge 2\Norm{\Nabla{s}\xi_s}_2^2
     -C_1\left(
     \Norm{\p_su_s}_\infty+\Norm{\Nabla{t}\p_su_s}_2
     \right)
     \left(
     \Norm{\xi_s}_\infty^2
     +\Norm{\xi_s}_\infty\Norm{\Nabla{t}\xi}_2
     \right)
   \\
     &\ge 2\Norm{\Nabla{s}\xi_s}_2^2
     -C_2 \Norm{\p_su_s}_{1,2}
     \Norm{\xi_s}_{1,2}^2
\end{split}
\end{equation*}
for positive constants
$C_1=C_1(c_0,c_1,\Vv,\norm{R}_{C^2})$
and $C_2=C_2(c_1,C_1)$.
Choose $\delta>0$
again smaller, if necessary,
namely such that
$\delta<1/(2\mu^2C_2)$.
Hence
$$
     \Norm{\p_su_s}_{1,2}
     \le \delta
     < \frac{1}{2\mu^2C_2}
$$
where the first inequality is by assumption.
Therefore
$$
     F^{\prime\prime}(s)
     \ge 2\Norm{\Nabla{s}\xi_s}_2^2
     -\frac{1}{2\mu^2}\Norm{\xi_s}_{1,2}^2
     \ge \Norm{\Nabla{s}\xi_s}_2^2
$$
where the second inequality
is by~(\ref{eq:kerD-injective}). But 
$$
     \norm{\Nabla{s}\xi_s}_2^2
     \ge\frac{1}{\mu^2} \norm{\xi_s}_2^2
     =\frac{2}{\mu^2} F(s)
$$
again by~(\ref{eq:kerD-injective})
and definition of $F$.
This proves the claim and 
theorem~\ref{thm:kerD-exp-decay}.
\end{proof}

\begin{lemma}[Symmetry of the Hessian]
\label{le:symmetry}
Fix a smooth map $\Vv:\Ll M\to\R$
and let $x:S^1\to M$
be a smooth loop. Then
$$
     \langle\Hh_\Vv(x)\xi,\eta\rangle
     =\langle\xi,\Hh_\Vv(x)\eta\rangle
$$
for all smooth vector fields
$\xi$ and $\eta$ along $x$.
\end{lemma}

\begin{proof}
Let $h:\R^2\to\Ll M$,
$(\sigma,\tau)\mapsto h(\sigma,\tau)$
be a smooth map
such that
$$
    h(0,0)=x,
    \qquad
    \left.\frac{\p}{\p\sigma}\right|_{0}h(\sigma,0)
    =\xi,
    \qquad
    \left.\frac{\p}{\p\tau}\right|_{0}h(0,\tau)
    =\eta.
$$
Observe that
\begin{equation*}
\begin{split}
    &\left.\frac{\p^2}{\p\tau\p\sigma}\right|_{(0,0)}
     \Vv(h(\sigma,\tau))\\
    &=
     \left.\frac{d}{d\tau}\right|_{0}
     d\Vv\mid_{h(0,\tau)}
     \left(\left.\frac{\p}{\p\sigma}\right|_{0}
     h(\sigma,\tau)\right)\\
    &=
     \left.\frac{d}{d\tau}\right|_{0}
     \left\langle\grad\Vv\mid_{h(0,\tau)},
     \left.\frac{\p}{\p\sigma}\right|_{0}
     h(\sigma,\tau)\right\rangle\\
    &=
     \left\langle
     \left.\frac{D}{d\tau}\right|_{0}
     \grad\Vv\mid_{h(0,\tau)},
     \left.\frac{\p}{\p\sigma}\right|_{0}
     h(\sigma,0)\right\rangle
     +
     \left\langle
     \grad\Vv(x),
     \left.\frac{D}{d\tau}\right|_{0}
     \left.\frac{\p}{\p\sigma}\right|_{0}
     h(\sigma,\tau)\right\rangle\\
   &=
     \left\langle
     \Hh_\Vv(x)\eta,\xi\right\rangle
     +
     \left\langle
     \grad\Vv(x),
     \left.\frac{D}{d\tau}\right|_{0}
     \left.\frac{\p}{\p\sigma}\right|_{0}
     h(\sigma,\tau)\right\rangle.
\end{split}
\end{equation*}
Now interchange the order of partial
differentiation and use the fact that
this is still valid for two-parameter maps.
\end{proof}

\subsection{The Fredholm operator}
\label{sec:linearized}

\begin{hypothesis}\label{hyp:fredholm}
Throughout this section we fix a perturbation
$\Vv$ that satisfies~{\rm (V0)--(V3)} and
two nondegenerate critical points $x^\pm$
of $\Ss_\Vv$. Fix
a smooth map $u:\R\times S^1\to M$
such that $u_s$ converges to $x^\pm$ in $W^{2,2}(S^1)$
and $\p_su_s$ converges to zero in $W^{1,2}(S^1)$,
as $s\to\pm\infty$.
Moreover, assume that
$\norm{\Nabla{t}\Nabla{t}\p_su_s}_2$
is bounded, uniformly in $s\in\R$;
see footnote below.
Set $x=x^-$ and $y=x^+$.
\end{hypothesis}

Note that by theorem~\ref{thm:par-exp-decay},
proved in section~\ref{subsec:exp-decay} below,
these assumptions are satisfied
if $\Ss_\Vv$ is Morse and $u$ is 
a finite energy solution
of the heat equation~(\ref{eq:heat}).
On the other hand,
the hypothesis guarantees that the
assumptions of the exponential decay
theorem~\ref{thm:kerD-exp-decay}
and the local regularity
theorem~\ref{thm:REG-L} --
only here~(V3) is needed --
are satisfied.
More precisely, set $a=\max\{\Ss_\Vv(x),\Ss_\Vv(y)\}$.
Then~(\ref{eq:action}) and~(\ref{eq:crit})
imply that
$$
     \Norm{\p_tx}_2^2=2a+2\Vv(x)
     \le 2(a+C_0),
     \qquad
     \Norm{\Nabla{t}\p_tx}_2=\Norm{\grad\,\Vv(x)}_2
     \le C_0.
$$
Here $C_0>0$ is the constant in axiom~(V0).
Similar estimates hold true for $y$.
Precisely as in the proof of 
theorem~\ref{thm:kerD-exp-decay}
it follows that $T=T(u)>0$ can be chosen
sufficiently large such that
$$
     \Norm{\p_tu_s}_2^2
     \le 2c_0,
     \qquad
     \Norm{\Nabla{t}\p_tu_s}_2=\Norm{\grad\,\Vv(x)}_2
     \le 2c_0
$$
whenever $\abs{s}\ge T_0$
and where $c_0=2(\abs{a}+C_0)$.
Hence by smoothness of $u$ and compactness
of the remaining domain $[-T,T]\times S^1$
we conclude that
\begin{equation}\label{eq:u-W-bound}
    \Norm{\p_tu_s}_\infty
    \le c_1  \Norm{\p_tu_s}_{W^{1,2}}
    \le c_2
\end{equation}
for every $s\in\R$
and where $c_2=c_2(x,y,u,\Vv)$.
Similarly it follows that
\begin{equation}\label{eq:p_su-W-bound}
    \Norm{\p_su_s}_\infty
    \le c_1\Norm{\p_su_s}_{W^{1,2}}
    \le c_3
\end{equation}
for every $s\in\R$
and some constant $c_3=c_3(x,y,u,\Vv)$.

Now consider the linear operator
$\Dd_u$ given by
\begin{equation}\label{eq:lin-op}
     \Dd_u\xi 
     = \Nabla{s}\xi
       - \Nabla{t}\Nabla{t}\xi 
       - R(\xi,\p_tu)\p_tu
       - \Hh_\Vv(u)\xi
\end{equation}
for smooth vector fields
$\xi$ along $u$.
Recall that $R$ denotes the
Riemannian curvature tensor on $M$.
The operator $\Dd_u$ arises, for instance,
by linearizing the heat equation~(\ref{eq:heat})
at a solution $u$;
see~\cite[app.~A.2]{Joa-PHD}.
Recall the definition of the 
Banach spaces $\Ll_u^p$ and $\Ww_u^{1,p}$
and their norms in~(\ref{eq:norms}).
The goal of this section
is to prove that
$\Dd_u:\Ww_u^{1,p}\to\Ll_u^p$
is a {\bf Fredholm operator}
whenever $p>1$ and $u$ satisfies
nondegenerate asymptotic boundary conditions 
as in hypothesis~\ref{hyp:fredholm}.
By definition this means that 
$\Dd_u$ is a bounded linear operator 
with closed range and
finite dimensional kernel and cokernel.
The difference of these dimensions
is called the {\bf Fredholm index} of $\Dd_u$
and denoted by $\INDEX\,\Dd_u$.
The {\bf formal adjoint operator}
$\Dd_u^*:\Ww^{1,p}_u\to\Ll^p_u$
with respect to the $L^2$-inner product
has the form
\begin{equation}\label{eq:formal-adjoint}
     \Dd_u^* \xi 
     = - \Nabla{s}\xi
       - \Nabla{t}\Nabla{t}\xi 
       - R(\xi,\p_tu)\p_tu
       - \Hh_\Vv(u)\xi.
\end{equation}

We proceed as follows.
In the case $p=2$ we show that
our situation suits the assumptions
of~\cite{RoSa-SPEC} where the
Fredholm property is proved.
Then we reduce the case $p>1$
to the case $p=2$ by
proving that the kernel and the cokernel
do actually not depend on $p$.
The argument is based on exponential decay
and local regularity,
theorem~\ref{thm:kerD-exp-decay} 
and theorem~\ref{thm:REG-L},
respectively.

\subsubsection*{Fredholm property
and index for \boldmath$p=2$}
\label{subsec:Fredholm_p=2}

To prove that $\Dd_u$ is Fredholm
it is useful to choose a representation
with respect to an orthonormal
frame along $u$. However, since $M$
is not necessarily orientable,
a frame which is periodic
in the $t$-variable might not exist.
Hence, given a smooth map
$u:\R\times S^1\to M$, we define
$$
     \sigma=\sigma(u)
     :=\begin{cases}
     +1,&\text{if $u^*TM\to\R\times S^1$ 
               is trivial}\\
     -1,&\text{else}
     \end{cases}
$$
and $E_\sigma:=\diag(\sigma,1,\dots,1)
\in\R^{n\times n}$.
The orthogonal group $\O(n)$
has two connected
components, one contains $E_1=\1$
and the other one $E_{-1}$.
Hence there exists an orthonormal frame
$\phi=\phi_\sigma:\R\times[0,1]\to u^*TM$
such that $\phi(s,1)=\phi(s,0)E_\sigma$
for all $s\in\R$.
The vector space of smooth sections
of $u^*TM$ is isomorphic to the space
$C^\infty_\sigma$
of all maps $X\in C^\infty(\R\times [0,1],\R^n)$
such that $X(s,1)=E_\sigma X(s,0)$,
for every $s\in\R$, and such that this
condition also holds for all derivatives
of $X$ with respect to the $t$-variable.

Denote by $W$ the closure of $C^\infty_\sigma$
with respect to the Sobolev
$W^{2,2}$ norm and by $H$
its closure with respect to the $L^2$ norm.
Then $\Dd_u:\Ww_u^{1,2}\to \Ll_u^2$
given by~(\ref{eq:lin-op})
is represented by the
Atiyah-Patodi-Singer type operator
\begin{equation}\label{eq:D_A}
     D_{A+C}:=\phi^{-1}\Dd_u\phi
     =\frac{d}{ds} + A(s) + C(s)
\end{equation}
from
$\Ww^{1,2}:=
L^2(\R,W)\cap W^{1,2}(\R,H)$
to
$L^2(\R,H)$.
Here $A(s)$ is the family 
of symmetric second order operators 
on $H$ with dense
domain $W$ given by
$$
     A(s)=
     -\frac{d^2}{dt^2}
     -B(s,t)-Q(s,t)
$$
where
$$
     Q
     =\phi^{-1} R(\phi,\p_tu)\p_tu
     +\phi^{-1}\Hh_\Vv(u)\phi
$$
and 
$$
     B=(\p_tP)+2P\p_t+P^2.
$$
The families of skew-symmetric
matrices $P(s,t)$ and $C(s,t)$
are determined by the identities
$$
     \phi^{-1}\Nabla{t}\phi=\p_t+P,\qquad
     \phi^{-1}\Nabla{s}\phi=\p_s+C.
$$
Hypothesis~\ref{hyp:fredholm}
implies that $\p_s u_s$
converges to zero
in $C^0(S^1)$, as $s\to\pm \infty$,
and therefore
$
     \lim_{s\to\pm\infty} C(s,t) = 0
$,
uniformly in $t$.
It follows that the family $C(s)$ of bounded operators
on $H$ -- defined pointwise by
matrix multiplication with $C(s,t)$ --
converges to zero in the norm topology
as $s\to\pm\infty$.
Hence the linear operator
$C:\Ww^{1,2}\to L^2$
is a compact perturbation
of $D_A$ by~\cite[lem.~3.18]{RoSa-SPEC}.
Since the Fredholm property
and the Fredholm index are
invariant under compact perturbations,
it suffices to
prove that $D_A$ is Fredholm
and compute its index.
By~\cite[thm.~A]{RoSa-SPEC}
it remains to verify
the following properties.
\begin{enumerate}
\item[(i)]
     The inclusion of Hilbert spaces $W\hookrightarrow H$
     is compact with dense image.
\item[(ii)]
     The operator
     $A(s):H\to H$ with dense domain $W$
     is unbounded and self-adjoint
     for every $s$.
\item[(iii)]
     The norm of $W$
     is equivalent to the graph norm
     of $A(s)$ for every $s$.
\item[(iv)]
     The map $\R\to\Ll(W,H):s\mapsto A(s)$
     is continuously differentiable
     with respect to the weak operator
     topology.
\item[(v)]
     There exist invertible operators
     $A^\pm\in\Ll(W,H)$ which
     are the limits
     of $A(s)$ in the norm topology,
     as $s$ tends to $\pm\infty$.
\end{enumerate}
Statements~(i) and~(ii)
follow by the Sobolev embedding
theorem, the well known fact
that the 1-dimensional Laplacian
$-d^2/dt^2$ on $[0,1]$
with periodic boundary conditions
is self-adjoint, and by the Kato-Rellich
Theorem since the perturbation
$B+Q$ is of relative bound zero;
see~\cite{ReSi-II}.
To prove~(iii) one has to
establish that
the $W$ norm is bounded above
by a constant times
the graph norm and vice versa.
The first inequality uses the
elliptic estimate for the operator $A(s)$
and the second one follows
since $\norm{\p_tu_s}_\infty$
and $\norm{\Nabla{t}\p_tu_s}_2$
are bounded by~(\ref{eq:u-W-bound})
and the Hessian
$\Hh_\Vv(u_s)$ is a bounded linear
operator on $L^2(S^1,{u_s}^*TM)$ by axiom~(V1).
To prove~(iv) we need to show that, 
given any $\xi\in W$
and $\eta\in H$, the map
$s\mapsto \langle\eta, A(s)\xi\rangle$
is in $C^1(\R,\R)$. 
This follows by the bounds in~(\ref{eq:u-W-bound}) 
and~(\ref{eq:p_su-W-bound}),
by the final estimate in axiom~(V2),
and the apparently 
unnatural\footnote{
     If in~\cite[thm.A]{RoSa-SPEC},
     hence in~(iv),
     \emph{continuously differentiable}
     could be replaced by
     \emph{continuous},
     then the assumption on
     $\norm{\Nabla{t}\Nabla{t}\p_su_s}_2$
     can be dropped in
     hypothesis~\ref{hyp:fredholm}
     and theorem~\ref{thm:fredholm}.}
assumption in hypothesis~\ref{hyp:fredholm}
that $\Nabla{t}\Nabla{t}\p_su_s$
be uniformly $L^2$ bounded.
Statement~(v) is true, since
the critical points $x^\pm$
are nondegenerate
and $u_s$ and $\p_tu_s$ converge in $C^0$
to $x^\pm$ and $\p_tx^\pm$, respectively,
and $\Nabla{t}\p_tu_s$ converges in $L^2$
to $\Nabla{t}\p_tx^\pm$, all as $s\to\pm\infty$.

The properties~(i--v)
are precisely the assumptions
of theorem~A in~\cite{RoSa-SPEC}
which asserts that the operator
$D_A:\Ww^{1,2}
\to L^2$ is Fredholm
and its index is given by the
spectral flow of the operator
family $A(s)$.
The spectral flow
represents the net change
in the number of negative
eigenvalues of $A(s)$ as
$s$ runs from $-\infty$ to $\infty$.
It is equal to
$\IND(A^-)-\IND(A^+)$
where $\IND(A^\pm)$
denotes the Morse index,
i.e. the number of negative 
eigenvalues of the
self-adjoint operator $A^\pm$.
To see this observe
that $\IND(A^+)$ equals
$\IND(A^-)$ plus the number of
eigenvalues changing from
positive to negative
minus the number of those
changing sign in the opposite direction.
Finally, the Fredholm indices of $D_A$
and $D_{A+C}$ are equal, since 
$\{D_{A+\tau C}\}_{\tau\in[0,1]}$ 
is an interpolating family of
Fredholm operators.
This proves theorem~\ref{thm:fredholm}
in the case $p=2$.

\begin{remark}[The formal adjoint]
\label{rmk:adjoint_fredholm_p=2}
If $\Dd_u:\Ww_u^{1,2}\to\Ll_u^2$ is represented
with respect to an orthonormal frame 
by the operator
$D_{A+C}$ in~(\ref{eq:D_A}),
then $\Dd_u^*$
is represented
by $-D_{-A-C}$. 
Above we proved that $A$
satisfies~(i-v), hence
so does $-A$. Thus
$D_{-A}$ is a Fredholm operator
again by~\cite[thm.~A]{RoSa-SPEC}
and its index is given by
minus the spectral flow of
the operator family $A=A(s)$.
But if $D_{-A}$
is Fredholm, so is its negative $-D_{-A}$
and both Fredholm indices are equal,
since both kernels and both cokernels coincide.
Now $-D_{-A}$ and $-D_{-A-C}$ are homotopic
through the family
$\{-D_{-A-\tau C}\}_{\tau\in[0,1]}$
of Fredholm operators.
This proves that
the formal adjoint operator
$\Dd_u^*:\Ww_u^{1,2}\to\Ll_u^2$ 
is Fredholm
and $\INDEX \Dd_u^*
=-\INDEX \Dd_u$.
\end{remark}

\subsubsection*{Fredholm property and index
for \boldmath$p>1$}
\label{subsec:Fredholm_p>1}

Still assuming hypothesis~\ref{hyp:fredholm}
consider the vector space given by
\begin{equation*}
\begin{split}
     X_0
    &:=\Bigl\{
     \xi\in
     C^\infty(\R\times S^1,u^*TM)
     \mid\text{
     $\Dd_u\xi=0$,
     $\exists c,\delta>0\;
     \forall s\in\R:$}
     \\
    &\qquad
     \text{
     $
       \Norm{\xi_s}_\infty
       +\Norm{\Nabla{t}\xi_s}_\infty
       +\Norm{\Nabla{t}\Nabla{t}\xi_s}_\infty
       +\Norm{\Nabla{s}\xi_s}_\infty
       \le ce^{\delta\abs{s}}
     $}
     \Bigr\}.
\end{split}
\end{equation*}
Define $X_0^*$ by using $\Dd_u^*$
in the definition.
Note that $p$ does not enter.

\begin{proposition}\label{prop:kernel-smooth}
Let $p>1$, then
$$
     \ker\left[\Dd_u:
     \Ww^{1,p}_u\to\Ll^p_u\right]
     =X_0,
$$
and
$$
     \ker\left[\Dd_u^*:
     \Ww^{1,p}_u\to\Ll^p_u\right]
     =X_0^*.
$$
\end{proposition}

\begin{proof}
The inclusion $\supset$ is trivial.
To prove the
inclusion $\subset$
assume that $\xi\in\Ww^{1,p}$
solves $\Dd_u\xi=0$ almost everywhere.
Being a local property
smoothness of $\xi$
follows from theorem~\ref{thm:REG-L}
using integration by parts.
Exponential $L^\infty$ decay
follows by combining the apriori estimates
theorem~\ref{thm:kerD-apriori}
and theorem~\ref{thm:kerD-apriori-II}
with the $L^2$ exponential
decay results
theorem~\ref{thm:kerD-exp-decay}
and remark~\ref{rmk:decay-forward}.
The last two results require nondegeneracy of
the critical points $x^\pm$
and boundedness of the map
$s\mapsto\norm{\xi_s}_2$.
To see the latter note that
$\norm{\xi_s}_p$
and $\norm{\Nabla{t}\xi_s}_p$
converge to zero
as $s\to\pm\infty$, because
$\xi$ and $\Nabla{t}\xi$ 
are $L^p$ integrable on $\R\times S^1$. 
Hence $\norm{\xi_s}_p
+\norm{\Nabla{t}\xi_s}_p\le C$
for some constant $C=C(p,\xi)$.
Now observe that
$$
     \Norm{\xi_s}_2
     \le\Norm{\xi_s}_\infty
     \le c_q\left(\Norm{\xi_s}_p
     +\Norm{\Nabla{t}\xi_s}_p\right)
$$
by the Sobolev embedding
$W^{1,p}(S^1)\hookrightarrow C^0(S^1)$
with constant $c_p$.
This proves that $X_0$ is the kernel of $\Dd_u$.
The result for $\Dd_u^*$ follows 
by reflection $s\mapsto -s$.
\end{proof}

\begin{proposition}\label{prop:closed-range}
The range of
$\Dd_u,\Dd_u^*:\Ww^{1,p}_u\to\Ll^p_u$
is closed whenever $p>1$.
\end{proposition}

\begin{proof}
The structure of proof 
is standard; see 
e.g.~\cite[sec.~2]{Sa-FLOER}.
We sketch the two key steps
for $\Dd_u$.
Step one is the linear estimate
$$
     \Norm{\xi}_{\Ww^{1,p}}
     \le c_p\left(
     \Norm{\Dd_u\xi}_p
     +\Norm{\xi}_p\right)
$$
for compactly supported
vector fields $\xi$ along $u$.
This follows immediately
from proposition~\ref{prop:par-linear},
lemma~\ref{le:plus-minus},
the $L^\infty$ bound for $\p_tu$ in~(\ref{eq:u-W-bound})
and axiom~(V1).
Step two is to prove bijectivity
of $\Dd_u$ in the case of
the constant cylinder $u(s,t)=x(t)$,
whenever $x$ is a nondegenerate
critical point of $\Ss_\Vv$.
We give a proof for $p\ge2$
in the related
situation of half cylinders
in theorem~\ref{thm:onto} below.
The case $1<p\le2$
follows by duality;
see~\cite[exc.~2.5]{Sa-FLOER}.
Both steps are then combined
by a cutoff function
argument; see~\cite[thm~2.2]{Sa-FLOER}.
\end{proof}

Proposition~\ref{prop:closed-range}
enables us to define the cokernels
of $\Dd_u:\Ww^{1,p}_u\to\Ll^p_u$
and 
$\Dd_u^*:\Ww^{1,p}_u\to\Ll^p_u$
as Banach space quotients, namely
for $p>1$ set
$$
      \coker\Dd_u
      :=\frac{\Ll^p_u}{\im \Dd_u},
$$
and
$$
      \coker\Dd_u^*
      :=\frac{\Ll^p_u}{\im \Dd_u^*}.
$$
The next result
shows that these spaces are again
independent of $p$.

\begin{proposition}\label{prop:coker-ker}
Let $p>1$, then
$$
     \coker\left[\Dd_u:
     \Ww^{1,p}_u\to\Ll^p_u\right]
     =X_0^*,
$$
and 
$$
     \coker\left[\Dd_u^*:
     \Ww^{1,p}_u\to\Ll^p_u\right]
     =X_0.
$$
\end{proposition}

\begin{proof}
We prove the second identity.
The other one follows 
by reflection $s\mapsto -s$.
Note that there is a natural complement
of the image of $\Dd_u^*$ in $\Ll^p_u$,
namely its orthogonal complement
with respect to the $L^2$ inner product.
Hence we identify
$$
     \coker\Dd_u^*
     \simeq\left(\im \Dd_u^*\right)^\perp.
$$
The inclusion $\supset$ is trivial.
To prove the inclusion $\subset$
assume that $\xi\in\left(\im \Dd_u^*\right)^\perp$.
This means that $\xi\in\Ll^p_u$ and that
$
     \langle\xi,\Dd_u^*\eta\rangle=0
$
for all $\eta\in C^\infty_0(\R\times S^1)$.
Hence $\xi$ is smooth by
theorem~\ref{thm:REG-L}.
Integration by parts then
shows that $\Dd_u\xi=0$.
Exponential decay
follows by combining
theorem~\ref{thm:kerD-apriori}
and
theorem~\ref{thm:kerD-apriori-II}
with theorem~\ref{thm:kerD-exp-decay}
and remark~\ref{rmk:decay-forward}
as explained in the proof of 
proposition~\ref{prop:kernel-smooth}.
\end{proof}

\begin{remark}\label{rmk:onto}
It is an easy but important
consequence of 
proposition~\ref{prop:coker-ker}
that if $\Dd_u:\Ww^{1,p}_u\to\Ll^p_u$
is surjective for some $p>1$,
then it is surjective for all
$p>1$. This justifies
the phrase
``$\Dd_u$ is surjective''
encountered occasionally.
\end{remark}

\begin{proof}[Proof of theorem~\ref{thm:fredholm}]
The range of
$\Dd_u:\Ww^{1,p}_u\to\Ll^p_u$ 
is closed by
proposition~\ref{prop:closed-range}.
Moreover, by
proposition~\ref{prop:kernel-smooth}
and 
proposition~\ref{prop:coker-ker}
the kernel and the cokernel
of $\Dd_u:\Ww^{1,p}_u\to\Ll^p_u$ 
are given by $X_0$ and $X_0^*$, respectively.
Now these vector 
spaces do not depend on $p>1$.
But for $p=2$ we proved in the
previous subsection
that they are finite dimensional
and the difference
of their dimensions equals
$\IND_\Vv(x^-)-\IND_\Vv(x^+)$.
The claim for $\Dd_u^*$ follows similarly.
\end{proof}

\section{Solutions of the nonlinear heat equation}
\label{sec:nonlinear-heat}

\subsection{Regularity and compactness}
\label{subsec:regularity-compactness}

Throughout this subsection
we embed the compact Riemannian manifold $M$
isometrically into some
Euclidean space $\R^N$
and view any continuous map
$u:Z=(-T,0]\times S^1\to M$ as a map
into $\R^N$ taking values
in the embedded manifold.
We indicate this by the notation
$u:Z\to M\hookrightarrow \R^N$.
Then the heat equation~(\ref{eq:heat})
is of the form
\begin{equation}\label{eq:heat-embedded}
     \p_su-\p_t\p_tu
     =\Gamma(u)\left(\p_tu,\p_tu\right)
     +F.
\end{equation}
Here and throughout
this section $\Gamma$
denotes the second fundamental form
associated to the embedding 
$M\hookrightarrow\R^N$ and
the map $F:Z\to \R^N$ is given by
\begin{equation}\label{eq:F}
     F(s,t):=(\grad\Vv(u_s))(t).
\end{equation}
Recall the definition of
the $\Ww^{k,p}$ and the $\Cc^k$
norm in~(\ref{eq:parabolic-Wk}) 
and~(\ref{eq:parabolic-Ck}), respectively.

\begin{proposition}\label{prop:reg-estimate}
Fix a perturbation 
$\Vv:\Ll M\to\R$ that satisfies~{\rm (V0)--(V3)},
constants $p>2$ and $\mu_0>0$,
and cylinders
$$
     Z=(-T,0]\times S^1,\qquad
     Z^\prime=(-T^\prime,0]\times S^1,\qquad
     T>T^\prime>0.
$$
Then for every integer $k\ge 1$
there is a constant
$c_k=c_k(p,\mu_0,T,T^\prime,\Vv)$
such that the following is true.
If $u:Z\to M\hookrightarrow \R^N$
is a $\Ww^{1,p}$ map such that
\begin{equation}\label{eq:reg-assumption}
     \Norm{u}_p
     +\Norm{\p_su}_p
     +\Norm{\p_tu}_p
     +\Norm{\p_t\p_tu}_p
     \le \mu_0
\end{equation}
and which satisfies the heat
equation~(\ref{eq:heat-embedded})
almost everywhere, then
$$
     \Norm{u}_{\Ww^{k,p}(Z^\prime,\R^N)}
     \le c_k.
$$
\end{proposition}

Proposition~\ref{prop:reg-estimate} follows
by induction from
the bootstrap
proposition~\ref{prop:bootstrap}
using all axioms~{\rm (V0)--(V3)}
and a product estimate,
lemma~\ref{le:product-s} below.
By standard arguments
proposition~\ref{prop:reg-estimate}
immediately implies
theorem~\ref{thm:regularity-local}
on regularity and
theorem~\ref{thm:compactness-gradbound}
on compactness.

\begin{theorem}[Regularity]
\label{thm:regularity-local}
Fix a perturbation
$\Vv:\Ll M\to\R$ that satisfies~{\rm (V0)--(V3)}
and constants $p>2$ and $a<b$. Let $u$
be a map $(a,b]\times S^1\to M\hookrightarrow \R^N$
which is of Sobolev class
$\Ww^{1,p}$ and solves the heat
equation~(\ref{eq:heat-embedded})
almost everywhere.
Then $u$ is smooth.
\end{theorem}

\begin{theorem}[Compactness]
\label{thm:compactness-gradbound}
Fix a perturbation $\Vv:\Ll M\to\R$ that
satisfies~{\rm (V0)--(V3)}
and constants $p>2$ and $a<b$.
Let $u^\nu:(a,b]\times S^1\to M\hookrightarrow \R^N$
be a sequence of smooth solutions
of the heat
equation~(\ref{eq:heat-embedded})
such that
$$
     \sup_\nu\Norm{\p_tu^\nu}_\infty
     +\sup_\nu\Norm{\p_su^\nu}_p
     <\infty.
$$
Then there is
a smooth solution $u:(a,b]\times S^1\to M$
of~(\ref{eq:heat-embedded})
and a subsequence, still denoted by $u^\nu$,
such that $u^\nu$ converges to $u$,
uniformly with all derivatives
on every compact subset
of $(a,b]\times S^1$.
\end{theorem}

\begin{lemma}\label{le:product-s}
Fix a constant $p>2$ and
a bounded open subset
$\Omega\subset\R^2$ with area $\abs{\Omega}$.
Then for every
integer $k\ge 1$
there is a constant $c=c(k,\abs{\Omega})$
such that
$$
     \Norm{\p_su\cdot v}_{\Ww^{k,p}}
     \le c \Norm{\p_su}_{\Ww^{k,p}}\Norm{v}_\infty
     +c\left(\Norm{u}_{\Cc^k}+\Norm{\p_tu}_{\Cc^k}\right)
     \Norm{v}_{\Ww^{k,p}}
$$
for all functions $u,v\in
C^\infty(\overline{\Omega})$.
\end{lemma}

\begin{proof}
The proof is by induction on $k$.
By definition of the
$\Ww^{\ell,p}$ norm
\begin{equation}\label{eq:ell+1-s}
\begin{split}
     \Norm{\p_su\cdot v}_{\Ww^{\ell+1,p}}
    &\le 
     \Norm{\p_su\cdot v}_{\Ww^{\ell,p}}
     +\Norm{\p_t\p_su\cdot v+\p_su\cdot\p_tv}_{\Ww^{\ell,p}}\\
    &\quad+\Norm{\p_t\p_t\p_su\cdot v+2\p_t\p_su\cdot\p_tv
     +\p_su\cdot\p_t\p_tv}_{\Ww^{\ell,p}}\\
    &\quad
     +\Norm{\p_s\p_su\cdot v+\p_su\cdot\p_sv}_{\Ww^{\ell,p}}.
\end{split}
\end{equation}

\vspace{.1cm}
\noindent
{\bf\boldmath Step $k=1$.}
Estimate~(\ref{eq:ell+1-s}) for $\ell=0$ shows that
\begin{equation*}
\begin{split}
     \Norm{\p_su\cdot v}_{\Ww^{1,p}}
    &\le\left(\Norm{\p_su}_p
     +\Norm{\p_t\p_su}_p+\Norm{\p_t\p_t\p_su}_p
     +\Norm{\p_s\p_su}_p\right)\Norm{v}_\infty\\
    &\quad
     +2\Norm{\p_t\p_su}_\infty\Norm{\p_tv}_p\\
    &\quad+\Norm{\p_su}_\infty
     \left(\Norm{\p_tv}_p+\Norm{\p_t\p_tv}_p
     +\Norm{\p_sv}_p\right).
\end{split}
\end{equation*}
Since $\p_t\p_su=\p_s\p_tu$
this proves the lemma for $k=1$.

\vspace{.1cm}
\noindent
{\bf\boldmath Induction step $k\Rightarrow k+1$.}
Consider estimate~(\ref{eq:ell+1-s}) for $\ell=k$,
then inspect the right hand side term by term
using the induction hypothesis
for the appropriate functions
to conclude the proof.
To illustrate this we give full details
for the last term in~(\ref{eq:ell+1-s}),
namely
\begin{equation*}
\begin{split}
     \Norm{\p_su\cdot\p_sv}_{\Ww^{k,p}}
    &\le c\Norm{\p_su}_{\Ww^{k,p}}
     \Norm{\p_sv}_\infty
     +c\left(\Norm{u}_{\Cc^k}+\Norm{\p_tu}_{\Cc^k}\right)
     \Norm{\p_sv}_{\Ww^{k,p}}\\
    &\le cc_1\Abs{\Omega} \Norm{\p_su}_{\Cc^k}
     \Norm{\p_sv}_{\Ww^{1,p}}
     +c\left(\Norm{u}_{\Cc^k}+\Norm{\p_tu}_{\Cc^k}\right)
     \Norm{v}_{\Ww^{k+1,p}}\\
    &\le cc_1\Abs{\Omega} \Norm{u}_{\Cc^{k+1}}
     \Norm{v}_{\Ww^{2,p}}
     +c\left(\Norm{u}_{\Cc^k}+\Norm{\p_tu}_{\Cc^k}\right)
     \Norm{v}_{\Ww^{k+1,p}}.
\end{split}
\end{equation*}
The first step is by the induction
hypothesis for the function $\p_sv$. 
In the second step
we pulled out the $L^\infty$ norms
of all derivatives of $\p_su$
and for the term $\p_sv$
we applied the Sobolev embedding
$\Ww^{1,p}\subset W^{1,p}\hookrightarrow C^0$
with constant $c_1$.
Here our assumptions $p>2$
and $\Omega$ bounded enter.
Step three is obvious.
Note that $k\ge1$ implies that
$\Ww^{k+1,p}\hookrightarrow\Ww^{2,p}$.
\end{proof}

\begin{proof}[Proof of
proposition~\ref{prop:reg-estimate}]
Consider the family
$$
     T_r:=T^\prime +\frac{T-T^\prime}{r},
     \qquad r\in[1,\infty),
$$
and the corresponding nested 
sequence of cylinders
$Z_r:=(-T_r,0]\times S^1$ with
$$
     Z=Z_1\supset
     Z_2\supset
     Z_3\supset
     \ldots\supset
     Z^\prime.
$$
Denote by $C_0$ the constant in~(V0).
More generally, for $\ell\ge1$ choose $C_\ell$
larger than $C_{\ell-1}$ and 
larger than all constants
$C(k^\prime,\ell^\prime,\Vv)$ in~(V3)
for which $2k^\prime+\ell^\prime\le \ell$.

\vspace{.1cm}
\noindent
{\bf Claim.} 
{\it
The map $F$ given by~(\ref{eq:F}) is in
$\Ww^{\ell,p}(Z_{\ell+1})$
for every integer $\ell\ge 1$.
}

\vspace{.1cm}
\noindent
Proposition~\ref{prop:reg-estimate}
immediately follows:
Given any integer $k\ge1$,
then 
$F\in\Ww^{k,p}(Z_{k+1})$ by the claim.
Furthermore, by
inclusion $Z_{k+1}\subset Z$ 
and~(\ref{eq:reg-assumption})
$$
     \Norm{u}_{\Ww^{1,p}(Z_{k+1})}
     \le\Norm{u}_{\Ww^{1,p}(Z)}
     \le\mu_0.
$$
Hence by corollary~\ref{cor:bootstrap}
for the pair $Z_{k+2}\subset Z_{k+1}$ 
there is a constant
$c_{k+1}$ depending on
$p$, $\mu_0$, $Z_{k+2}$, $Z_{k+1}$,
$\norm{\Gamma}_{C^{2k+2}}$, and
$\norm{F}_{\Ww^{k,p}(Z_{k+1})}$
such that
$$
    \Norm{u}_{\Ww^{k+1,p}(Z^\prime)}
    \le \Norm{u}_{\Ww^{k+1,p}(Z_{k+2})}
    \le c_{k+1}.
$$
It remains 
to prove the claim.
The proof is by induction.

\vspace{.1cm}
\noindent
{\bf\boldmath Step $\ell=1$.}
We need to prove that
$F$, $\p_tF$, $\p_sF$, and $\p_t\p_tF$
are in $L^p(Z_2)$.
The domain of all norms
of $\Gamma$ and its derivatives
is the compact manifold $M$.
The domain of all other norms
is the cylinder $Z$ unless indicated differently.
By axiom~(V0) with constant $C_0$
it follows
(even on the larger domain $Z$)
that
\begin{equation}\label{eq:Finfty}
     \Norm{F}_\infty
     =\sup_{s\in(-T,0]} 
     \Norm{\grad\Vv(u_s)}_{L^\infty(S^1)}
     \le C_0
\end{equation}
and therefore
$$
     \Norm{F}_p
     \le \Norm{F}_\infty
     \left(\Vol\, Z\right)^{1/p}
     \le C_0 T^{1/p}.
$$
Next we  use axiom~(V1)
with constant $C_1\ge C_0$
to obtain that
\begin{equation*}
\begin{split}
     \Norm{\p_t F}_p
    &\le
     \Norm{\Nabla{t}\grad\Vv(u)}_p
     +\Norm{\Gamma(u)\left(\p_tu,
     \grad\Vv(u)\right)}_p
     \\
    &\le
     C_1\left(1+\Norm{\p_tu}_p
     \right)
     +\Norm{\Gamma}_\infty
     \Norm{\p_tu}_p
     \Norm{F}_\infty
     \\
    &\le
     C_1(1+\mu_0)+\Norm{\Gamma}_\infty \mu_0 C_0.
\end{split}
\end{equation*}
Here we used the
assumption~(\ref{eq:reg-assumption})
in the last step.
Now by the bootstrap
proposition~\ref{prop:bootstrap}~(i)
for $k=1$ and the pair
$Z_{4/3}\subset Z$ there is a constant
$a_1$ depending on $p$, $\mu_0$, $Z_{4/3}$, $Z$,
$\norm{\Gamma}_{C^4}$,
and the $L^p(Z)$ norms of $F$ and $\p_tF$
such that
$
     \norm{\p_tu}_{\Ww^{1,p}(Z_{4/3})}
     \le a_1
$.
Then by the Sobolev embedding
$W^{1,p}\hookrightarrow C^0$
with constant $c^\prime=c^\prime(p,Z_{5/3})$
it follows that
$\p_tu$ is continuous on $Z_{4/3}$
and
\begin{equation}\label{eq:u-C0}
     \Norm{\p_tu}_{C^0(Z_{5/3})}
     \le c^\prime
     \Norm{\p_tu}_{\Ww^{1,p}(Z_{5/3})}
     \le a_1 c^\prime.
\end{equation}
Again using axiom~(V1)
we obtain similarly that
\begin{equation*}
\begin{split}
     \Norm{\p_s F}_p
    &\le
     \Norm{\Nabla{s}\grad\Vv(u)}_p
     +\Norm{\Gamma(u)\left(\p_su,
     \grad\Vv(u)\right)}_p \\
    &\le
     2C_1\Norm{\p_su}_p
     +\Norm{\Gamma}_\infty
     \Norm{\p_su}_p
     \Norm{F}_\infty \\
    &\le\mu_0\left(
     2C_1+\Norm{\Gamma}_\infty C_0\right).
\end{split}
\end{equation*}
In order to estimate $\p_t\p_tF$
observe first that
\begin{equation*}
\begin{split}
     \Norm{\Nabla{t}\p_tu}_{L^p(Z_{5/3})}
    &\le
     \Norm{\p_t\p_tu}_{L^p(Z_{5/3})}
     +\Norm{\Gamma}_\infty
     \Norm{\Abs{\p_tu}\cdot\Abs{\p_tu}}_{L^p(Z_{5/3})}
     \\
    &\le \mu_0
     +\Norm{\Gamma}_\infty
     \Norm{\p_tu}_{C^0(Z_{5/3})}
     \Norm{\p_tu}_{L^p(Z_{5/3})}
     \\
    &\le \mu_0
     +\Norm{\Gamma}_\infty 
     a_1c^\prime \mu_0.
\end{split}
\end{equation*}
Here the last step
uses assumption~(\ref{eq:reg-assumption})
and the $C^0$ estimate~(\ref{eq:u-C0})
for $\p_tu$
which requires shrinking of the domain.
Now by axiom~(V3) for $k=0$ and $\ell=2$
there is a constant still denoted by $C_1=C_1(\Vv)$
such that
\begin{equation}\label{eq:axiomV2-tt}
     \Abs{\Nabla{t}\Nabla{t} F} 
     \le C_1\Bigl(1+\Abs{\p_tu}
     +\Abs{\Nabla{t}\p_tu}
     \Bigr)
\end{equation}
pointwise for every $(s,t)$.
Integrating this inequality to the power $p$
implies that
\begin{equation*}
\begin{split}
     \Norm{\Nabla{t}\Nabla{t} F}_{L^p(Z_{5/3})}
    &\le
     C_1\left(1+\Norm{\p_tu}_{L^p(Z_{5/3})}
     +\Norm{\Nabla{t}\p_tu}_{L^p(Z_{5/3})}
     \right)
     \\
    &\le 
     C_1\left(1+2\mu_0
     +\Norm{\Gamma}_\infty a_1c^\prime \mu_0
     \right).
\end{split}
\end{equation*}
Straightforward calculation shows that
\begin{equation*}
\begin{split}
     \Norm{\p_t\p_t F}_{L^p(Z_{5/3})}
    &\le
     \Norm{\Nabla{t}\Nabla{t} F}_{L^p}
     +\Norm{d\Gamma}_\infty
     \Norm{\p_tu}_{C^0}
     \Norm{\p_tu}_{L^p}
     \Norm{F}_{C^0}\\
    &\quad+\Norm{\Gamma}_\infty
     \Norm{\p_t\p_tu}_{L^p}
     \Norm{F}_{C^0}
     +2\Norm{\Gamma}_\infty
     \Norm{\p_tu}_{C^0}
     \Norm{\p_tF}_{L^p}\\
    &\quad+\Norm{\Gamma}_\infty^2
     \Norm{\p_tu}_{C^0}
     \Norm{\p_tu}_{L^p}
     \Norm{F}_{C^0}
\end{split}
\end{equation*}
is bounded by a constant
$c=c(p,\mu_0,c^\prime,C_1,
\norm{\Gamma}_{C^1})$.
Here all $C^0$ and $L^p$ norms
are on the domain $Z_{5/3}$.
We used again
assumption~(\ref{eq:reg-assumption}),
the estimates for $F$
and its derivatives
obtained earlier, and~(\ref{eq:u-C0}).

\vspace{.1cm}
\noindent
{\bf\boldmath Induction step $\ell\Rightarrow \ell+1$.}
Let $\ell\ge 1$ and assume that
the claim is true for $\ell$.
This means that
$F$ is in $\Ww^{\ell,p}(Z_{\ell+1}$,
hence
$$
     \alpha_\ell:=\Norm{F}_{\Ww^{\ell,p}(Z_{\ell+1})}<\infty.
$$
Therefore by
corollary~\ref{cor:bootstrap}
for the integer $\ell$ and
the pair of sets $Z_{\ell+1}\supset Z_{\ell+3/2}$
there is a constant
$c_{\ell}=c_\ell(p,\mu_0,T_{\ell+1},T_{\ell+3/2},
\norm{\Gamma}_{C^{2\ell+2}},\alpha_\ell)$
such that
\begin{equation}\label{eq:u-ell+1}
     \Norm{u}_{\Ww^{\ell+1,p}(Z_{\ell+3/2})}
     \le c_\ell,\qquad
     \Norm{u}_{\Cc^\ell(Z_{\ell+3/2})}
     \le c_\ell.
\end{equation}
The second inequality follows
from the first
by the Sobolev embedding
$W^{1,p}\hookrightarrow C^0$ applied to
each term in the $\Cc^\ell$ norm.
Then choose $c_\ell$ larger, if necessary.
It remains to prove
that the $\Ww^{\ell,p}(Z_{\ell+2})$ norms
of $\p_tF$, $\p_sF$,
and $\p_t\p_tF$ are finite.
Similarly as in step $\ell=1$
we obtain that
\begin{equation*}
\begin{split}
     \Norm{\p_t F}_{\Ww^{\ell,p}(Z_{\ell+3/2})}
    &\le
     \Norm{\Nabla{t} F}_{\Ww^{\ell,p}}
     +\Norm{\Gamma(u)\left(\p_tu,F\right)}_{\Ww^{\ell,p}}
     \\
    &\le
     C_1\left(\Norm{1}_{\Ww^{\ell,p}}
     +\Norm{\p_tu}_{\Ww^{\ell,p}}\right)\\
    &\quad
     +\tilde{c}\Norm{\Gamma}_{\Cc^\ell}
     \left(
     \Norm{\p_tu}_{\Ww^{\ell,p}} \Norm{F}_\infty
     +\Norm{u}_{\Cc^\ell} \Norm{F}_{\Ww^{\ell,p}}
     \right)\\
    &\le C_1\,(T^{1/p}+c_\ell)
     +\tilde{c}\Norm{\Gamma}_{\Cc^\ell}
     \left(
     c_\ell C_0+c_\ell \alpha_\ell
     \right).
\end{split}
\end{equation*}
Here the domain of all norms, except the one
of $\Gamma$, is $Z_{\ell+3/2}$.
The first step is by definition
of the covariant derivative
and the triangle inequality.
Step two uses axiom~(V1)
and lemma~\ref{le:product-Sobolev}
with constant $\tilde{c}$.
The last step uses the
estimates~(\ref{eq:Finfty}),~(\ref{eq:u-ell+1}),
and the definition of $\alpha_\ell$ 
in the induction hypothesis.
Now by the refined bootstrap
proposition~\ref{prop:bootstrap}
there is a constant $a_{\ell+1}$ such that
\begin{equation}\label{eq:p_tu-ell+1}
     \Norm{\p_tu}_{\Ww^{\ell+1,p}(Z_{\ell+2})}
     \le a_{\ell+1},\qquad
     \Norm{\p_tu}_{\Cc^\ell(Z_{\ell+2})}
     \le a_{\ell+1}.
\end{equation}
Next observe that
\begin{equation*}
\begin{split}
    &\Norm{\p_s F}_{\Ww^{\ell,p}(Z_{\ell+2})}\\
    &\le
     \Norm{\Nabla{s} F}_{\Ww^{\ell,p}}
     +\Norm{\Gamma(u)\left(\p_su,F\right)}_{\Ww^{\ell,p}}
     \\
    &\le
     2C_1\Norm{\p_su}_{\Ww^{\ell,p}}
     +C^\prime\Norm{\Gamma}_{\Cc^\ell}
     \left(
     \Norm{\p_su}_{\Ww^{\ell,p}} \Norm{F}_\infty
     +\left(\Norm{u}_{\Cc^\ell}+\Norm{\p_tu}_{\Cc^\ell}\right)
     \Norm{F}_{\Ww^{\ell,p}}
     \right)\\
    &\le 2C_1c_\ell
     +C^\prime\Norm{\Gamma}_{\Cc^\ell}
     \left( c_\ell C_0 +(c_\ell+a_{\ell+1})\alpha_\ell
     \right).
\end{split}
\end{equation*}
Here the domain of all norms, except the one
of $\Gamma$, is $Z_{\ell+2}$.
Again the first step is by definition
of the covariant derivative
and the triangle inequality.
Step two uses axiom~(V1)
and lemma~\ref{le:product-s}
with constant $C^\prime$.
The last step uses the
estimates~(\ref{eq:Finfty}),~(\ref{eq:u-ell+1}),~(\ref{eq:p_tu-ell+1}),
and the definition of $\alpha_\ell$
in the induction hypothesis.
Similarly as in step $\ell=1$
we obtain that
\begin{equation*}
\begin{split}
    &\Norm{\p_t\p_t F}_{\Ww^{\ell,p}(Z_{\ell+2})}\\
    &\le
     \Norm{\Nabla{t}\Nabla{t} F}_{\Ww^{\ell,p}}
     +\Norm{d\Gamma(u)\left(\p_tu,\p_tu,F\right)}_{\Ww^{\ell,p}}
     \\
    &\quad
     +\Norm{\Gamma(u)\left(\p_t\p_tu,F\right)}_{\Ww^{\ell,p}}
     +2\Norm{\Gamma(u)\left(\p_tu,\p_tF\right)}_{\Ww^{\ell,p}}
     \\
    &\quad
     +\Norm{\Gamma(u)\left(\p_tu,\Gamma(u)\left(\p_tu,F\right)
     \right)}_{\Ww^{\ell,p}}
     \\
    &\le C_1\left(T^{1/p}
     +\Norm{\p_tu}_{\Ww^{\ell,p}}
     +\Norm{\p_t\p_tu}_{\Ww^{\ell,p}}
     +\Norm{\Gamma}_{\Cc^\ell}\Norm{\p_tu}_{\Cc^\ell}
     \Norm{\p_tu}_{\Ww^{\ell,p}}\right)\\
    &\quad
     +\Norm{d\Gamma}_{\Cc^\ell}
     \Norm{\p_tu}_{\Cc^\ell}^2\Norm{F}_{\Ww^{\ell,p}}\\
    &\quad
     +\tilde{c}\Norm{\Gamma}_{\Cc^\ell}
     \left(
     \Norm{\p_t\p_tu}_{\Ww^{\ell,p}}\Norm{F}_\infty
     +\Norm{\p_tu}_{\Cc^\ell}\Norm{F}_{\Ww^{\ell,p}}
     \right)\\
    &\quad
     +2\Norm{\Gamma}_{\Cc^\ell}
     \Norm{\p_tu}_{\Cc^\ell}
     \Norm{\p_tF}_{\Ww^{\ell,p}}\\
    &\quad
     +\Norm{\Gamma}_{\Cc^\ell}^2
     \Norm{\p_tu}_{\Cc^\ell}^2
     \Norm{F}_{\Ww^{\ell,p}}.
\end{split}
\end{equation*}
Here the domain of all norms, except the one
of $\Gamma$, is $Z_{\ell+2}$.
In the second step we used axiom~(V2)
with constant $C_1$ to estimate the term
$\Nabla{t}\Nabla{t} F$ and we spelled
out the covariant derivative arising
in $\Nabla{t}\p_tu$. Moreover,
crudely pulling out $\Cc^\ell$ norms
worked for all terms but the third one,
the one involving $\p_t\p_tu$, here we
used lemma~\ref{le:product-s}
with constant $\tilde{c}$
for the functions $\p_t\p_tu$ and $F$.
Now all terms appearing on the right hand side
have been estimated earlier.
This proves the induction step
and therefore the claim and
proposition~\ref{prop:reg-estimate}.
\end{proof}

\begin{proof}[Proof of
theorem~\ref{thm:regularity-local}]
Fix any point
$z\in Z=(a,b]\times S^1$
and a subcylinder
$Z^\prime=(a^\prime,b]\times S^1$
that contains $z$ and
where $a^\prime\in(a,b)$.
Set $\mu_0=\norm{u}_{\Ww^{1,p}(Z,\R^N)}$,
then proposition~\ref{prop:reg-estimate}
for the function $\tilde{u}(s,t):=u(s+b,t)$
and the constants $T=b-a$ and $T^\prime=b-a^\prime$
implies that
$$
     u\in
     \bigcap_{k\ge 0} \Ww^{k,p}(Z^\prime,\R^N)
     =\bigcap_{k\ge 0} W^{k,p}(Z^\prime,\R^N)
     =C^\infty(\overline{Z^\prime},\R^N).
$$
See~\cite[app.~B.1]{MS}
for the last step.
Hence $u$ is locally smooth.
\end{proof}

\begin{proof}[Proof of
theorem~\ref{thm:regularity}]
Theorem~\ref{thm:regularity-local}.
\end{proof}

\begin{proof}[Proof of
theorem~\ref{thm:compactness-gradbound}]
Shifting the $s$ variable
by $b$ and setting $T=b-a$,
if necessary, we may assume
without loss of generality that
the maps $u^\nu$ are defined on $(-T,0]$
and, furthermore, 
by composition with
the isometric embedding
$M\hookrightarrow \R^N$ that
they take values in $\R^N$.
All norms are taken on 
the domain $(-T,0]\times S^1$,
unless indicated otherwise.
To apply proposition~\ref{prop:reg-estimate}
we need to verify that the maps
$u^\nu:(-T,0]\times S^1\to \R^N$
satisfy the four apriori estimates
in~(\ref{eq:reg-assumption})
for some constant $\mu_0$
independent of $\nu$.
To see this observe that
$$
     \Norm{u^\nu}_p
     \le \Norm{u^\nu}_\infty \Vol\, ((-T,0]\times S^1)
     \le c_1 T^{1/p}
$$
for some constant $c_1$
depending only on the
isometric embedding $M\hookrightarrow \R^N$
and the diameter of the compact manifold $M$.
By assumption
there is a constant $c_2$ independent of $\nu$
such that
$$
     \Norm{\p_tu^\nu}_p
     \le \Norm{\p_tu^\nu}_\infty T^{1/p}
     \le c_2 T^{1/p}
$$
and
$$
     \Norm{\p_su^\nu}_p
     \le c_2.
$$
Then it follows 
by the heat
equation~(\ref{eq:heat-embedded})
that
$$
     \Norm{\Nabla{t}\p_tu^\nu}_p
     \le \Norm{\p_su^\nu}_p +\Norm{\grad\Vv(u^\nu)}_p
     \le c_2+C_0T^{1/p}.
$$
In the second step we used~(V0)
to estimate $\grad\Vv(u^\nu)$
in $L^\infty$  from above by
a constant $C_0=C_0(\Vv)$.
By definition of
the covariant derivative
\begin{equation*}
\begin{split}
     \Norm{\p_t\p_tu^\nu}_p
    &\le \Norm{\Nabla{t}\p_tu^\nu}_p
     +\Norm{\Gamma}_{C^0(M)} 
     \Norm{\p_tu^\nu}_\infty
     \Norm{\p_tu^\nu}_p\\
    &\le c_2+C_0T^{1/p}
     +c_2^2T^{1/p}\Norm{\Gamma}_{C^0(M)}.
\end{split}
\end{equation*}
Now set
$
     \mu_0
     :=c_2+C_0T^{1/p}
     +c_2^2T^{1/p}\Norm{\Gamma}_{C^0(M)}
     +(c_1+c_2)T^{1/p}
$.
Then proposition~\ref{prop:reg-estimate}
asserts that for every
constant $T^\prime\in(0,T)$
and every integer $k\ge 2$
there is a constant $c_k=c_k(p,\mu_0,T,T^\prime,\Vv)$
such that
$$
     \Norm{u^\nu}_{\Ww^{k,p}(Q,\R^N)}
     \le c_k
$$
where $Q=[-T^\prime,0]\times S^1$.
Recall that the inclusion
$W^{k,p}(Q)\hookrightarrow C^{k-1}(Q)$
is compact;
see e.g.~\cite[B.1.11]{MS}.
Hence there is a subsequence
which converges on $Q$ in the $C^k$
topology. We denote the limit by $u\in C^k(Q)$.
Since this is true for every $k\ge2$
there is a subsequence, still denoted by $u^\nu$,
converging on $Q$ to $u$, uniformly with all derivatives.
Since this is true for every compact
subcylinder $Q$ of $(-T,0]\times S^1$,
the theorem follows by choosing
a diagonal subsequence associated
to an exhausting sequence
by such $Q$'s.
Because, in particular, the convergence is in $C^0$
and the $u^\nu$ take values in $M$,
so does the limit $u$.
By $C^k$ convergence with $k\ge2$
the limit $u$ satisfies
the heat equation~(\ref{eq:heat-embedded}).
\end{proof}

\subsection{An apriori estimate}
\label{subsec:apriori}

\begin{theorem}
\label{thm:apriori-t}
Fix a perturbation $\Vv:\Ll M\to\R$ that 
satisfies~{\rm (V0)--(V1)} and a constant $c_0>0$. 
Then there is a constant
$C=C(c_0,\Vv)>0$ such that the following holds.
If $u:\R\times S^1\to M$ is a smooth solution
of~(\ref{eq:heat}) such that
\begin{equation}\label{eq:action-bound}
     \sup_{s\in\R}
     \Ss_\Vv(u(s,\cdot))\le c_0
\end{equation}
then
$\left\|\p_tu\right\|_\infty\le C$.
\end{theorem}

The proof of theorem~\ref{thm:apriori-t}
is based on the following mean value
inequality.
For $r>0$ define the 
{\bf open parabolic rectangle}
$P_r\subset \R^2$ by 
$$
     P_r:=(-r^2,0)\times (-r,r).
$$

\begin{lemma}[{\cite[lemma~B.1]{SaJoa-LOOP}}]
\label{le:apriori-basic}
There is a constant $c_1>0$
such that the following holds 
for all $r\in(0,1]$ and $a\ge 0$. If
$w:P_r\to\R$, 
$(s,t)\mapsto w(s,t)$,
is $C^1$ in the $s$-variable 
and $C^2$ in the $t$-variable
such that
$$
     (\p_t\p_t-\p_s)w\ge -aw, \qquad w\ge0,
$$
then
$$
     w(0)\le\frac{c_1e^{ar^2}}{r^3}
     \int_{P_r} w.
$$
\end{lemma}

\begin{corollary}
\label{co:apriori-basic}
Fix two constants
$r\in(0,1]$ and $\mu\ge 0$.
Let $c_1$ be the constant
of lemma~\ref{le:apriori-basic}.
If $F:[-r^2,0]\to\R$
is a $C^2$ function satisfying
$$
     -F^\prime+\mu F\ge 0, \qquad F\ge0,
$$
then
$$
     F(0)\le\frac{2c_1e^{\mu r^2}}{r^2}
     \int_{-r^2}^0 F(s)\: ds.
$$
\end{corollary}

\begin{proof}
This follows immediately
from lemma~\ref{le:apriori-basic}
with $w(s,t):=f(s)$.
\end{proof}

\begin{proof}
[Proof of theorem~\ref{thm:apriori-t}]
The idea is to first derive
slicewise $L^2$ bounds, then
verify the differential
inequality in lemma~\ref{le:apriori-basic}
and apply the lemma
using the slicewise bounds
on the right hand side.
The slicewise bound
for $\p_t u$ follows easily
from the assumption
$$
     c_0
     \ge \Ss_\Vv(u_s)
     =\frac{1}{2}\norm{\p_t u_s}_{L^2(S^1)}^2
     -\Vv(u_s)
$$
where $u_s(t):=u(s,t)$.
Let $C_0$ denote the constant in~(V0),
then this implies
\begin{equation}\label{eq:slice-p_tu}
     \norm{\p_t u_s}_{L^2(S^1)}^2
     \le 2c_0+2\Vv(u_s)
     \le 2c_0 + 2C_0
\end{equation}
for every $s\in\R$.
Consider the pointwise
differential inequality
given by
\begin{equation*}
\begin{split}
\left(\p_t\p_t-\p_s\right)\Abs{\p_t u}^2
&=2\Abs{\Nabla{t}\p_t u}^2
     +2\inner{(\Nabla{t}\Nabla{t}-\Nabla{s})
      \p_t u}{\p_t u} \\
&=2\Abs{\Nabla{t}\p_t u}^2
     -2\inner{\Nabla{t} \grad \Vv(u)}
     {\p_t u} \\
&\ge -2C_1\left(1+\Abs{\p_t u}\right) 
     \Abs{\p_t u} \\
&\ge -C_1-3C_1\Abs{\p_t u}^2.
\end{split}
\end{equation*}
To obtain the second
step we replaced $\Nabla{t}\p_t u$
according to the
heat equation~(\ref{eq:heat})
and used the fact that
$\Nabla{t}\p_s u=\Nabla{s}\p_t u$.
The third step is by
condition~(V1) with constant $C_1$.
Choose $(s_0,t_0)\in \R\times S^1$ and
apply lemma~\ref{le:apriori-basic}
in the case $r=1$ and with
$$
     w(s,t):=\frac{1}{3}+\abs{\p_t u(s_0+s,t_0+t)}^2
$$
and $a=3C_1$ to obtain
\begin{equation*}
\begin{split}
     w(0)
    &\le c_1e^a\int_{-1}^0\int_{-1}^{+1}\left(
     \frac{1}{3}
     +\Abs{\p_tu(s_0+s,t_0+t)}^2\right)dtds\\
    &= c_1e^{3C_1}\left(
     \frac{2}{3}
     +2\int_{-1}^0
     \Norm{\p_tu_{s_0+s}}_{L^2(S^1)}^2
     ds \right).
\end{split}
\end{equation*}
Theorem~\ref{thm:apriori-t}
then follows from the slicewise
estimate~(\ref{eq:slice-p_tu}).
\end{proof}

\begin{lemma}\label{le:energy-bound}
Fix a constant $c>0$
and a perturbation $\Vv:\Ll M\to\R$ that 
satisfies~{\rm (V0)} with constant $C>0$.
If $u:\R\times S^1\to M$ is a solution
of~(\ref{eq:heat}) then
$$
     \sup_{s\in\R}\Ss_\Vv(u(s,\cdot))
     \le c
     \quad
     \Rightarrow\quad
     E(u)\le c+C.
$$
\end{lemma}

\begin{proof}
Let $u_s(t):=u(s,t)$ and
choose $T>0$, then
\begin{equation*}
\begin{split}
     E_{[-T,T]}(u)
    &=\int_{-T}^T\int_0^1
     \Abs{\p_s u(s,t}^2 \: dt ds\\
    &=-\int_{-T}^T\langle
     \nabla \Ss_\Vv(u_s),\p_su_s
     \rangle_{L^2} ds\\
    &=-\int_{-T}^T\frac{d}{ds}\Ss_\Vv(u_s)\: ds\\
    &=\Ss_\Vv(u_{-T})-\Ss_\Vv(u_T).
\end{split}
\end{equation*}
Here we used the fact
that the heat equation~(\ref{eq:heat})
is the negative $L^2$ gradient
flow equation for the action functional.
Now the crucial property
of the action functional is
its boundedness from
below, namely
$\Ss_\Vv(x)\ge -C$
for every $x\in\Ll M$
by~(V0). Hence
$\Ss_\Vv(u_{-T})-\Ss_\Vv(u_T)
\le c+C$
and this proves the lemma.
\end{proof}

\subsection{Gradient bounds}
\label{subsec:gradient}

\begin{theorem}
\label{thm:gradient}
Fix a perturbation $\Vv:\Ll M\to\R$ that 
satisfies~{\rm (V0)--(V2)} and a constant $c_0>0$. 
Then there is a constant
$C=C(c_0,\Vv)>0$ such that the following holds.
If $u:\R\times S^1\to M$ 
is a smooth solution
of~(\ref{eq:heat}) that
satisfies~(\ref{eq:action-bound}),
i.e. $\sup_{s\in\R}\Ss_\Vv(u(s,\cdot))\le c_0$,
then
\begin{equation*}
\begin{gathered}
     \Abs{\p_s u(s,t)}^2
     +\Abs{\Nabla{t}\p_s u(s,t)}^2
     \le C E_{[s-1,s]}(u)
     \\
     \Abs{\Nabla{s}\p_s u(s,t)}^2
     +\Abs{\Nabla{t}\Nabla{t}\p_s u(s,t)}^2
     \le C E_{[s-2,s]}(u)
\end{gathered}
\end{equation*}
for every $(s,t)\in\R\times S^1$.
Here $E_I(u)$ denotes the energy
of $u$ over the set $I\times S^1$.
\end{theorem}

\begin{proof}
By theorem~\ref{thm:apriori-t}
there is a constant $C_0=C_0(c_0,\Vv)>0$
such that
$$
     \Norm{\p_tu}_\infty
     \le C_0.
$$
Let $C=C(C_0,\Vv)$ be the constant
of theorem~\ref{thm:kerD-apriori}
with this choice of $C_0$.
Observe that $\xi:=\p_su$
solves the linearized heat equation.
Hence theorem~\ref{thm:kerD-apriori}
shows that
$$
     \Abs{\p_s u(s,t)}^2
     \le C^2 E_{[s-1,s]}(u)
     \le C^2(c_0+c^\prime)
$$
for every $(s,t)\in\R\times S^1$.
Here the last step
is by lemma~\ref{le:energy-bound}
and axiom~(V0) with constant $c^\prime$.
Use that $u$ solves~(\ref{eq:heat})
and satisfies axiom~(V0)
to obtain that
$$
     \Norm{\Nabla{t}\p_tu}_\infty
     \le \Norm{\p_su}_\infty
     +\Norm{\grad\Vv(u)}_\infty
     \le C\sqrt{c_0+c^\prime} + c^\prime.
$$
Now choose $C_0$
larger than 
$2C\sqrt{c_0+c^\prime} + c^\prime$
and let $C=C(C_0,\Vv)$ 
be the constant
of theorem~\ref{thm:kerD-apriori}
with this new choice of $C_0$.
Theorem~\ref{thm:kerD-apriori}
then proves the desired
estimate for $\abs{\Nabla{t}\p_su}$.
It follows that
$\norm{\Nabla{t}\p_su}_\infty$
is bounded. Therefore
$\norm{\Nabla{t}\Nabla{t}\p_tu}_\infty$
is bounded by~(\ref{eq:heat})
and axiom~(V1).
Hence
theorem~\ref{thm:kerD-apriori-II}
applies with a new choice of $C_0$
and proves the
remaining two estimates
of theorem~\ref{thm:gradient}.
\end{proof}

\begin{proof}
[Proof of theorem~\ref{thm:apriori}]
Theorem~\ref{thm:apriori-t},
theorem~\ref{thm:gradient}
and lemma~\ref{le:energy-bound}.
Only~{\rm (V0)--(V1)} are used.
Use~(\ref{eq:heat})
and~(V0) to obtain the
estimate for $\Nabla{t}\p_tu$.
\end{proof}

\subsection{Exponential decay}
\label{subsec:exp-decay}

\begin{theorem}\label{thm:exp-decay}
Fix a perturbation $\Vv:\Ll M\to\R$ 
that satisfies~{\rm (V0)--(V2)}.
Suppose $\Ss_\Vv$ is Morse and let $a\in\R$ be 
a regular value of $\Ss_\Vv$.
Then there exist constants
$\delta,c,\rho>0$ 
such that the following holds.
If $u:\R\times S^1\to M$ 
is a smooth solution
of~(\ref{eq:heat}) that
satisfies~(\ref{eq:action-bound}),
i.e. $\sup_{s\in\R}\Ss_\Vv(u(s,\cdot))\le a$,
and
\begin{equation}\label{eq:small-energy}
     E_{\R\setminus[-T_0,T_0]}(u)
     <\delta
\end{equation}
for some $T_0>0$, then
$$
     E_{\R\setminus[-T,T]}(u)
     \le ce^{-\rho(T-T_0)}
     E_{\R\setminus[-T_0,T_0]}(u)
$$
for every $T\ge T_0+1$.
\end{theorem}

\begin{corollary}\label{cor:exp-decay}
Fix a perturbation $\Vv:\Ll M\to\R$ 
that satisfies~{\rm (V0)--(V2)}.
Suppose $\Ss_\Vv$ is Morse and let
$x^\pm\in\Pp(\Vv)$.
Then there exist constants
$\delta,c,\rho>0$ 
such that the following holds.
Suppose that $u\in\Mm(x^-,x^+;\Vv)$
satisfies~(\ref{eq:small-energy})
for some $T_0>0$.
Then
$$
     \Abs{\p_su(s,t)}^2
     +\Abs{\Nabla{t}\p_su(s,t)}^2
     \le ce^{-\rho(s-T_0)}
     E_{\R\setminus[-T_0,T_0]}(u)
$$
for every $s\ge T_0+2$.
\end{corollary}

\begin{proof}
Theorem~\ref{thm:gradient} and
theorem~\ref{thm:exp-decay}.
\end{proof}

The proof of theorem~\ref{thm:exp-decay}
is based on the following lemma
which asserts existence
of a true critical point nearby
an approximate one.

\begin{lemma}[Critical point nearby 
approximate one]
\label{le:nearby-Crit}
Fix a perturbation $\Vv:\Ll M\to\R$ 
that satisfies~{\rm (V0)}
and let $a\in\R$ be
a regular value of $\Ss_\Vv$. 
Then, for every $\delta_0>0$,
there is a constant $\delta_1>0$ 
such that the following is true.
Suppose $x:S^1\to M$ is
a smooth loop such that
$$
     \Ss_\Vv(x)\le a,\qquad
     \Norm{\Nabla{t}\p_t x+\grad\Vv(x)}_\infty
     <\delta_1.
$$
Then there is a critical point
$x_0\in\Pp^a(\Vv)$ and a vector field
$\xi_0$ along $x_0$ such that
$
    x=\exp_{x_0}(\xi_0)
$
and
$$
    \Norm{\xi_0}_\infty 
    +\Norm{\Nabla{t}\xi_0}_\infty 
    +\Norm{\Nabla{t}\Nabla{t}\xi_0}_\infty
    \le\delta_0.
$$
\end{lemma}

\begin{proof}
First note that 
$$
     \Norm{\p_t x}_2^2
     =\int_0^1\left|\p_t x(t)\right|^2 dt
     =2\Ss_\Vv(x)+2\Vv(x)
     \le 2(a+C)
$$
where $C$ is the constant in~(V0).
Now, assuming $\delta_1\le 1$, we have
\begin{equation*}
\begin{split}
     \left|\frac{d}{dt}\left|\p_t x\right|^2\right|
    &=2\bigl|\inner{\p_t x}
     {\Nabla{t}\p_t x+\grad\Vv(x)}
     -\inner{\p_t x}{\grad\Vv(x)}\bigr| \\
    &\le 2\left(\delta_1+C\right)\Abs{\p_t x}
     \le \left(1+C\right)^2 
     + \Abs{\p_t x}^2.
\end{split}
\end{equation*}
Integrate this inequality to obtain that
$$
     \Abs{\p_t x(t_1)}^2-\Abs{\p_t x(t_0)}^2
     \le \left(1+C\right)^2 + \Norm{\p_t x}_2^2
$$
for $t_0,t_1\in[0,1]$.
Integrating again over the interval 
$0\le t_0\le 1$ gives
\begin{equation}\label{eq:yc}
     \Norm{\p_t x}_\infty
     \le \sqrt{\left(1+C\right)^2 
     + 2\Norm{\p_t x}_2^2}
     \le c
\end{equation}
where $c^2:=\left(1+C\right)^2+4\left(a+C\right)$. 

Now suppose that the assertion is wrong. 
Then there is a constant $\delta_0>0$ 
and a sequence of smooth 
loops $x_\nu:S^1\to M$ satisfying 
$$
     \Ss_\Vv(x_\nu)\le a,\qquad
     \lim_{\nu\to\infty}\bigl(
     \Norm{\Nabla{t}\p_t x_\nu
     +\grad\Vv(x_\nu)}_\infty
     \bigr) = 0,
$$
but not the conclusion of 
the lemma for the given constant 
$\delta_0$.
By~(V0) we have
$\sup_\nu\Norm{\Nabla{t}\p_t x_\nu}_\infty<\infty$
and~(\ref{eq:yc}) implies
$\sup_\nu\Norm{\p_t x_\nu}_\infty<\infty$.
Hence, by the Arzela--Ascoli theorem, 
there exists a  subsequence,
still denoted by $x_\nu$, that 
converges in the $C^1$-topology.
Let $x_0\in C^1(S^1,M)$ be the limit.
We claim that this subsequence 
actually converges in the 
$C^2$-topology.
Then $\Nabla{t}\p_t x_0+\grad\Vv(x_0)=0$.
Hence $x_0\in\Pp^a(\Vv)$ and
$x_\nu$ converges
to $x_0$ in the $C^2$-topology.  
This contradicts our assumption on the sequence 
$x_\nu$ and proves the lemma.

It remains to prove the claim.
For simplicity
let us assume that $M$ is
isometrically embedded in
Euclidean space $\R^N$
for some sufficiently large integer $N$.
Since 
$\sup_\nu\Norm{\Nabla{t}\p_t x_\nu}_2<\infty$,
the Banach-Alaoglu Theorem
asserts existence of a subsequence,
still denoted by $x_\nu$,
and an element
$v\in L^2$
such that $\Nabla{t}\p_t x_\nu$
converges to $v$ weakly in $L^2$.
In fact $v$ equals the weak
$t$-derivative of $\p_t x$.
Now $\grad\Vv(x_\nu)$ converges
to $\grad\Vv(x_0)$ in $L^2$
and to $-v$ weakly in $L^2$.
But the weak limit 
equals the strong limit, hence
$v=-\grad\Vv(x_0)\in C^1$.
Therefore $\p_t x_0\in C^1$ and
$\Nabla{t}\p_t x_0$ equals
the weak $t$-derivative $v$
of $\p_t x_0$.
Now $x_0\in C^2$ satisfies
\begin{equation}\label{eq:lim-crit}
     \Nabla{t}\p_t x_0+\grad\Vv(x_0)
     =0,
\end{equation}
because $\Nabla{t}\p_t x_\nu$
converges to $v=\Nabla{t}\p_t x_0$
weakly in $L^2$
and to $-\grad\Vv(x_0)$ strongly in $L^2$.
By induction~(\ref{eq:lim-crit})
implies that $x_0\in C^\infty$.
Moreover, it follows using~(\ref{eq:lim-crit})
that $\Nabla{t}\p_t x_\nu$
converges to $\Nabla{t}\p_t x_0$
in $C^0$ and this proves the claim.
\end{proof}

\begin{proof}[Proof of theorem~\ref{thm:exp-decay}]
Given $a$ and $\Vv$,
let $C=C(a,\Vv)$ be the constant of
theorem~\ref{thm:apriori}
and theorem~\ref{thm:gradient}
with this choice.
Let $C_0>1$ be the constant in~(V0). Then
$E(u)\le a+C_0$ by
lemma~\ref{le:energy-bound} and
$
     \norm{\p_su}_\infty
     \le C E(u)
     \le C(a+C_0)
$
by theorem~\ref{thm:gradient}.
Hence
$$
     \Norm{\p_tu}_\infty
     +\Norm{\Nabla{t}\p_tu}_\infty
     \le c_0
$$
by theorem~\ref{thm:apriori}
and by replacing $\Nabla{t}\p_tu$
according to the heat
equation~(\ref{eq:heat}).
Here
$c_0=C(a+2C_0)+C_0$.
Let $\delta_0$ and $\rho_0$
be the positive constants of
theorem~\ref{thm:kerD-exp-decay}
with this choice of $c_0$.
Choose $\delta_0$
smaller than one quarter
the minimal $C^0$
distance $\kappa=\kappa(a)$ of any
two elements of $\Pp^a(\Vv)$.
Let $\delta_1>0$
be the constant of
lemma~\ref{le:nearby-Crit}
associated to $a$ and $\delta_0$
and set
$$
     \delta
     :=
     \min\left\{\frac{\delta_0^2}{4C},
     \frac{\delta_1^2}{4C}\right\}.
$$
Note that $\delta_0$, $\rho_0$, 
$\delta_1$, and $\delta$
depend only on $a$, $\Vv$, and 
the constant $C_0$ of axiom~(V0).
Note furthermore that
$\xi:=\p_su$
solves the linear
heat equation~(\ref{eq:kerD-heat})
and that the continuous function
$s\mapsto\norm{\p_su_s}_{L^2(S^1)}$
is bounded, because
its integral over $\R$
is the energy $E(u)$
which is finite.

If $\abs{s}\ge T_0+1$, then
$E_{[s-1,s]}(u)\le
E_{\R\setminus [-T_0,T_0]}(u)$
and it follows that
\begin{equation}\label{eq:delta1}
     \Norm{\p_su_s}_\infty
     +\Norm{\Nabla{t}\p_su_s}_\infty
     \le \sqrt{CE_{[s-1,s]}(u)}
     \le \sqrt{C\delta}
     <\min\left\{\delta_0,
     \delta_1\right\}.
\end{equation}
Here we used
theorem~\ref{thm:gradient}
in step one,
assumption~(\ref{eq:small-energy})
in step two,
and the definition of $\delta$
in the last step.
Hence, by 
lemma~\ref{le:nearby-Crit},
there are critical points
$x^\pm\in\Pp^a(\Vv)$
such that
$$
     u_s=\exp_{x^\pm}(\eta^\pm_s),\qquad
     \norm{\eta_s}_{C^2(S^1)}\le\delta_0
$$
whenever $\pm s\ge T_0+1$.
Although the critical points
$x^\pm$ apriori depend on $s$
they are in fact independent,
because $\delta_0<\kappa/4$
and $\Pp^a(\Vv)$ is a finite set by the
Morse condition.
Moreover, injectivity of the operators
$A_{x^\pm}$ is equivalent
to nondegeneracy
of the critical points $x^\pm$
and this is true again by
the Morse condition.
Now
theorem~\ref{thm:kerD-exp-decay}
and remark~\ref{rmk:decay-forward}
conclude the
proof of 
theorem~\ref{thm:exp-decay}.
\end{proof}

\begin{proof}[Proof of
 theorem~\ref{thm:par-exp-decay}]
We prove exponential decay
in three steps.

I) Firstly, the energy of $u$ is finite.
In the case (B) this is part
of the assumptions. 
In the case (F)
it follows as in the
proof of lemma~\ref{le:energy-bound}
for $u:[0,\infty)\times S^1\to\R$.
Namely, let $C_0>0$ be the 
constant in~(V0)
and set $u_0(t):=u(0,t)$, then
$
     E(u)\le \Ss_\Vv(u_0)+C_0
$.

II) Secondly, we establish
the existence of asymptotic
limits. Consider the forward case (F).
We claim that
$\p_su(s,t)\to0$ as $s\to\infty$,
uniformly in $t$.
Let $C>0$ be the constant
in theorem~\ref{thm:gradient}
and let $s\ge1$, then
$$
     \abs{\p_su(s,t)}
     \le C E_{[s-1,s]}(u)
     =C\int_{s-1}^s
     \norm{\p_su_\sigma}^2_{L^2(S^1)} 
     d\sigma
     \stackrel{s\to\infty}{\longrightarrow} 0.
$$
Here the last step follows by
finite energy of $u$ and
this proves the claim.
Because $\p_su_s$
converges to zero in
$L^\infty(S^1)$ so does
$\Nabla{t}\p_tu_s+\grad \Vv(u_s)$
by~(\ref{eq:heat}).
Hence it follows from 
lemma~\ref{le:nearby-Crit}
that there is a critical point
$x^+\in\Pp(\Vv)$ and,
for every sufficiently large $s$,
there is a smooth vector field
$\xi_s$ along $x^+$
such that
$$
     u_s=\exp_{x^+} (\xi_s),
     \qquad
     \norm{\xi_s}_\infty
     +\norm{\Nabla{t}\xi_s}_\infty
     +\norm{\Nabla{t}\Nabla{t}\xi_s}_\infty
     \stackrel{s\to\infty}{\longrightarrow} 0.
$$
(Here we used the fact that --
since $\Ss_\Vv$ is Morse --
there are only finitely many elements 
in $\Pp(\Vv)$ below any fixed action level.)
This and
the identities for the maps $E_{ij}$ 
in~(\ref{eq:exponential-identity})
imply that
\begin{equation}\label{eq:Linfty-bounds}
     \norm{\p_su}_\infty
     +\norm{\p_tu}_\infty
     +\norm{\Nabla{t}\p_tu}_\infty
     <\infty.
\end{equation}
The same arguments apply in
case (B) with corresponding
asymptotic limit $x^-$.

III) The third step
is to prove exponential
decay of the $C^k$ norm of $\p_su$.
Consider the forward case (F).
We prove by induction that
for every $k\in\N$
there is a constant
$c_k^\prime>0$
such that
$$
     \Norm{\p_su}_{W^{k,2}
     ([s,\infty)\times S^1)}
     \le c_k^\prime \Norm{\p_su}
     _{L^2([s-k,\infty)\times S^1)}
$$
for every $s\ge k$.
This estimate,
the definition of the energy 
in~(\ref{eq:par-energy}), and
theorem~\ref{thm:exp-decay} 
with constants
$\delta,c,\rho,T_0>0$, where $T_0$
is chosen sufficiently large such
that~(\ref{eq:small-energy})
holds true, then show that
$$
     \Norm{\p_su}_{W^{k,2}
     ([s,\infty)\times S^1)}
     \le c_k^\prime
     \sqrt{E_{[s-k,\infty]}(u)}
     \le c_k^\prime
     \sqrt{c\delta} 
     e^{-\rho(s-k-T_0)/2}
$$
whenever $s\ge k+T_0+1$.
The Sobolev embedding
$W^{k,2}\hookrightarrow C^{k-2}$,
e.g.~on the compact set
$[s,s+1]\times S^1$,
concludes the proof of forward
exponential decay~(F).

It remains to carry out the
induction argument.
It is based on the identity
\begin{equation}\label{eq:p_su-in-ker}
     \left(\Nabla{s}
     -\Nabla{t}\Nabla{t}\right)\p_su
     =R(\p_su,\p_tu)\p_tu+\Hh_\Vv(u)\p_su
\end{equation}
-- which follows by
linearizing the heat
equation~(\ref{eq:heat})
in the $s$-direction
to obtain that
$\p_su \in \ker\, \Dd_u$
in the notation of
section~\ref{sec:linearized} --
and the subsequent estimate.
Proposition~\ref{prop:par-linear}
with $p=2$
applies\footnote{Formally
add to $u$
\emph{any} smooth half cylinder imposing
a uniform limit as $s\to-\infty$.}
by~(\ref{eq:Linfty-bounds})
and shows that there
is a constant $c^\prime>0$
with the following significance.
If $s_0\ge1$ then
\begin{equation}\label{eq:half-cylinder}
\begin{split}
    &\Norm{\Nabla{s}\xi}
     _{L^2([s_0,\infty)\times S^1)}
     +\Norm{\Nabla{t}\xi}
     _{L^2([s_0,\infty)\times S^1)}
     +\Norm{\Nabla{t}\Nabla{t}\xi}
     _{L^2([s_0,\infty)\times S^1)}\\
    &\le c^\prime \left(
    \Norm{\Nabla{s}\xi-\Nabla{t}\Nabla{t}\xi}
     _{L^2([s_0-1,\infty)\times S^1)}
     + \Norm{\xi}
     _{L^2([s_0-1,\infty)\times S^1)}
     \right)
\end{split}
\end{equation}
for every $\xi\in\Omega^0([0,\infty)
\times S^1)$ of compact support.
To see this fix a smooth
nondecreasing cutoff function
$\beta:\R\to[0,1]$
which equals zero for $s\le s_0-1$
and one for $s\ge s_0$
and whose slope is at most two.
Via extension by zero
we interpret $\beta\xi$
as a smooth compactly supported vector field
along the extended cylinder $u:\R\times S^1\to M$.
Now proposition~\ref{prop:par-linear}
applies to $\beta \xi$ and
proves~(\ref{eq:half-cylinder}).
Note that $c^\prime$
depends on the $L^\infty$
norms of $\p_s\beta$, $\p_t\beta$,
and $\p_t\p_t\beta$.
We also used lemma~\ref{le:plus-minus}
to deal with the term $\Nabla{t}\xi$ 
which appears on the right hand side.

We prove the induction
hypothesis in the case $k=1$.
Let $s\ge1$
and denote by $C_1>0$
the constant in~(V1).
By~(\ref{eq:half-cylinder})
with $\xi=\p_su$ and~(\ref{eq:p_su-in-ker})
it follows that
\begin{equation*}
\begin{split}
    &\Norm{\Nabla{s}\p_su}
     _{L^2([s,\infty)\times S^1)}
     +\Norm{\Nabla{t}\p_su}
     _{L^2([s,\infty)\times S^1)}
     +\Norm{\Nabla{t}\Nabla{t}\p_su}
     _{L^2([s,\infty)\times S^1)}\\
    &\le c^\prime \left(
    \Norm{(\Nabla{s}-\Nabla{t}\Nabla{t})\p_su}
     _{L^2([s-1,\infty)\times S^1)}
     + \Norm{\p_su}
     _{L^2([s-1,\infty)\times S^1)}
     \right)\\
    &= c^\prime \left(
    \Norm{R(\p_su,\p_tu)\p_tu+\Hh_\Vv(u)\p_su}
     _{L^2([s-1,\infty)\times S^1)}
     + \Norm{\p_su}
     _{L^2([s-1,\infty)\times S^1)}
     \right)\\
    &\le c^\prime \left(
    \norm{R}_\infty
    \norm{\p_tu}_\infty^2
    +2C_1+1\right)
     \Norm{\p_su}
     _{L^2([s-1,\infty)\times S^1)}.
\end{split}
\end{equation*}

We prove the induction
hypothesis for $k=2$.
Assume $s\ge 2$.
Then by~(\ref{eq:half-cylinder})
with $\xi=\Nabla{s}\p_su$
and~(\ref{eq:p_su-in-ker})
it follows that
\begin{equation*}
\begin{split}
    &\Norm{\Nabla{s}\Nabla{s}\p_su}
     _{L^2([s,\infty)\times S^1)}
     +\Norm{\Nabla{t}\Nabla{s}\p_su}
     _{L^2([s,\infty)\times S^1)}
     +\Norm{\Nabla{t}\Nabla{t}\Nabla{s}\p_su}
     _{L^2([s,\infty)\times S^1)}\\
    &\le c^\prime \Bigl(
     \Norm{\Nabla{s}\left( R(\p_su,\p_tu)\p_tu
     +\Hh_\Vv(u)\p_su\right)
     +[\Nabla{s},\Nabla{t}\Nabla{t}]\p_su}
     _{L^2([s-1,\infty)\times S^1)}\\
    &\quad + \Norm{\Nabla{s}\p_su}
     _{L^2([s-1,\infty)\times S^1)}
     \Bigr).
\end{split}
\end{equation*}
Now use $s\ge 2$,
the apriori estimates~(\ref{eq:Linfty-bounds}),
axiom~(V2),
and the case $k=1$
to bound the right hand side
by a constant times $\norm{\p_su}
_{L^2([s-2,\infty)\times S^1)}$.
An $L^2$ bound for
$\Nabla{t}\Nabla{t}\p_su$
was obtained earlier in the case $k=1$
and the identity
$$
     \Nabla{s}\Nabla{t}\p_su
     =\Nabla{t}\Nabla{s}\p_su
     -R(\p_tu,\p_su)\p_su
$$
implies one for
$\Nabla{s}\Nabla{t}\p_su$.

Proving the induction
hypothesis in the case $k=3$
requires additional information:
Theorem~\ref{thm:apriori-t}
and theorem~\ref{thm:gradient}
only assume an upper action bound
for the heat flow solution.
In the case at hand
this is provided
by $\Ss_\Vv(u(0,\cdot))$.
This reproves~(\ref{eq:Linfty-bounds})
and in addition shows that
$
     \norm{\Nabla{t}\p_su}_\infty<\infty
$.
This estimate is crucial,
since~(\ref{eq:half-cylinder})
with $\xi=\Nabla{s}\Nabla{s}\p_su$
and~(\ref{eq:p_su-in-ker})
lead to terms of the form
$$
     \norm{R(\Nabla{s}\p_su,
     \Nabla{t}\p_su)\p_tu}
     _{L^2([s,\infty)\times S^1)},
$$
but our induction hypothesis
in the case $k=2$
only provides a $C^0$
bound for $\p_su$.
The remaining part of
proof follows the same
pattern as in the case $k=2$.
Here we use axiom~(V3).

Fix an integer $k\ge3$ and
assume the induction
hypothesis is true for
every $\ell\in\{1,\dots,k\}$.
In particular, we have
$W^{k,2}$ and $C^{k-2}$
bounds for $\p_su$
on the appropriate domains.
Apply~(\ref{eq:half-cylinder})
with $\xi={\Nabla{s}}^k\p_su$
and~(\ref{eq:p_su-in-ker})
to obtain $L^2$ bounds
for ${\Nabla{s}}^{k+1}\p_su$ and
$\Nabla{t}{\Nabla{s}}^k\p_su$.
Here we use axiom~(V3)
and the induction hypothesis for
$\ell\in\{1,\dots,k\}$.
A problem of the type
encountered in the case $k=3$ does not
arise, since we have $C^{k-2}$ bounds
for $\p_su$ with $k\ge3$.
To obtain $L^2$ estimates
for the  remaining terms of the form
${\Nabla{t}}^j{\Nabla{s}}^{k-j}\p_su$ with
$j\ge2$ use~(\ref{eq:p_su-in-ker})
to treat any $\Nabla{t}\Nabla{t}$ for
one $\Nabla{s}$.
This reduces the order of the term,
hence the induction hypothesis
can be applied.
This completes the induction step
and proves~(F).
The backward case (B) follows similarly.
This proves
theorem~\ref{thm:par-exp-decay}.
\end{proof}

\begin{lemma}\label{le:fredholm}
Fix a perturbation $\Vv:\Ll M\to\R$
that satisfies~{\rm (V0)--(V3)},
a constant $p>1$,
and nondegenerate critical points
$x^\pm$ of $\Ss_\Vv$.
If $u\in\Mm(x^-;x^+;\Vv)$, 
then the operators 
$\Dd_u,\Dd_u^*:\Ww^{1,p}_u\to\Ll^p_u$
are Fredholm and 
$$
     \INDEX\,\Dd_u 
     =\IND_\Vv(x^-)-\IND_\Vv(x^+)
     =-\INDEX\,\Dd_u^*.
$$
\end{lemma}

\begin{proof}
By theorem~\ref{thm:par-exp-decay}
on exponential decay
$u$ satisfies the assumptions
of the Fredholm 
theorem~\ref{thm:fredholm}.
\end{proof}

\subsection{Compactness up to 
broken trajectories}
\label{subsec:compactness}

\begin{proposition}[Convergence on compact sets]
\label{prop:unifconvergence-compactsets}
Assume that the perturbation
$\Vv:\Ll M\to\R$
satisfies~{\rm (V0)--(V3)}
and that $\Ss_\Vv$ is Morse.
Fix critical points
$x^\pm\in\Pp(\Vv)$ and
a sequence of connecting
trajectories $u^\nu\in\Mm(x^-,x^+;\Vv)$.
Then there is a pair
$x_0,x_1\in\Pp(\Vv)$,
a connecting trajectory
$u\in\Mm(x_0,x_1;\Vv)$,
and a subsequence,
still denoted by $u^\nu$,
such that
the following hold:
\begin{enumerate}
\item[\rm\bfseries(i)]
  The subsequence $u^\nu$
  converges to $u$,
  uniformly with all derivatives
  on every compact subset
  of $\R\times S^1$.
\item[\rm\bfseries(ii)]
  For all $s\in\R$
  and $T>0$
  \begin{align*}
     \Ss_\Vv\bigl(u(s,\cdot)\bigr)
     =\lim_{\nu\to\infty}
     \Ss_\Vv\bigl(u^\nu(s,\cdot)\bigr)\\
     E_{[-T,T]}(u)
     =\lim_{\nu\to\infty}
     E_{[-T,T]}(u^\nu).
  \end{align*}
\end{enumerate}
\end{proposition}

\begin{proof}
Since the flow lines $u^\nu$
connect $x^-$ to $x^+$
and the action $\Ss_\Vv$ decreases
along flow lines,
it follows that
$$
     \sup_{s\in\R}\Ss_\Vv(u^\nu(s,\cdot))
     =\Ss_\Vv(x^-)=:c_0.
$$
Hence by the apriori estimates
theorem~\ref{thm:apriori-t}
and theorem~\ref{thm:gradient}
there is a constant $C=C(c_0,\Vv)$
such that
$$
     \Abs{\p_tu^\nu(s,t)}\le C,
$$
and
$$
     \Abs{\p_su^\nu(s,t)}
     \le C\sqrt{\Ss_\Vv(x^-)-\Ss_\Vv(x^+)},
$$
for every $(s,t)\in\R\times S^1$.
To obtain the second estimate
we used the energy identity~(\ref{eq:par-energy})
for connecting orbits.
Now fix a constant $p>2$
and pick an integer $\ell\ge2$.
Then the assumptions of
theorem~\ref{thm:compactness-gradbound}
are satisfied for the sequence
$u^\nu$ restricted to the cylinder 
$Z_\ell=(-\ell,\ell]\times S^1$.
Hence there is a smooth solution $u:Z_\ell\to M$
of the heat equation~(\ref{eq:heat})
and a subsequence, still denoted by $u^\nu$,
such that $u^\nu$ converges to $u$,
uniformly with all derivatives
on the compact subset
$[-\ell+1,\ell]\times S^1$ of $Z_\ell$.
Now~(i) follows by choosing
a diagonal subsequence
associated to the exhausting sequence
$Z_2\subset Z_3\subset \dots$
of $\R\times S^1$.

To prove~(ii) note that
\begin{equation*}
\begin{split}
     E_{[-T,T]}(u)
    &=\lim_{\nu\to\infty}
     \int_{-T}^T\int_0^1\Abs{\p_su^\nu}^2\,dt\,ds\\
    &=\lim_{\nu\to\infty} E_{[-T,T]}(u^\nu)\\
    &\le \Ss_\Vv(x^-)-\Ss_\Vv(x^+)
\end{split}
\end{equation*}
for every $T>0$. Here the first step
uses that, by~(i), the sequence
$\p_su^\nu$ converges to $\p_su$,
uniformly on compact sets.
The second step is by definition of the energy
and the last step is again by the 
energy identity~(\ref{eq:par-energy}).
Hence the limit $u:\R\times S^1\to M$
has finite energy and so, by
theorem~\ref{thm:par-exp-decay},
belongs to the moduli space
$\Mm(x_0,x_1;\Vv)$ for some $x_0,x_1\in\Pp(\Vv)$.
To prove convergence of
the action at time $s$ note that
$$
     \Vv\bigl(u(s,\cdot)\bigr)
     =\lim_{\nu\to\infty}\Vv\bigl(u^\nu(s,\cdot)\bigr),
$$
because $\Vv$ is continuous
with respect to the $C^0$
topology on $\Ll M$ by axiom~(V0).
Convergence of
the action at time $s$
then follows from the fact that
$\p_tu^\nu(s,\cdot)$ converges
to $\p_tu(s,\cdot)$ in
$L^\infty(S^1)$.
\end{proof}

\begin{lemma}[Compactness up to broken trajectories]
\label{le:broken}
Assume that the perturbation
$\Vv:\Ll M\to\R$
satisfies~{\rm (V0)--(V3)}
and that $\Ss_\Vv$ is Morse.
Fix distinct critical points
$x^\pm\in\Pp(\Vv)$ 
and let $u^\nu\in\Mm(x^-,x^+;\Vv)$
be a sequence of connecting
trajectories.
Then there exist a subsequence, 
still denoted by $u^\nu$,
finitely many critical points
$x_0$,\ldots,$x_m$
with $x_0=x^+$ and $x_m=x^-$,
finitely many solutions
$$
     u_k\in\Mm(x_k,x_{k-1};\Vv),
     \qquad \p_su_k\not\equiv 0,
     \qquad k=1,\ldots,m,
$$
and finitely many sequences $s_k^\nu$,
such that the shifted sequence
$u^\nu(s_k^\nu+s,t)$
converges to $u_k(s,t)$,
uniformly with all derivatives
on every compact subset
of $\R\times S^1$.
Moreover, these limit solutions
satisfy
$\sum_{k=1}^mE(u_k)
=\Ss_\Vv(x^-)-\Ss_\Vv(x^+)$.
\end{lemma}

\begin{proof}
In~\cite[Proof of lemma~10.3]{SaJoa-LOOP}
replace lemma~10.2 by
prop.~\ref{prop:unifconvergence-compactsets}.
\end{proof}

\section{The implicit function theorem}
\label{sec:IFT}

Throughout this section we
fix a smooth perturbation
$\Vv:\Ll M\to\R$ that
satisfies~{\rm (V0)--(V3)}
and two
nondegenerate critical points
$x^\pm$ of $\Ss_\Vv$.
The idea to prove
the manifold property
and the dimension formula in
theorem~\ref{thm:IFT}
is to construct
a smooth Banach manifold
which contains the moduli space
$\Mm(x^-,x^+;\Vv)$ and
then prove these statements
locally near each element of 
the moduli space.

Fix a real number
$p>2$ and denote
by
\begin{equation}\label{eq:B}
     \Bb^{1,p}
     =\Bb^{1,p}(x^-,x^+)
\end{equation}
the space of continuous
maps $u:\R\times S^1\to M$,
which satisfy the
limit conditions~(\ref{eq:limit}),
are locally of class $\Ww^{1,p}$,
and satisfy the asymptotic
conditions
$\xi^-\in\Ww^{1,p}
((-\infty,-T]\times S^1,u^*TM)$
and
$\xi^+\in\Ww^{1,p}
([T,\infty)\times S^1,u^*TM)$
for some sufficiently large $T>0$.
Here $\xi^\pm$ are defined pointwise
by the identity
$\exp_{x^\pm(t)}\xi^\pm(s,t)=u(s,t)$.
For $p>2$
the space $\Bb^{1,p}$ carries
the structure of a smooth 
infinite dimensional Banach manifold.
The tangent space $T_u\Bb^{1,p}$
is given by the Banach space $\Ww_u^{1,p}$
whose norm is defined in~(\ref{eq:norms}).
Around any \emph{smooth}  map $u$
local coordinates
are provided by
the inverse of the map
${\varphi_u}^{-1}:
V_u\to\Bb^{1,p}$
given by
$\xi\mapsto[(s,t)\mapsto \exp_{u(s,t)}\xi(s,t)]$
where $V_u\subset\Ww_u^{1,p}$
is a sufficiently small
neighborhood of zero.
By abuse of notation we shall denote
this map again by $\xi\mapsto \exp_u\xi$.
Moreover, note that if \emph{some} $u\in\Bb^{1,p}$
satisfies the heat equation~(\ref{eq:heat})
almost everywhere, then
$u$ is smooth by
theorem~\ref{thm:regularity},
hence $u\in\Mm(x^-,x^+;\Vv)$.

For $x\in M$ and $\xi\in T_xM$
denote parallel transport with
respect to the Levi-Civita
connection along the geodesic
$\tau\mapsto\exp_x(\tau\xi)$ by
$$
     \Phi(x,\xi):T_xM\to T_{\exp_x(\xi)}M.
$$
For $u\in\Bb^{1,p}$ the map
$\Ff_u:\Ww^{1,p}_u\to \Ll^p_u$
defined by
\begin{equation}\label{eq:Ff_u}
     \Ff_u(\xi)
     :=
     \Phi(u,\xi)^{-1}
     \left(
     \p_s(\exp_u\xi)
     -\Nabla{t}\p_t(\exp_u\xi)
     -\grad\Vv(\exp_u\xi)
     \right)
\end{equation}
is induced by pointwise
evaluation at $(s,t)$.
Its significance
lies in the following three facts.
Firstly, it is a smooth map
between Banach spaces,
hence the implicit
function theorem for
Banach spaces applies.
Secondly, the differential
$d\Ff_u(0):\Ww^{1,p}_u\to \Ll^p_u$
is given by the linear operator
$\Dd_u$; see~\cite[app.~A.3]{Joa-PHD}.
Thirdly, the map $\xi\mapsto\exp_u\xi$
identifies a neigborhood $V$ of zero
in ${\Ff_u}^{-1}(0)$ 
with a neigborhood of $u$
in $\Mm(x^-,x^+;\Vv)$.
Now theorem~\ref{thm:IFT}
follows immediately.

\begin{proof}[Proof of theorem~\ref{thm:IFT}]
Fix $p>2$.
Then the operator
$d\Ff_u(0)=\Dd_u:
\Ww^{1,p}_u\to \Ll^p_u$
is Fredholm by
theorem~\ref{thm:fredholm}
and surjective by assumption.
Since every surjective
Fredholm operator admits
a right inverse,
the implicit function
theorem for Banach spaces,
see e.g.~\cite[thm~A.3.3]{MS},
applies to $\Ff_u$ restricted to a
small neighborhood $V$ of zero.
It asserts that ${\Ff_u}^{-1}(0)\cap V$
is a smooth manifold
whose tangent space at zero
is given by the kernel of $\Dd_u$.
Since $\Dd_u$ is onto, it follows that
$\dim \ker \Dd_u
=\INDEX\, \Dd_u$ by definition of the Fredholm index.
On the other hand, the Fredholm index equals
$\IND_\Vv(x^-)-\IND_\Vv(x^+)$
by theorem~\ref{thm:fredholm}.
\end{proof}

\begin{proof}[Proof of
proposition~\ref{prop:finite-set}]
Set $c_*=\frac{1}{2}(\Ss_\Vv(x^-)-\Ss_\Vv(x^+))$
and identify
$$
     \widehat\Mm(x^-,x^+;\Vv)
     \simeq
     \Mm^*:=\{u\in \Mm(x^-,x^+;\Vv)\mid
     \Ss_\Vv(u(0,\cdot))=c_*\}.
$$
Here we use that
the action $\Ss_\Vv$ is 
strictly decreasing along
nonconstant (in the $s$-variable)
heat flow trajectories.
This is standard and
follows from the first
variation formula
for the functional $\Ss_\Vv$;
see e.g.~\cite[sec.~12]{MILNOR1}.
Now choose a sequence
$u^\nu$ in $\Mm^*$.
By lemma~\ref{le:broken}
there is a subsequence, still denoted by $u^\nu$,
finitely many critical points
$x_0=x^+,x_1,\ldots,x_m=x^-$,
finitely many connecting trajectories
$u_k\in\Mm(x_k,x_{k-1};\Vv)$
and sequences $s_k^\nu$ where
$k=1,\ldots,m$,
such that each shifted sequence
$u^\nu(s_k^\nu+s,t)$ converges
to $u_k(s,t)$ in $C^\infty_{loc}$.
Note that $m\ge 1$.
By the Morse--Smale assumption
theorem~\ref{thm:IFT} applies to all
moduli spaces.
Since $\p_su_k\not\equiv0$
and the heat equation~(\ref{eq:heat})
is $s$-shift invariant
this implies
$$
     \IND_\Vv(x_k)-\IND_\Vv(x_{k-1})
     =\dim \Mm(x_k,x_{k-1};\Vv)
     \ge 1,
     \quad \forall k\in\{1,\ldots,m\}.
$$
Use these inequalities
to obtain that
$\IND_\Vv(x^-)-\IND_\Vv(x^+)\ge m\ge 1$.
But by assumption the index difference is one
and therefore $m=1$.
Now this means that the
subsequence $u^\nu$
converges in $C^\infty_{loc}$ to
$u:=u_1\in\Mm(x^-,x^+;\Vv)$. 
In fact, convergence of the action functional
for fixed time $s=0$, see
proposition~\ref{prop:unifconvergence-compactsets}~(ii),
shows that $u\in\Mm^*$.
Hence $\Mm^*$ is compact
in the $C^\infty_{loc}$ topology.
On the other hand, 
the moduli space $\Mm(x^-,x^+;\Vv)$
is a manifold of dimension one by
theorem~\ref{thm:IFT}.
Now the $\R$ action
is free and therefore the quotient,
hence $\Mm^*$, is a manifold of
dimension zero.
But a zero dimensional
compact manifold consists of
finitely many points.
\end{proof}

\subsection*{The refined
implicit function theorem}
\label{subsec:refined-IFT}

\begin{proposition}[The estimate for the right inverse]
\label{prop:inverse}
Fix a perturbation $\Vv:\Ll M\to\R$
that satisfies~{\rm (V0)--(V3)}
and nondegenerate critical points
$x^\pm$ of $\Ss_\Vv$.
Assume $u\in\Mm(x^-;x^+;\Vv)$ and
$\Dd_u$ is onto.
Then, for every $p>1$, 
there is a positive constant 
$c=c(p,u)$ invariant under $s$-shifts of~$u$
such that
\begin{equation}\label{eq:D}
      \Norm{\xi^*}_{\Ww_u^{1,p}}
      \le c\Norm{\Dd_u\xi^*}_p
\end{equation}
for every
$\xi^*\in \im(
\Dd_u^*:\Ww^{2,p}_u\to\Ww^{1,p}_u)$.
Here $\Ww^{2,p}_u:=\{\xi\in\Ww^{1,p}_u
\mid \Dd_u\xi\in\Ww^{1,p}_u\}$.
\end{proposition}

\begin{proof}
[Proof of proposition~\ref{prop:inverse}]
The proof of~\cite[lemma~4.5]{DOSA}
carries over. 
We include it for the sake of completeness.
Fix $p>1$ and let $1/q+1/p=1$.
By lemma~\ref{le:fredholm}
the operators $\Dd_u$ and $\Dd_u^*$ 
are Fredholm.  
Since $\Dd_u$ is onto, the operator
$\Dd_u^*$ is injective
by proposition~\ref{prop:kernel-smooth}
and proposition~\ref{prop:coker-ker}
(hypothesis~\ref{hyp:fredholm}
is satisfied
by theorem~\ref{thm:par-exp-decay} 
on exponential decay).
Hence by the open mapping theorem
$\Dd_u^*$ satisfies the injectivity
estimate
\begin{equation}\label{eq:step1}
     \norm{\eta}_q+\norm{\Nabla{s}\eta}_q
     +\norm{\Nabla{t}\Nabla{t}\eta}_q
     \le c_1\Norm{\Dd_u^*\eta}_q
\end{equation}
for every
$\eta\in\Ww^{1,q}_u$
and with shift invariant
constant $c_1=c_1(q,u)>0$.
Next observe that
\begin{equation}\label{eq:step2}
     \frac{\langle\Dd_u^*\xi,\Dd_u^*\eta\rangle}
     {\norm{\Dd_u^*\eta}_q}
     =\frac{\langle\Dd_u\Dd_u^*\xi,\eta\rangle}
     {\norm{\Dd_u^*\eta}_q}
     \le \Norm{\Dd_u\Dd_u^*\xi}_p
     \frac{\norm{\eta}_q}
     {\norm{\Dd_u^*\eta}_q}
     \le c_1 \Norm{\Dd_u\Dd_u^*\xi}_p
\end{equation}
for all $\xi\in\Ww^{2,p}_u$
and $\eta\in\Ww^{1,q}_u$.
Here the first step
is by definition of the
formal adjoint
and the second one by
H\"older's inequality.
The third step is
by~(\ref{eq:step1}).
Now there is a shift invariant constant
$c_2=c_2(p,u)>0$ such that
\begin{equation}\label{eq:step3}
     \Norm{\Dd_u^*\xi}_p
     \le c_2 \sup_{\eta\in\Ww^{1,q}_u}
     \frac{\langle\Dd_u^*\xi,\Dd_u^*\eta\rangle}
     {\norm{\Dd_u^*\eta}_q}
\end{equation}
for every $\xi\in\Ww^{2,p}_u$.
The argument uses that
$\Dd_u$ is onto and
$\dim\ker\Dd_u<\infty$.
The constant $c_2$ depends also
on the choice of an $L^2$ orthonormal
basis of $\ker\Dd_u$.
Full details are given 
in step~2 of the proof of
lemma~4.5 in~\cite{DOSA}.
Now the linear estimate
proposition~\ref{prop:par-linear}
for $\xi^*:=\Dd_u^*\xi$
shows that
$$
     \Norm{\xi^*}_{\Ww^{1,p}_u}
     \le c_3\left(\Norm{\Dd_u\xi^*}_p
     +\Norm{\xi^*}_p\right)
$$
where the constant
$c_3(p,u)$ is again shift invariant.
To estimate the second term in the sum
apply~(\ref{eq:step3})
and~(\ref{eq:step2})
to obtain that
$\norm{\xi^*}_p
\le c_1c_2 \norm{\Dd_u\xi^*}_p$.
\end{proof}

\begin{proposition}[Quadratic estimate]
\label{prop:Quadest0}
Fix a perturbation $\Vv:\Ll M\to\R$
that satisfies~{\rm (V0)--(V1)}.
Let $\iota>0$ be the
injectivity radius of $M$
and fix constants
$1<p<\infty$ and $c_0>0$.
Then there is a constant $C=C(p,c_0)>0$ 
such that the following is true.
If $u:\R\times S^1\to M$
is a smooth map
and $\xi$
is a compactly supported
smooth vector field along $u$ 
such that
$$
     \Norm{\p_su}_\infty+\Norm{\p_tu}_\infty
     +\Norm{\Nabla{t}\p_tu}_\infty\le c_0,\quad
     \Norm{\xi}_\infty \le \iota,
$$
then
$$
     \Norm{\Ff_u(\xi)-\Ff_u(0)-d\Ff_u(0) \xi}_p
     \le C \Norm{\xi}_\infty \Norm{\xi}_{\Ww^{1,p}_u} 
     \left( 1+ \Norm{\xi}_{\Ww^{1,p}_u} \right).
$$
\end{proposition}

\begin{proof}
Recall the definition~(\ref{eq:exponential-identity}) 
of the maps $E_i$ and $E_{ij}$
and write
$$
     \Ff_u(\xi)-\Ff_u(0)
     -\left.\frac{d}{d\tau}\right|_{\tau=0}
     \Ff_u(\tau\xi)
     =f(\xi)-g(\xi)-h(\xi)
$$
where
\begin{equation*}
\begin{split}
     f(\xi)
    &:=\Phi(u,\xi)^{-1}\p_sE(u,\xi)-\p_su
     -\left.\frac{d}{d\tau}\right|_{\tau=0}
     \Phi(u,\tau\xi)^{-1}\p_su\\
    &\quad
     -\left.\frac{d}{d\tau}\right|_{\tau=0}
     \p_s E(u,\tau\xi)
    \\
     g(\xi)
    &:=\Phi(u,\xi)^{-1}\Nabla{t}\p_tE(u,\xi)
     -\Nabla{t}\p_tu
     +\left(\Nabla{2}\Phi(u,0)\xi\right)
     \Nabla{t}\p_tu\\
    &\quad
     -\left.\frac{d}{d\tau}\right|_{\tau=0}
     \Nabla{t}\p_t E(u,\tau\xi)
    \\
     h(\xi)
    &:=\Phi(u,\xi)^{-1}\grad\,\Vv(E(u,\xi))
     -\grad\,\Vv(u)
     +\left(\Nabla{2}\Phi(u,0)\xi\right)
     \grad\,\Vv(u)\\
    &\quad
     -\left.\frac{d}{d\tau}\right|_{\tau=0}
     \grad\,\Vv(E(u,\tau\xi)).
\end{split}
\end{equation*}
Here we used that $\Phi(u,0)=\1$.
Straightforward calculation using the 
identities~(\ref{eq:E_ij(0)})
shows that
$f(\xi)=f_1(\xi)\Nabla{s}\xi+f_2(\xi)$
where
\begin{equation*}
\begin{split}
     f_1(\xi)\Nabla{s}\xi
    &=\left(\Phi(u,\xi)^{-1}E_2(u,\xi)
     -\1\right)\Nabla{s}\xi
    \\
     f_2(\xi)\p_su
    &=\left(\Phi(u,\xi)^{-1}E_1(u,\xi)-\1
     +\Nabla{2}\Phi(u,0)\xi\right)\p_su,
\end{split}
\end{equation*}
that
$$
     g
     =g_1\circ\Nabla{t}\p_tu
     +g_2\circ\left(\p_tu,\p_tu\right)
     +g_3\circ\Nabla{t}\Nabla{t}\xi
     +g_4\circ\left(\p_tu,\Nabla{t}\xi\right)
     +g_5\circ\left(\Nabla{t}\xi,\Nabla{t}\xi
     \right)
$$
where
\begin{equation*}
\begin{split}
     g_1(\xi)
    &=\Phi(u,\xi)^{-1}E_1(u,\xi)-\1
     +\Nabla{2}\Phi(u,0)\xi
    \\
     g_2(\xi)
    &=\Phi(u,\xi)^{-1}E_{11}(u,\xi)
     -\left.\frac{d}{d\tau}\right|_{\tau=0}
     E_{11}(u,\tau\xi)
    \\
     g_3(\xi)
    &=\Phi(u,\xi)^{-1}E_2(u,\xi)-\1
    \\
     g_4(\xi)
    &=2\Phi(u,\xi)^{-1}E_{12}(u,\xi)
    \\
     g_5(\xi)
    &=\Phi(u,\xi)^{-1}E_{22}(u,\xi),
\end{split}
\end{equation*}
and that
$$
     h(\xi)
     =\Phi(u,\xi)^{-1}\grad\,\Vv(E(u,\xi))
     -\left(\1
     -\left(\Nabla{2}\Phi(u,0)\xi\right)\right)
     \grad\,\Vv(u)
     -\Hh_\Vv(u)\xi.
$$
Here $\Hh_\Vv$ denotes the covariant Hessian
of $\Vv$ given by~(\ref{eq:Hess-Vv}).
It follows by inspection using the
identities~(\ref{eq:E_ij(0)})
that the maps $f_2,g_1,g_2$, and $h$
together with their first derivative
are zero at $\xi=0$.
Therefore there exists
a constant $c>0$ which depends
continuously on $\abs{\xi}$
and the constant in~(V1)
such that
$$
     \Abs{(f_2+g_1+g_2+h)(\xi)}
     \le c \Abs{\xi}^2\left(
     \Abs{\p_su}+\Abs{\Nabla{t}\p_tu}
     +\Abs{\p_tu}^2+1
     \right)
$$
pointwise at every $(s,t)$.
Similarly, it follows that
the remaining functions
are zero at $\xi=0$
and therefore
$$
     \Abs{(f_1+g_3+g_4+g_5)(\xi)}
     \le c \Abs{\xi}\left(
     \Abs{\Nabla{s}\xi}
     +\Abs{\Nabla{t}\Nabla{t}\xi}
     +\Abs{\Nabla{t}\xi}\Abs{\p_tu}
     +\Abs{\Nabla{t}\xi}^2
     \right).
$$
Take these pointwise estimates
to the power $p$,
integrate them over $\R\times S^1$
and pull out $L^\infty$ norms
of $\p_su,\p_tu$, and
$\Nabla{t}\p_tu$
to obtain the conclusion
of proposition~\ref{prop:Quadest0}.
The term $\Abs{\xi}\cdot\Abs{\Nabla{t}\xi}^2$
involving a product of
first order terms is taken
care of by the product estimate
lemma~\ref{le:product-derivatives} and
remark~\ref{rmk:product-estimate}.
Here we use the fact that
the (compact) support of $\xi$
is contained in some
set $(a,b]\times S^1$.
\end{proof}

\subsubsection*{Proof of the refined implicit function
theorem~\ref{thm:modified-IFT}}
\label{subsubsec:proof-modified-IFT}

Assume the result is false.
Then there exist constants $p>2$
and $c_0>0$ and a sequence
of smooth maps
$u_\nu:\R\times S^1\to M$
such that
$\lim_{s\to\pm\infty} u_\nu(s,\cdot)
=x^\pm(\cdot)$ exists, uniformly in $t$,
and
\begin{equation}\label{eq:IFT-decay}
     \Abs{\p_su_\nu(s,t)}
     \le \frac{c_0}{1+s^2},\qquad
     \Norm{\p_tu_\nu}_\infty
     \le c_0,\qquad
     \Norm{\Nabla{t}\p_tu_\nu}_\infty
     \le c_0,
\end{equation}
for all $(s,t)\in\R\times S^1$ 
and
\begin{equation}\label{eq:IFT-approx-zero}
     \Norm{\p_su_\nu-\Nabla{t}\p_tu_\nu
     -\grad\, \Vv(u_\nu)}_p
     \le \frac{1}{\nu},
\end{equation}
but which does not satisfy the conclusion
of theorem~\ref{thm:modified-IFT} for $c=\nu$.
This means that for every
$u_*\in\Mm(x^-,x^+;\Vv)$
and every $\xi^\nu\in \im\,\Dd_{u_*}^*
\cap \Ww_{u_*}$ the following holds.
If $u_\nu=\exp_{u_*}(\xi^\nu)$
then
\begin{equation}\label{eq:IFT-contra}
     \Norm{\p_su_\nu-\Nabla{t}\p_tu_\nu
     -\grad\, \Vv(u_\nu)}_p
     <\frac{1}{\nu}
     \Norm{\xi^\nu}_\Ww.
\end{equation}
The {\bf time shift} 
of a smooth map $u:\R\times S^1$
by $\sigma\in\R$
is defined pointwise by 
$$
     u^\sigma(s,t):=u(s+\sigma,t).
$$
Set $a_0:=2c_0^2$ and observe that
$$
     \Ss_\Vv(x^-)
     =\lim_{s\to-\infty}\Ss_\Vv(u_\nu(s,\cdot))
     =\frac12 \Norm{\p_tu_\nu(s,\cdot)}_2^2
     -\Vv(u_\nu(s,\cdot))
     \le \frac12 c_0^2+C_0
     \le a_0.
$$
Here we used
the assumption on asymptotic 
$W^{1,2}$ convergence,
estimate~(\ref{eq:IFT-decay}), 
and our choice of the constant
$c_0>1$ larger than the constant $C_0$ in~(V0).
Now fix a regular value $c_*$ of $\Ss_\Vv$
between $\Ss_\Vv(x^+)$ and $\Ss_\Vv(x^-)$.
Here we use that
the set $\Pp^{a_0}(\Vv)$ is finite,
because $\Ss_\Vv$ is Morse--Smale
below level $a_0$.
Applying time shifts, if necessary, we
may assume without loss
of generality that
\begin{equation}\label{eq:shift}
     \Ss_\Vv\left(
     u_\nu(0,\cdot)\right)
     =c_*.
\end{equation}
Furthermore we
set $\tilde{c}_0=a$
and let $C_0=C_0(a,\Vv)>0$
be the constant in 
theorem~\ref{thm:apriori}
with that choice. Then
we have the apriori estimates
\begin{equation}\label{eq:aprior-bound1}
     \left\|\p_su\right\|_\infty
     +\left\|\p_tu\right\|_\infty
     +\left\|\Nabla{t}\p_tu\right\|_\infty
     \le C_0
\end{equation}
for all $u\in\Mm(x,y;\Vv)$
and $x,y\in\Pp^a(\Vv)$.

\vspace{.1cm}
\noindent
{\bf Claim.}
{\it There is a subsequence, 
still denoted by $u_\nu$, a constant $C>0$,
a trajectory $u\in\Mm(x^-,x^+;\Vv)$,
and a sequence of times $\sigma_\nu$
such that the sequence
$\eta_\nu$ determined by the identity
$$
     u_\nu=\exp_{u^{\sigma_\nu}}
     (\eta_\nu)
$$
satisfies
$\eta_\nu\in \im\,\Dd_{u^{\sigma_\nu}}^*
\cap \Ww_{u^{\sigma_\nu}}$ and
\begin{equation}\label{eq:IFT-limit}
     \lim_{\nu\to\infty}\left(
     \Norm{\eta_\nu}_\infty
     +\Norm{\eta_\nu}_p\right)
     =0
     ,\qquad
     \Norm{\eta_\nu}_\Ww
     \le C.
\end{equation}
}

\vspace{.1cm}
\noindent
Before we prove the claim
we show how it leads to a contradiction.
Consider the trajectories
$u^{\sigma_\nu} \in\Mm(x^-,x^+;\Vv)$ 
and vector fields $\eta_\nu$
provided by the claim.
They satisfy the assumptions of
the quadratic estimate,
proposition~\ref{prop:Quadest0},
by~(\ref{eq:aprior-bound1})
and by choosing a further
subsequence, if necessary, 
to achieve that
$\norm{\eta_\nu}_\infty<\iota$.
Set $c_0^\prime=C_0(a,\Vv)$
and let $C_2=C_2(p,c_0^\prime)$
be the constant
in proposition~\ref{prop:Quadest0}
with that choice.
Furthermore, since
$\Mm(x^-,x^+;\Vv)/\R$ is a finite set
by proposition~\ref{prop:finite-set}
(and $\Pp^a(\Vv)$ is a finite set as well)
the estimate for the right inverse,
proposition~\ref{prop:inverse}, 
applies with constant $C_1$
depending only on $p$, $a$, and $\Vv$.
Now by the definition~(\ref{eq:Ff_u})
of the map $\Ff_{\hat{u}}$
and the fact that parallel
transport is an isometry
we obtain the first step 
in the following estimate, namely
\begin{equation*}
  \begin{split}
    \Norm{
    \p_su_\nu
    -\Nabla{t}\p_tu_\nu
    -\grad\Vv(u_\nu)
    }_p
   &=\Norm{\Ff_{\hat{u}}(\eta_\nu)}_p\\
   &\ge\Norm{\Dd_{\hat{u}} \eta_\nu}_p
    -\Norm{\Ff_{\hat{u}}(\eta_\nu)
    -\Ff_{\hat{u}}(0)
    -d\Ff_{\hat{u}}(0)\eta_\nu}_p\\
   &\ge\Norm{\eta_\nu}_\Ww
    \left(\frac{1}{C_1}
    -C_2\Norm{\eta_\nu}_\infty
    \left(1+\Norm{\eta_\nu}_\Ww
    \right)\right)\\
   &\ge\frac{1}{2C_1}
    \Norm{\eta_\nu}_\Ww.
  \end{split}
\end{equation*}
Step two uses that
$\Ff_{\hat{u}}(0)=
\p_s\hat{u}-\Nabla{t}\p_t\hat{u}
-\grad\Vv(\hat{u})=0$
and $d\Ff_{\hat{u}}(0)=\Dd_{\hat{u}}$.
Step three is by
proposition~\ref{prop:inverse} and
proposition~\ref{prop:Quadest0}.
By~(\ref{eq:IFT-limit})
the last step holds
for sufficiently large
$\nu$. For $\nu>2C_1$
the estimate
contradicts~(\ref{eq:IFT-contra})
and this proves 
theorem~\ref{thm:modified-IFT}.
It only remains to prove
the claim.
This takes four steps.

\vspace{.1cm}
\noindent
{\bf Step 1.}
{\it There is a subsequence
of $u_\nu$, still denoted by $u_\nu$,
and a trajectory
$u\in\Mm(x^-,x^+;\Vv)$ such that
\begin{equation}\label{eq:IFT-step1}
     u_\nu=\exp_u(\xi_\nu),\qquad
     \lim_{\nu\to\infty}\left(
     \Norm{\xi_\nu}_\infty
     +\Norm{\xi_\nu}_p\right)
     =0.
\end{equation}
}

\begin{proof}
We embed the compact Riemannian manifold $M$
isometrically into some
Euclidean space $\R^N$
and view any continuous map
to $M$ as a map into $\R^N$ taking values
in the embedded manifold.
By translation we may assume
that the embedded $M$ contains the origin.
Now $L^p$ and $L^\infty$ norms of $u_\nu$
are provided by the ambient Euclidean space.
By compactness of $M$
and, in particular, the $L^\infty$ bounds
in~(\ref{eq:IFT-decay})
we obtain on every compact cylindrical domain
$Z_T:=[-T,T]\times S^1$
the estimates
$$
     \Norm{u_\nu}_{L^p(Z_T)}
     \le (2T)^\frac{1}{p}\, \diam M,
     \quad
     \Norm{\p_tu_\nu}_{L^p(Z_T)}
     +\Norm{\Nabla{t}\p_tu_\nu}_{L^p(Z_T)}
     \le 2c_0(2T)^\frac{1}{p},
$$
and
\begin{equation}\label{eq:bd-2}
     \Norm{\p_su_\nu}_r
     \le 4c_0\quad
     \forall r\in(1,\infty].
\end{equation}
The latter follows by the estimate
$$
     \int_{-\infty}^\infty
     \left(\frac{1}{1+s^2}\right)^r \,ds
     \le
     2+2\int_1^\infty
     \frac{1}{s^{2r}}\,ds
     =
     \frac{4}{2-1/r}
     <4
$$
whenever $r>1$.
Hence the sequence $u_\nu$ is
uniformly bounded in
$\Ww^{1,p}(Z_T)$.
Thus by the Arzela-Ascoli
and the Banach-Alaoglu
theorem a suitable subsequence,
still denoted by $u_\nu$,
converges strongly in $C^0$
and weakly in $\Ww^{1,p}$
on every compact cylindrical
domain $Z_T$ to some
continuous map
$u:\R\times S^1\to M$
which is locally of class $\Ww^{1,p}$.
Hence $\p_su_\nu-\Nabla{t}\p_tu_\nu
-\grad\Vv(u_\nu)$
converges weakly in $L^p$ to
$\p_su-\Nabla{t}\p_tu-\grad\Vv(u)$.
On the other hand, 
by~(\ref{eq:IFT-approx-zero})
it converges to zero in $L^p$.
By uniqueness of limits
$u$ satisfies the heat
equation~(\ref{eq:heat})
almost everywhere.
Thus $u$ is smooth
by theorem~\ref{thm:regularity}.

Fix $s\in\R$ and observe that
by~(\ref{eq:IFT-decay})
there are uniform $C^1(S^1)$ bounds
for the sequence $\p_tu_\nu(s,\cdot)$.
Hence by Arzela-Ascoli
a suitable subsequence,
still denoted by $\p_tu_\nu(s,\cdot)$,
converges in $C^0(S^1)$
to $\p_tu(s,\cdot)$.
Thus
$$
     \lim_{\nu\to\infty}
     \Ss_\Vv(u_\nu(s,\cdot))
     =\Ss_\Vv(u(s,\cdot))
$$
and therefore
$\Ss_\Vv(u(0,\cdot))=c_*$
by~(\ref{eq:shift}).
Recall that $\p_su=\Nabla{t}\p_tu+\grad\,\Vv(u)$.
When restricted to $s=0$ this means 
that the vector field
$\p_su(0,\cdot)$  is equal to the 
$L^2$ gradient of $\Ss_\Vv$
at the loop $u(0,\cdot)$.
But $\Ss_\Vv(u(0,\cdot))=c_*$ and $c_*$
is a regular value.
Hence $\p_su(0,\cdot)$
cannot vanish identically.

On the other hand,
by~(\ref{eq:IFT-decay}) and
axiom~(V0) with constant $C_0$
it follows exactly as above that
$$
     \sup_\nu\Ss_\Vv(u_\nu(s,\cdot))
     =\sup_\nu\frac12\Norm{\p_tu_\nu(s,\cdot)}_2^2
     -\Vv(u_\nu)
     \le a_0.
$$
This shows that all 
relevant trajectories
including relevant limits over $s$ or $\nu$
lie in the sublevel set
$\Ll^{a_0} M$ on which $\Ss_\Vv$ 
is Morse--Smale by assumption.
In particular, we have that
$\sup_{s\in\R}\Ss_\Vv(u(s,\cdot))\le a_0$
and therefore the energy of $u$ is finite
by lemma~\ref{le:energy-bound}.
Hence by
the exponential decay
theorem~\ref{thm:par-exp-decay}
there are critical points 
$y^\pm\in\Pp^{a_0}(\Vv)$
such that $u(s,\cdot)$ converges to
$y^\pm$ in $C^2(S^1)$, as $s\to\pm\infty$.
Moreover, the limits $y^-$ and $y^+$
are distinct,
because the action along a
nonconstant trajectory is strictly
decreasing and the trajectory
is nonconstant because
$\p_su$ is not identically zero
as observed above. 

More generally,
a standard argument shows the following,
see e.g.~\cite[lemma~10.3]{SaJoa-LOOP}.
There exist critical points
$x^-=x^0,x^1,\dots,x^\ell=x^+
\in \Pp^{a_0}(\Vv)$
and trajectories 
$u^k\in\Mm(x^{k-1},x^k;\Vv)$,
$\p_su^k\not\equiv0$, 
for $k\in\{1,\dots,\ell\}$,
a subsequence, still denoted by $u_\nu$,
and sequences $s_\nu^k\in\R$,
$k\in\{1,\dots,\ell\}$,
such that the shifted sequence
$u_\nu(s_\nu^k+s,t)$
converges to $u^k(s,t)$
in an appropriate topology.
The point here is that 
$\p_su^k\not\equiv0$
and therefore
the Morse index strictly
decreases along the sequence 
$x^-=x^0,x^1,\dots,x^\ell=x^+$.
Namely, by the Morse--Smale condition
each Fredholm operator $\Dd_{u^k}$
is onto, hence its Fredholm index
is equal to the dimension of its kernel.
But this is strictly positive
because the kernel contains the nonzero
element $\p_su^k$.
On the other hand, by
lemma~\ref{le:fredholm}
the Fredholm index is given by 
the difference of Morse indices 
$\IND_\Vv(x^{k-1})-\IND_\Vv(x^k)$.
Our assumption that the pair $x^\pm$
has Morse index difference one
then implies that $\ell=1$
and this proves that
$u\in\Mm(x^-,x^+;\Vv)$.
The first assertion of step~1.

It remains to prove the second assertion,
that is~(\ref{eq:IFT-step1}).
The key fact to prove~(\ref{eq:IFT-step1}) 
is that $u_\nu(s,\cdot)$ not only converges
in $W^{1,2}(S^1)$ to $x^\pm$,
as $s\to\pm\infty$,
but that the rate of convergence
is independent of $\nu$.
More precisely, we prove that
for every $\eps>0$ there is a time
$T=T(\eps)>1$ such that
\begin{equation}\label{eq:uniform-decay}
     s>T 
     \qquad\Longrightarrow\qquad
     d\left( u_\nu(s,t),x^+(t)\right)<\eps
\end{equation}
for all $t\in S^1$ and $\nu\in\N$.
Recall that $M$ is embedded
isometrically in $\R^N$.
By the fundamental theorem
of calculus and uniform
decay~(\ref{eq:IFT-decay}) we have that
\begin{equation}\label{eq:uniform-decay2}
     \Abs{x^+(t)-u_\nu(\sigma,t)}_{\R^N}
     =\Abs{\int_{\sigma}^\infty  
     \p_su_\nu(s,t)\,ds}_{\R^N}
     \le \int_{\sigma}^\infty 
     \frac{c_0}{s^2} ds
     =\frac{c_0}{\sigma}
\end{equation}
for all $t\in S^1$,
$\nu\in\N$, and
$\sigma>1$ sufficiently large.
The Riemannian distance
$d$ in $M$ and the restriction
of the Euclidean distance in $\R^N$
to the compact manifold $M$ 
are locally equivalent.
Hence~(\ref{eq:uniform-decay2})
implies~(\ref{eq:uniform-decay}).
Let $Z_T^+
:=[T,\infty)\times S^1$
denote the positive end
of the cylinder $\R\times S^1$
and $Z_T^-$ the negative end.
Let $\iota>0$ be the
injectivity radius of $M$.
Now fix $\eps\in(0,\iota/2)$ and
choose $T=T(\eps)>0$
such that the ends $u(Z_T^\pm)$
and $u_\nu(Z_T^\pm)$
for all $\nu$ are
contained in the
$(\eps/6)$-neighborhood
of $x^\pm(S^1)$.
Such $T$ exists
by~(\ref{eq:uniform-decay}).
Since $u_\nu$ converges to $u$
uniformly on $Z_T$, there exists
$\nu_0=\nu_0(T(\eps))\in\N$ such that
$\norm{\xi_\nu}_{L^\infty(Z_T)}<\eps/3$
for every $\nu\ge\nu_0$.
Hence
\begin{equation}
\begin{split}
     \Norm{\xi_\nu}_\infty
    &=\Norm{\xi_\nu}_{L^\infty(Z_T^-)}
     +\Norm{\xi_\nu}_{L^\infty(Z_T)}
     +\Norm{\xi_\nu}_{L^\infty(Z_T^+)}\\
    &\le
     \sup_{Z_T^-}\left( d(u_\nu,x^-)
     +d(x^-,u)\right)
     +\Norm{\xi_\nu}_{L^\infty(Z_T)}\\
    &\quad
     +\sup_{Z_T^+}\left( d(u_\nu,x^+)
     +d(x^+,u)\right)\\
    &\le \eps
\end{split}
\end{equation}
for every $\nu\ge\nu_0$.
This proves that
the $L^\infty$ limit
in~(\ref{eq:IFT-step1})
is zero.
To prove that the $L^p$ limit
is zero one uses again the decomposition
of $\R\times S^1$ into the
compact part $Z_T$ and the
two ends $Z_T^\pm$.
The left hand side
of~(\ref{eq:uniform-decay2}) is
$p$-integrable over the ends $Z_T^\pm$.
The key fact is that
the value of this integral
does not depend on $\nu$
and converges to zero as
$\abs{T}\to\infty$.
A similar integral is needed
in the case of $u$.
Here the exponential decay
theorem~\ref{thm:par-exp-decay}
shows that the
integral exists and
converges to zero as
$\abs{T}\to\infty$.
This concludes the proof of
step~1.
\end{proof}

\noindent
{\bf Step 2.}
{\it Set $\eps_\nu:=\norm{\xi_\nu}_\infty
+\norm{\xi_\nu}_p$
and let $C_0$ be the constant
in~(\ref{eq:aprior-bound1}).
Then there is a constant
$\sigma_0>0$ and integer $\nu_0\ge 1$
such that
$\eta=\eta(\sigma,\nu)$
is determined by the identity
$u_\nu=\exp_{u^\sigma}\eta$
and satisfies
$\norm{\eta}_\infty<\iota/2$
for all
$\sigma\in[-\sigma_0,\sigma_0]$
and $\nu\ge\nu_0$.
Furthermore, there is a constant
$c_2=c_2(a_0,\sigma_0)>0$
such that
$$
      \Norm{\eta}_\infty
      \le \eps_\nu+C_0\Abs{\sigma}
      ,\qquad
      \Norm{\eta}_p
      \le 2\eps_\nu+c_2\Abs{\sigma}
$$
and
$$
      \Norm{\Nabla{s}\eta}_p
      \le c_2,\qquad
      \Norm{\Nabla{t}\eta}_\infty
      \le c_2,\qquad
      \Norm{\Nabla{t}\Nabla{t}\eta}_p
      \le c_2
$$
for all $\sigma\in[-\sigma_0,\sigma_0]$
and $\nu\ge\nu_0$.
}

\begin{proof}
Existence of $\sigma_0$ and $\nu_0$
follows from the fact that $\eta(\nu,0)=\xi_\nu$,
continuity of time shift,
and the $L^\infty$ limit in~(\ref{eq:IFT-step1}).
Now denote by $L$
the length functional.
Then for every $\sigma\in\R$
and $\gamma(r):=u(s+r\sigma,t)$
for $r\in[0,1]$
we have that
\begin{equation}\label{eq:est-length}
     d\left( u(s,t),u(s+\sigma,t)
     \right)
     \le L(\gamma)
     =\Abs{\sigma}
     \int_0^1\Abs{\p_su(s+r\sigma,t)}dr
     \le \Abs{\sigma}
     \Norm{\p_su}_\infty.
\end{equation}
Since
$d\left( u_\nu(s,t),u(s,t)\right)
=\Abs{\xi_\nu(s,t)}\le\eps_\nu$,
the first estimate
of step~2 follows from
$\Abs{\eta(s,t)}=d\left( u_\nu(s,t),
u(s+\sigma,t)\right)$,
the triangle inequality,
and~(\ref{eq:aprior-bound1}).
To prove the second estimate
note that the triangle inequality
also implies that
$$
     \Norm{\eta}_p^p
     \le2^{p-1}\Norm{\xi_\nu}_p^p
     +2^{p-1}\int_{-\infty}^\infty\int_0^1
     d\left( u(s,t),u(s+\sigma,t)
     \right)^p\, dtds.
$$
By 
theorem~\ref{thm:par-exp-decay}
on exponential decay
there are constants $\rho,c_3>2$
such that for all 
$(\tilde{s},t)\in\R\times S^1$
we have that
\begin{equation}\label{eq:bd1}
     \Abs{\p_su(\tilde{s},t)}
     \le c_3e^{-\rho\abs{\tilde{s}}},\qquad
     \Norm{\p_su}_r\le c_3\quad
     \forall r>1.
\end{equation}
Note that the
constants $\rho$ and $c_3$
depend only on $a_0$, since
the set $\Pp^{a_0}(\Vv)$ is finite and
there are only finitely
many elements of $\Mm(x^-,x^+;\Vv)$
which satisfy~(\ref{eq:shift}).
By the first inequality
in~(\ref{eq:est-length})
and the first estimate in~(\ref{eq:bd1})
with $\tilde{s}=s+r\sigma$
$$
     d\left( u(s,t),u(s+\sigma,t)
     \right)
     \le\Abs{\sigma}
     \int_0^1
     \Abs{\p_su(s+r\sigma,t)} dr
     \le\Abs{\sigma} c_3
     e^{\rho\sigma_0}
     e^{-\rho\abs{s}}.
$$
Hence the left hand side
is $L^p$ integrable.
This concludes the proof
of the second estimate of
step~2.
To prove the next two estimates
we differentiate the identity
$
     \exp_{u^\sigma}\eta
     =u_\nu
$
with respect to $s$ and $t$
to obtain that
\begin{align}\label{eq:nabla-s-eta}
     E_1(u^\sigma,\eta)\p_su^\sigma
     +E_2(u^\sigma,\eta)\Nabla{s}\eta
    &=\p_su_\nu
     \\\label{eq:nabla-t-eta}
     E_1(u^\sigma,\eta)\p_tu^\sigma
     +E_2(u^\sigma,\eta)\Nabla{t}\eta
    &=\p_tu_\nu.
\end{align}
Here the maps $E_i$ are defined
by~(\ref{eq:exponential-identity}).
Since
$\norm{\p_su^\sigma}_p\le c_3$
by~(\ref{eq:bd1})
and $\norm{\p_su_\nu}_p\le 4c_0$
by~(\ref{eq:bd-2}),
the $L^p$ norm of $\Nabla{s}\eta$
is uniformly bounded as well.
Similarly, since
$\norm{\p_tu^\sigma}_\infty\le C_0$
by~(\ref{eq:aprior-bound1}) and
$\norm{\p_tu_\nu}_\infty\le c_0$
by~(\ref{eq:IFT-decay}),
the $L^\infty$
norm of $\Nabla{t}\eta$
is uniformly bounded.
To prove the last estimate
of step~2 differentiate
(\ref{eq:nabla-t-eta})
covariantly with respect to $t$
and abbreviate
$E_{ij}=E_{ij}(u^\sigma,\eta)$
to obtain
\begin{equation*}
\begin{split}
     &E_{11}(u^\sigma,\eta)
      \left(\p_tu^\sigma,\p_tu^\sigma\right)
      +E_{12}(u^\sigma,\eta)
      \left(\p_tu^\sigma,\Nabla{t}\eta\right)
      +E_1(u^\sigma,\eta)\,
      \Nabla{t}\p_tu^\sigma
      \\
     &+E_{21}(u^\sigma,\eta)
      \left(\Nabla{t}\eta,\p_tu^\sigma\right)
      +E_{22}(u^\sigma,\eta)
      \left(\Nabla{t}\eta,\Nabla{t}\eta\right)
      +E_2(u^\sigma,\eta)\,
      \Nabla{t}\Nabla{t}\eta
      \\
     &+\grad\,\Vv(u_\nu)-\p_su_\nu\\
     &=\Nabla{t}\p_tu_\nu
      +\grad\,\Vv(u_\nu)-\p_su_\nu.
\end{split}
\end{equation*}
This identity implies
a uniform $L^p$ bound for
$\Nabla{t}\Nabla{t}\eta$
as follows.
The right hand side
is bounded in $L^p$
by $1/\nu$
and the last term of the
left hand side by $4c_0$
according to~(\ref{eq:bd-2}).
Since $E_{ij}(u^\sigma,0)=0$
and we have uniform $L^\infty$ bounds
for each of the two linear terms
to which
$E_{ij}(u^\sigma,\eta)$ is applied,
we can estimate
the $L^p$ norm by
a constant times $\norm{\eta}_p$.
The only terms left
are term three and term seven 
of the left hand side.
By the heat equation~(\ref{eq:heat})
their sum equals
$$
     E_1(u^\sigma,\eta)\,
     \p_su^\sigma
     -E_1(u^\sigma,\eta)\,
     \grad\,\Vv(u^\sigma)
     +\grad\Vv(u_\nu).
$$
Since
$\norm{\p_su^\sigma}_p\le c_3$
by~(\ref{eq:bd1}),
the $L^p$ norm
of the first term is
uniformly bounded.
Consider the
remaining two terms
as a function $f$ of $\eta$.
Then $f(0)=0$, because
$E_1(u^\sigma,0)=\1$ and
$\eta=0$ means
$u_\nu=u^\sigma$.
Hence $\norm{f}_p$ is
uniformly bounded
by a constant times
$\norm{\eta}_p$.
Here we used axiom~(V0).
This proves step~2.
\end{proof}

\noindent
{\bf Step 3.}
{\it For
$\sigma\in[-\sigma_0,\sigma_0]$
consider the function
$
     \theta_\nu(\sigma)
     :=-\langle\p_su^\sigma,
     \eta\rangle
$
where $\eta=\eta(\sigma,\nu)$
has been defined in step~2 
by the identity
$u_\nu=\exp_{u^\sigma}\eta$
and where $\langle\cdot,\cdot\rangle$ 
denotes the $L^2(\R\times S^1)$ 
inner product.
This function has the property that
$$
     \theta_\nu(\sigma)=0
     \quad\Longleftrightarrow\quad
     \eta\in\im \Dd_{u^\sigma}^*.
$$
Moreover, there exist new constants
$\sigma_0>0$ and $\nu_0\in\N$
such that
$$
     \abs{\theta_\nu(0)}
     \le c_3\eps_\nu,\qquad
     \frac{d}{d\sigma} \theta_\nu(\sigma)
     \ge \frac{\mu}{2},
$$
for all $\sigma\in[-\sigma_0,\sigma_0]$
and $\nu\ge\nu_0$
where $\mu:=\Ss_\Vv(x^-)-\Ss_\Vv(x^+)>0$.
}

\begin{proof}
`$\Leftarrow$'
follows by definition
of the formal adjoint operator
using that $\p_su^\sigma\in\ker \Dd_{u^\sigma}$.
We prove `$\Rightarrow$'.
The kernel of the linear operator
$\Dd_{u^\sigma}$
is 1-dimensional:
It is Fredholm of index one
by theorem~\ref{thm:fredholm}
and it is onto by the Morse--Smale
condition.
This kernel is spanned
by the (nonzero) element
$\p_su^\sigma$.
Now consider $\Dd_{u^\sigma}^*$
on the domain $\Ww^{2,p}$
and apply proposition~\ref{prop:coker-ker}
to obtain that
$\Ww^{1,p}=\ker \Dd_{u^\sigma}
\oplus\im\Dd_{u^\sigma}^*$.
The implication '$\Rightarrow$'
now follows immediately
by contradiction.

Set $1/q+1/p=1$.
By~(\ref{eq:bd1}) and the
definition of the sequence
$\eps_\nu\to 0$ in step~2
it follows that
$$
     \abs{\theta_\nu(0)}
     =\Abs{\langle\p_su,
     \xi_\nu\rangle_{L^2}}
     \le \Norm{\p_su}_q
     \Norm{\xi_\nu}_p
     \le c_3\eps_\nu.
$$
Abbreviate $E_i=E_i(u^\sigma,\eta)$.
Then straightforward
calculation using the
identity~(\ref{eq:nabla-s-eta})
for $\Nabla{s}\eta$
shows that
\begin{equation*}
\begin{split}
     \frac{d}{d\sigma}
     \theta_\nu(\sigma)
    &=
     -\langle\Nabla{s}\p_su^\sigma,
     \eta\rangle_{L^2}
     -\langle\p_su^\sigma,
     -\p_su^\sigma+\p_su^\sigma
     -E_2^{-1}E_1\p_su^\sigma
     \rangle_{L^2}\\
    &\ge
     -\Norm{\Nabla{s}\p_su^\sigma}_q
     \Norm{\eta}_p
     +\Norm{\p_su^\sigma}_2^2
     -\Norm{\p_su^\sigma}_q
     \Norm{\p_su^\sigma}_\infty
     c_4 \Norm{\eta}_p\\
    &=\Norm{\p_su}_2^2
     -\Norm{\eta}_p\left(
     \Norm{\Nabla{s}\p_su}_q
     +c_4\Norm{\p_su}_q
     \Norm{\p_su}_\infty
     \right)\\
    &\ge \Norm{\p_su}_2^2
     -(2\eps_\nu+c_2\abs{\sigma})
     (c_5+c_3^2c_4)
\end{split}
\end{equation*}
for some constant $c_4=c_4(a_0,\sigma_0)>0$.
The last step is
by~(\ref{eq:bd1}) with constant $c_3$.
We also used that
$\norm{\Nabla{s}\p_su}_q\le c_5$
for some positive constant
$c_5=c_5(a_0)$,
which follows
from exponential decay
of $\Nabla{s}\p_su$ according to
theorem~\ref{thm:par-exp-decay}.
The energy
identity~(\ref{eq:par-energy})
shows that
$\norm{\p_su}_2^2=\mu>0$.
Now choose $\sigma_0>0$
sufficiently small
and $\nu_0$ sufficiently large
to conclude the proof of step~3.
\end{proof}

\noindent
{\bf Step 4.}
{\it We prove the claim.
}

\begin{proof}
By step~3 there exists,
for every sufficiently large $\nu$,
an element $\sigma_\nu\in[-\sigma_0,\sigma_0]$
such that
$
     \theta_\nu(\sigma_\nu)
     =0
$
and
$
     \Abs{\sigma_\nu}
     \le \eps_\nu (2c_3/\mu)
$.
Set
$\eta_\nu:=\eta(\sigma_\nu,\nu)$.
Then 
$\eta_\nu\in\im \Dd_{u^{\sigma_\nu}}^*$
again by step~3 and
$$
     \Norm{\eta_\nu}_\infty
     +\Norm{\eta_\nu}_p
     \le \eps_\nu
     \left(3+(c_2+C_0)2c_3/
     \mu\right)
     ,\qquad
    \Norm{\eta_\nu}_\Ww
    \le C,
$$
by step~2.
This proves~(\ref{eq:IFT-limit}),
hence the claim, and therefore
theorem~\ref{thm:modified-IFT}.
\end{proof}

\section{Unique Continuation}
\label{sec:UC}

To prove
unique continuation
for the nonlinear heat equation
we slightly extend
a result of Agmon
and Nirenberg~\cite{AgNi-67}
(to the case $C_1\not=0$).
This generalization
is needed to deal with the nonlinear 
heat equation~(\ref{eq:heat}), since
here nonzero order terms
appear on the right hand side
of~(\ref{eq:E.1}).
For the linear heat equation
the original result for $C_1=0$ is sufficient.

\begin{theorem}\label{thm:AgNi}
Let $H$ be a real Hilbert space and
let $A(s):\dom A(s)\to H$ be a family
of symmetric linear operators.
Assume that $\zeta:[0,T]\to H$
is continuously differentiable in the
weak topology such that 
$\zeta(s)\in\dom A(s)$ and
\begin{equation}\label{eq:E.1}
     \Norm{\zeta^\prime(s)-A(s)\zeta(s)}
     \le c_1 \Norm{\zeta(s)}
     +C_1
     \Abs{\langle A(s)\zeta(s),\zeta(s)\rangle}^{1/2}
\end{equation}
for every $s\in[0,T]$
and two constants $c_1,C_1\ge0$.
Here
$\zeta^\prime(s)\in H$ denotes the
derivative of $\zeta$ with respect to $s$.
Assume further that the function
$s\mapsto \langle \zeta(s),A(s)\zeta(s)\rangle$
is also continuously differentiable
and satisfies
\begin{equation}\label{eq:E.2}
     \frac{d}{ds}
     \langle\zeta,A\zeta\rangle
     -2\langle\zeta^\prime,A\zeta\rangle
     \ge -c_2 \Norm{A\zeta}\Norm{\zeta}
     -c_3 \Norm{\zeta}^2
\end{equation}
pointwise for every
$s\in[0,T]$
and constants $c_2,c_3>0$.
Then the following holds.

\noindent
{\rm (1)}
     If $\zeta(0)=0$
     then $\zeta(s)=0$ for all $s\in[0,T]$.

\noindent
{\rm(2)}  
     If $\zeta(0)\not=0$
     then $\zeta(s)\not=0$ for all $s\in[0,T]$
     and, moreover,
$$
     \log \Norm{\zeta(s)}^2
     \ge
     \log \Norm{\zeta(0)}^2
     -\left(2
     \frac{\langle\zeta(0),A(0)\zeta(0)\rangle}
     {\norm{\zeta(0)}^2} +\frac{b}{a}
     \right) \frac{e^{as}-1}{a}
     -2c_1 s
$$
where
$a=2{C_1}^2+c_2$
and $b=4{c_1}^2+{c_2}^2/2+2c_3$.
\end{theorem}

\begin{proof}
A beautyful exposition in the case $C_1=0$
was given by Salamon
in~\cite[appendix~E]{Sa-SW}
in the case $C_1=0$.
It generalizes easily.
A key step is
to prove that
the function
$$
     \varphi(s)
     :=
     \log \norm{\zeta(s)}^2
     -\int_0^s
     \frac{2\langle\zeta(\sigma),
     \zeta^\prime(\sigma)-A(\sigma)\zeta(\sigma)\rangle}
     {\norm{\zeta(\sigma)}^2} d\sigma
$$
satisfies the differential
inequality
\begin{equation}\label{eq:diff-ineq}
     \varphi^{\prime\prime}
     +a\Abs{\varphi^\prime}+b
     \ge 0
\end{equation}
for two constants $a,b>0$.

In~\cite{Sa-SW}
it is shown that
assumption~(\ref{eq:E.2})
implies the inequality
$$
     \varphi^{\prime\prime}
     \ge
     2\Norm{\eta-\langle\eta,\xi\rangle\xi}^2
     -\frac{2\Norm{\zeta^\prime-A\zeta}^2}{\Norm{\zeta}^2}
     -2c_2\Norm{\eta}
     -2c_3
$$
where
$$
     \xi
     :=\frac{\zeta}{\Norm{\zeta}},\qquad
     \eta:=\frac{A\zeta}{\Norm{\zeta}}.
$$
Now it follows by
assumption~(\ref{eq:E.1})
that
$$
     \frac{2\Norm{\zeta^\prime-A\zeta}^2}{\Norm{\zeta}^2}
     \le 4{c_1}^2
     +4{C_1}^2
     \frac{\Abs{\langle A\zeta,\zeta\rangle}}
     {\Norm{\zeta}^2}
     =4{c_1}^2
     +4{C_1}^2\Abs{\langle \eta,\xi\rangle}
$$
and therefore
$$
     \varphi^{\prime\prime}
     \ge
     2\Norm{\eta-\langle\eta,\xi\rangle\xi}^2
     -4{c_1}^2
     -4{C_1}^2\Abs{\langle \eta,\xi\rangle}
     -2c_2\Norm{\eta}
     -2c_3.
$$
To obtain the
inequality~(\ref{eq:diff-ineq})
it remains to prove that
$$
     2\Norm{\eta-\langle\eta,\xi\rangle\xi}^2
     -4{c_1}^2
     -4{C_1}^2\Abs{\langle \eta,\xi\rangle}
     -2c_2\Norm{\eta}
     -2c_3
     \ge-a\Abs{\varphi^\prime}-b.
$$
Since
$\varphi^\prime=2\langle \xi,\eta\rangle$
this is equivalent to
$$
     c_2\Norm{\eta}
     \le
     \Norm{\eta-\langle\eta,\xi\rangle\xi}^2
     +(a-2{C_1}^2)\Abs{\langle \eta,\xi\rangle}
     +(b/2-2{c_1}^2-c_3).
$$
Abbreviate
$$
     u:=\Norm{\eta-\langle\eta,\xi\rangle\xi}^2,\qquad
     v:=\Abs{\langle \eta,\xi\rangle},
$$
then $\norm{\eta}^2=u^2+v^2$ and
the desired inequality
has the form
$$
     c_2\sqrt{u^2+v^2}
     \le u^2 +(a-2{C_1}^2)v
     +(b/2-2{c_1}^2-c_3).
$$
Since
$
  c_2\sqrt{u^2+v^2}
  \le
  c_2u+c_2v
  \le
  u^2+c_2v+{c_2}^2/4
$
this is
satisfies with
$$
     a=2{C_1}^2+c_2,\qquad
     b=4{c_1}^2+{c_2}^2/2+2c_3.
$$
This proves the
inequality~(\ref{eq:diff-ineq}).
The remaining
part of the proof
of theorem~\ref{thm:AgNi}
carries over from~\cite{Sa-SW}
unchanged.
\end{proof}

\subsection{Linear equation}
\label{subsec:UC-L}

Unique continuation
for the linearized
heat equation is used to prove
proposition~\ref{prop:onto-universal}
on transversality of 
the universal section
and the unstable manifold
theorem~\ref{thm:unstable-mf}.

\begin{proposition}\label{prop:UC-L}
Fix a perturbation
$\Vv:\Ll M\to\R$ 
that satisfies~{\rm (V0)--(V2)}
and two constants $a<b$.
Let $u:[a,b]\times S^1\to M$ be
a smooth map and
let $\xi=\xi(s,t)$ be a smooth
vector field along $u$ such that
$\Dd_u\xi=0$ or $\Dd_u^*\xi=0$,
where the operators are defined
by~(\ref{eq:lin-op})
and~(\ref{eq:formal-adjoint}),
respectively.
Abbreviate $\xi(s,\cdot)$ by $\xi(s)$.
Then the following is true.

\vspace{.2cm}
{\rm (a)}
     If $\xi(s_*)=0$ for some $s_*$,
     then $\xi(s)=0$ for all $s\in[a,b]$.

\vspace{.2cm}
{\rm (b)}  
     If $\xi(s_*)\not=0$ for some $s_*$,
     then $\xi(s)\not=0$ for all $s\in[a,b]$.
\end{proposition}

\begin{proof}
We represent $\Dd_u$ by the operator 
$
     D_{A+C}=\frac{d}{ds}+A(s)+C(s)
$
given by~(\ref{eq:D_A}).
Here the family $A(s)$
consists of self-adjoint
operators on the 
Hilbert space $H:=L^2(S^1,\R^n)$
with dense domain $W$;
see~(ii) and~(iv)
in section~\ref{sec:linearized}.
The space $W$ has been defined
prior to~(\ref{eq:D_A}).
Recall that if the vector bundle 
$u^*TM\to [a,b]\times S^1$
is trivial then $W=W^{2,2}(S^1,\R^n)$
and otherwise some boundary
condition enters.
In either case 
$W=:\dom A(s)$ is independent of $s$.

{\rm (b)}
Let $\xi\in\ker D_{A+C}$
satisfy $\xi(s_*)\not=0$.
Assume by contradiction
that $\xi(s_0)=0$ for some
$s_0\in[a,b]$.
Now if $s_0>s_*$,
then replace $\xi(s)$ by
$\xi(s+s_*)$ and set $T=b-s_*$
and $s_1=s_0-s_*$,
otherwise
replace $\xi(s)$ by
$\xi(-s+s_*)$ and set $T=-a+s_*$
and $s_1=-s_0+s_*$.
Hence we may assume without loss
of generality that $\xi\in\ker D_{A+C}$
maps $[0,T]$ to $H$ and satisfies
$\xi(0)\not=0$ and $\xi(s_1)=0$
for some $s_1\in(0,T]$.

Next we check that the conditions
in theorem~\ref{thm:AgNi}
are satisfied:
Firstly, the vector field $\xi$ 
is smooth by assumption.
Secondly, the family $A(s)$ consists 
of self-adjoint
operators by~(ii)
in section~\ref{sec:linearized}.
Thirdly, the function
$s\mapsto \langle \xi(s),A(s)\xi(s)\rangle$
is continuously differentiable. Here
we use the first condition in axiom~(V2),
which tells that the Hessian
$\Hh_\Vv$ is a zeroth order operator,
and the fact that by compactness 
of the domain the vector fields
$\p_tu$, $\p_su$, $\Nabla{t}\p_su$,
and $\Nabla{t}\Nabla{t}\p_su$
are bounded in
$L^\infty([0,T]\times S^1)$
by a constant $c_T>0$.
Next assumption~(\ref{eq:E.1})
is satisfied with $C_1=0$, because
$$
     \Norm{\xi^\prime(s)-A(s)\xi(s)}
     =\Norm{C(s)\xi(s)}
     \le c_T^\prime \Norm{\xi(s)}
$$
where the constant
$c_T^\prime=\sup_{[0,T]\times S^1} 
\norm{C(s,t)}_{\Ll(\R^n)}$
is finite by compactness
of the domain.
To verify the 
inequality~(\ref{eq:E.2})
note that its left
hand side is given by
$\langle\xi(s),A^\prime(s)\xi(s)\rangle$;
see~\cite[Rmk. in sec.~1]{AgNi-67}
and~\cite[Rmk.~F.3]{Sa-SW}.
Now
\begin{equation*}
\begin{split}
     \langle\xi(s),A^\prime(s)\xi(s)\rangle
     &\ge -\Norm{\xi(s)}
      \Norm{A^\prime(s)\xi(s)}\\
     &\ge -c_T^{\prime\prime}\Norm{\xi(s)}
      \left( \Norm{\xi(s)}
      +\Norm{\p_t\xi(s)} \right).
\end{split}
\end{equation*}
where the second step is by
straightforward calculation
of $A^\prime(s)$.
Replacing $\norm{\p_t\xi(s)}$ 
according to the elliptic
estimate for $A(s)$
yields~(\ref{eq:E.2}).

Now the Agmon-Nirenberg
theorem~\ref{thm:AgNi}
applies and 
part~{\rm (2)}
tells that
$\xi(s)\not=0$
for all $s\in[0,T]$.
This contradiction
proves~(b) for elements
in the kernel of $\Dd_u$.
The same argument
covers the case 
of the operator $\Dd_u^*$, 
since it is represented by
$-D_{-A-C}$ according to
remark~\ref{rmk:adjoint_fredholm_p=2}.

{\rm (a)}
This follows either by
a time reversing argument 
(see proof of the Agmon-Nirenberg 
Theorem in~\cite{Sa-SW})
and application of~(b)
or by a line of argument
analoguous to the proof of~(b)
given above, where in the 
final step
part~(2) of theorem~\ref{thm:AgNi}
is replaced by part~(1).
\end{proof}

\subsection{Nonlinear equation}
\label{subsec:UC-NL}

Unique continuation
for the nonlinear
heat equation is used to prove
the unstable manifold
theorem~\ref{thm:unstable-mf}.

\begin{theorem}[Unique Continuation
for compact cylindrical domains]
\label{thm:UC-NL}
Fix two constants $a<b$
and a perturbation
$\Vv:\Ll M\to\R$
that satisfies~{\rm (V0)} and~{\rm (V1)}.
If two smooth solutions
$u,v:[a,b]\times S^1\to M$
of the heat equation~(\ref{eq:heat})
coincide along one loop,
then $u=v$.
\end{theorem}

\begin{proof}
Abbreviate
$u_s=u(s,\cdot)$
and assume $u_\sigma=v_\sigma:S^1\to M$
for some $\sigma\in[a,b]$.
Moreover, we may assume without loss
of generality that
$\p_su$ is nonzero
at some point $(s,t)$.
Otherwise $u$
coincides with a critical
point $x$ of the action
functional $\Ss_\Vv$ and,
since $v_\sigma=u_\sigma=x$,
so does $v$ and we are done.
It follows similarly that $\p_s v$
is nonzero somewhere.
Hence
\begin{equation}\label{eq:delta}
     \delta
     :=
     \frac{\iota}{2+\Norm{\p_su}_\infty+\Norm{\p_sv}_\infty}
     \in(0,\iota/2).
\end{equation}
Here $\iota>0$ denotes
the injectivity radius
of our compact
Riemannian manifold.

The first step is to prove that
the restrictions
of $u$ and $v$ to
$[\sigma-\delta,\sigma+\delta]\times S^1$
are equal.
(In fact we should take the intersection
with $[a,b]\times S^1$, but \emph{suppress
this throughout}
for simplicity of notation.)
The key idea is to express the
difference of $u$ and $v$
near $\sigma$ with respect
to geodesic normal
coordinates based at $u_\sigma$
and show that this difference
$\zeta$ and a suitable operator $A$
satisfy the requirements
of theorem~\ref{thm:AgNi}
(with nonzero constant $C_1$).
Then, since $\zeta(\sigma)=0$,
part~(1) of the theorem
shows that $\zeta=0$ and therefore
$u=v$ on 
$[\sigma-\delta,\sigma+\delta]\times S^1$.

Once the above has been achieved
we successively
restrict $u$ and $v$ 
to cylinders of the form
$[\sigma+(2k-1)\delta,
\sigma+(2k+1)\delta]\times S^1$,
where $k\in\Z$,
and use that $u$ and $v$ coincide
along one of the two boundary components
to conclude by the same argument as above
that $u=v$ on each of these cylinders.
Due to compactness of $Z$
the same constants $c_1$ and $C_1$
can be chosen in~(\ref{eq:E.1})
for \emph{all} cylinders.
After finitely many steps
the union of these cylinders covers
$[a,b]\times S^1$ and this
proves the theorem.

It remains to carry out the 
first step. Consider the
interval 
$I=[\sigma-\delta,\sigma+\delta]$ and
the cylinder
$$
     Z=I\times S^1
     =[\sigma-\delta,\sigma+\delta]\times S^1.
$$
From now on $u$ and $v$ are restricted
to the domain $Z$.
Note that the Riemannian distance
between $u(\sigma,t)$ and 
$u(s,t)$ is less than half
the injectivity radius $\iota$
for every $(s,t)\in Z$.
Hence the identities
$$
     u(s,t)
     =\exp_{u(\sigma,t)}\xi(s,t),
     \qquad
     v(s,t)
     =\exp_{u(\sigma,t)}\eta(s,t)
$$
for $(s,t)\in Z$
uniquely determine
smooth families
of vector fields $\xi$
and $\eta$
along the loop $u_\sigma$.
The domain of $\xi$ and $\eta$ is $Z$,
they satisfy the estimates
$$
     \Norm{\xi}_\infty<\frac{\iota}{2},\qquad
     \Norm{\eta}_\infty<\frac{\iota}{2},
$$
and $\xi(\sigma,t)=0=\eta(\sigma,t)$
for every $t\in S^1$.
Moreover, since $\xi(s,t)$ and $\eta(s,t)$
live in the same tangent
space $T_{u(\sigma,t)}M$
their difference
$\zeta=\xi-\eta$
is well defined.
\\
Now consider the
Hilbert space
$H=L^2(S^1,{u_\sigma}^*TM)$
and the symmetric
differential operator
$A=\Nabla{t}\Nabla{t}$
with domain
$W=W^{2,2}(S^1,{u_\sigma}^*TM)$.
Here $\Nabla{t}$ denotes the
covariant derivative
along the loop $u_\sigma$.
Hence the operator $A$
is independent of $s$ and 
condition~(\ref{eq:E.2})
in the Agmon-Nirenberg
theorem~\ref{thm:AgNi}
is vacuous.
If we can verify
condition~(\ref{eq:E.1})
as well, then
$\zeta(\sigma)=0$
implies that $\zeta(s)=0$
for every $s\in I$
by theorem~\ref{thm:AgNi}~(1).
Since $\zeta$ is smooth,
this means that on $Z$
we have $\xi=\eta$ pointwise
and therefore $u=v$.
\\
It remains
to verify~(\ref{eq:E.1}).
Use~(\ref{eq:exponential-identity})
to obtain the identities
\begin{equation}\label{eq:exp-coord}
\begin{split}
     \p_su
    &=E_2(u_\sigma,\xi)\p_s\xi\\
     \Nabla{t}\p_tu
    &=E_{11}(u_\sigma,\xi)\bigl(\p_tu_\sigma,\p_tu_\sigma\bigr)
     +2E_{12}(u_\sigma,\xi)
     \bigl(\p_tu_\sigma,\Nabla{t}\xi\bigr)\\
    &\quad
     +E_1(u_\sigma,\xi)\Nabla{t}\p_tu_\sigma
     +E_{22}(u_\sigma,\xi)
     \bigl(\Nabla{t}\xi,\Nabla{t}\xi\bigr)
     +E_2(u_\sigma,\xi)\Nabla{t}\Nabla{t}\xi
\end{split}
\end{equation}
pointwise for $(s,t)\in Z$
and similarly for $v$ and $\eta$.
To obtain the second identity
we used the symmetry
property~(\ref{eq:E12-symmetry}) of $E_{12}$.
Now consider the 
heat equation~(\ref{eq:heat})
and replace $\p_su$ and 
$\Nabla{t}\p_tu$ according
to~(\ref{eq:exp-coord}),
then solve for 
$\p_s\xi-\Nabla{t}\Nabla{t}\xi$.
Do the same for $v$ and $\eta$
to obtain a similar
expression for 
$-\p_s\eta+\Nabla{t}\Nabla{t}\eta$.
Add both expressions to get
the pointwise identity
\begin{equation*}
\begin{split}
    &\bigl(\p_s-\Nabla{t}\Nabla{t}\bigr)
     \bigl(\xi-\eta\bigr)\\
    &=\left(E_2(u_\sigma,\xi)^{-1}E_{11}(u_\sigma,\xi)
     -E_2(u_\sigma,\eta)^{-1}E_{11}(u_\sigma,\eta)\right)
     \bigl(\p_tu_\sigma,\p_tu_\sigma\bigr)\\
    &\quad
     +\left(E_2(u_\sigma,\xi)^{-1}E_1(u_\sigma,\xi)
     -E_2(u_\sigma,\eta)^{-1}E_1(u_\sigma,\eta)\right)
     \Nabla{t}\p_tu_\sigma\\
    &\quad
     +2\left(E_2(u_\sigma,\xi)^{-1}
     E_{21}(u_\sigma,\xi)\Nabla{t}\xi
     -E_2(u_\sigma,\eta)^{-1}
     E_{21}(u_\sigma,\eta)\Nabla{t}\eta\right)
     \p_tu_\sigma\\
    &\quad
     +E_2(u_\sigma,\xi)^{-1}\grad\Vv(\exp_{u_\sigma}\xi)
     -E_2(u_\sigma,\eta)^{-1}\grad\Vv(\exp_{u_\sigma}\eta)\\
    &\quad
     +E_2(u_\sigma,\xi)^{-1}E_{22}(u_\sigma,\xi)
     \bigl(\Nabla{t}\xi,\Nabla{t}\xi\bigr)
     -E_2(u_\sigma,\eta)^{-1}E_{22}(u_\sigma,\eta)
     \bigl(\Nabla{t}\eta,\Nabla{t}\eta\bigr).
\end{split}
\end{equation*}
Now by compactness of 
the domain $Z$ there is a constant 
$C>0$ such that
$$
     \norm{\p_tu_\sigma}_{L^\infty(S^1)}
     \le\norm{\p_tu}_{L^\infty(Z)}<C,\qquad
     \norm{\Nabla{t}\p_tu_\sigma}_{L^\infty(S^1)}<C.
$$
Moreover, since the maps $E_i$ and $E_{ij}$
are uniformly continuous
on the radius $\iota/2$
disk tangent bundle $\Oo\subset{TM}$
in which $\xi$ and $\eta$
take their values,
there exists a constant 
$c_1>0$ such that
\begin{equation*}
\begin{split}
    &\Abs{\p_s(\xi-\eta)
     -\Nabla{t}\Nabla{t}(\xi-\eta)}\\
    &\le(c_1C^2+c_1C)
     \Abs{\xi-\eta}\\
    &\quad
     +2C
     \Abs{E_2(u_\sigma,\xi)^{-1}E_{21}(u_\sigma,\xi)\Nabla{t}\xi
     -E_2(u_\sigma,\eta)^{-1}E_{21}(u_\sigma,\eta)\Nabla{t}\eta
     }\\
    &\quad
     +\Abs{E_2(u_\sigma,\xi)^{-1}\grad\Vv(\exp_{u_\sigma}\xi)
     -E_2(u_\sigma,\eta)^{-1}\grad\Vv(\exp_{u_\sigma}\eta)}\\
    &\quad
     +\Abs{E_2(u_\sigma,\xi)^{-1}E_{22}(u_\sigma,\xi)
     \bigl(\Nabla{t}\xi,\Nabla{t}\xi\bigr)
     -E_2(u_\sigma,\eta)^{-1}E_{22}(u_\sigma,\eta)
     \bigl(\Nabla{t}\eta,\Nabla{t}\eta\bigr)}
\end{split}
\end{equation*}
pointwise for $(s,t)\in Z$.
It remains to estimate the
last three terms in the sum.
First we estimate term three.
Use linearity and the symmetry property~(\ref{eq:E12-symmetry})
of $E_{22}$ to obtain the first identity in
the pointwise estimate
\begin{equation*}
\begin{split}
    &\Abs{E_2(u_\sigma,\xi)^{-1}E_{22}(u_\sigma,\xi)
     \bigl(\Nabla{t}\xi,\Nabla{t}\xi\bigr)
     -E_2(u_\sigma,\eta)^{-1}E_{22}(u_\sigma,\eta)
     \bigl(\Nabla{t}\eta,\Nabla{t}\eta\bigr)}\\
    &=\bigl| E_2(u_\sigma,\xi)^{-1}E_{22}(u_\sigma,\xi)
     \bigl(\Nabla{t}\xi-\Nabla{t}\eta,
     \Nabla{t}\xi\bigr)\\
    &\quad+E_2(u_\sigma,\eta)^{-1}E_{22}(u_\sigma,\eta)
     \bigl(\Nabla{t}\xi-\Nabla{t}\eta,
     \Nabla{t}\eta\bigr)\\
    &\quad+\left(E_2(u_\sigma,\xi)^{-1}E_{22}(u_\sigma,\xi)
     -E_2(u_\sigma,\eta)^{-1}E_{22}(u_\sigma,\eta)\right)
     \bigl(\Nabla{t}\xi,\Nabla{t}\eta\bigr)
     \bigr| \\
    &\le\Norm{{E_2}^{-1}E_{22}}_{L^\infty(\Oo)}
     \left(\Norm{\Nabla{t}\xi}_\infty
     +\Norm{\Nabla{t}\eta}_\infty\right)
     \Abs{\Nabla{t}(\xi-\eta)}\\
    &\quad +c_1 \Norm{\Nabla{t}\xi}_\infty
     \Norm{\Nabla{t}\eta}_\infty
     \Abs{\xi-\eta}\\
    &\le \mu_1\Abs{\Nabla{t}(\xi-\eta)}
     +\mu_2\Abs{\xi-\eta}
\end{split}
\end{equation*}
where $\mu_1=2{c_2}^2C(1+c_2)$,
$\mu_2=c_1{c_2}^2C^2(1+c_2)^2$,
and the constant
$c_2>0$ is chosen
sufficiently large such that
for $j=0,1$ we have
$$
     \Norm{E_j}_{L^\infty(\Oo)}
     +\Norm{{E_2}^{-1}}_{L^\infty(\Oo)}
     +\Norm{{E_2}^{-1}E_{22}}_{L^\infty(\Oo)}
     +\Norm{{E_2}^{-1}E_{21}}_{L^\infty(\Oo)}
     \le c_2.
$$
Moreover, we used that
by the first identity 
in~(\ref{eq:exponential-identity})
$$
     \Nabla{t}\xi
     =E_2(u_\sigma,\xi)^{-1}\left(
     \p_tu
     -E_1(u_\sigma,\xi)\p_tu_\sigma
     \right).
$$
Hence
$\norm{\Nabla{t}\xi}_\infty
\le c_2C(1+c_2)$ and similarly
for $\Nabla{t}\eta$.
Next we estimate term one.
Replace $\Nabla{t}\xi$ by
$\Nabla{t}\xi-\Nabla{t}\eta
+\Nabla{t}\eta$,
then similarly as above
we obtain that 
\begin{equation*}
\begin{split}
    &2C
     \Abs{E_2(u_\sigma,\xi)^{-1}E_{21}(u_\sigma,\xi)\Nabla{t}\xi
     -E_2(u_\sigma,\eta)^{-1}E_{21}(u_\sigma,\eta)\Nabla{t}\eta
     }\\
    &\le 2c_2C\Abs{\Nabla{t}(\xi-\eta)}
     +2c_1c_2C^2(1+c_2)\Abs{\xi-\eta}
\end{split}
\end{equation*}
pointwise for $(s,t)\in Z$.
Next rewrite term two
setting $X:=\eta-\xi$ and replacing $\eta$
accordingly to obtain pointwise
at $(s,t)\in Z$ the identity
\begin{equation*}
\begin{split}
    &E_2(u_\sigma,\xi)^{-1}
     \grad\Vv(\exp_{u_\sigma}\xi)
     -E_2(u_\sigma,\xi+X)^{-1}
     \grad\Vv(\exp_{u_\sigma}\xi+X)\\
    &=:f(X)\\
    &=f(0)+\frac{d}{d\tau} f(\tau X)\\
    &=\frac{d}{d\tau}\left(
     E_2(u_\sigma,\xi+\tau X)^{-1}
     \grad\Vv(\exp_{u_\sigma}\xi+\tau X)
     \right)
\end{split}
\end{equation*}
for some $\tau\in[0,1]$.
Since $f(0)=0$, this implies that
\begin{equation*}
\begin{split}
     \Abs{f(X)}
    &\le\Norm{{E_2}^{-1}E_{22}}_{L^\infty(\Oo)}\Abs{X}\cdot
     \Norm{{E_2}^{-1}}_{L^\infty(\Oo)}
     \Abs{\grad\,\Vv(\exp_{u_\sigma}(\xi+\tau X))}\\
    &\quad+\Norm{{E_2}^{-1}}_{L^\infty(\Oo)}
     \Abs{\Nabla{\tau}
     \grad\,\Vv(\exp_{u_\sigma}(\xi+\tau X))}\\
    &\le c_2^2C_0\Abs{X}
     +c_2^2C_1\left(\Abs{X}+\Norm{X_s}_{L^1(S^1)}
     \right)
\end{split}
\end{equation*}
pointwise at $(s,t)\in Z$.
Here $C_0$ and $C_1$ denote the constants
in axiom~(V0) and~(V1), respectively.
To obtain the final step we applied
the first estimate in axiom~(V1) to the curve
$\tau\mapsto \exp_{u_\sigma}(\xi_s+\tau X_s)$ 
in the loop space $\Ll M$.
Now replace $X$ by $\eta-\xi$.

Putting things together
we have proved that due
to compactness of the domain $Z$
there exists a positive constant
$\mu=\mu(Z,g)$ such that
for every $s\in I$
$$
     \Norm{\zeta^\prime(s)-A\zeta(s)}
     \le \mu\left(\Norm{\zeta(s)}
     +\Norm{\Nabla{t}\zeta(s)}\right).
$$
Here the norm is in $L^2(S^1,{u_\sigma}^*TM)$.
Now by integration by parts
$$
     \Norm{\Nabla{t}\zeta}^2
     =\langle\Nabla{t}\zeta,\Nabla{t}\zeta\rangle
     =-\langle A\zeta,\zeta\rangle
     \le\Abs{\langle A\zeta,\zeta\rangle}.
$$
Hence~(\ref{eq:E.1})
is satisfied and this concludes
the proof of theorem~\ref{thm:UC-NL}.
\end{proof}

In the proof of
the unstable manifold
theorem~\ref{thm:unstable-mf}
we use backward unique continuation
for the nonlinear
heat equation.

\begin{theorem}[Forward and backward unique continuation]
\label{thm:UC-NL-noncompact}
Fix a perturbation $\Vv:\Ll M\to\R$ that 
satisfies~{\rm (V0)--(V1)}.
\begin{enumerate}
\item[\rm\bfseries(F)]
  Let $u$ and $v$
  be smooth solutions
  of the heat equation~(\ref{eq:heat})
  defined on the
  forward halfcylinder
  $[0,\infty)\times S^1$.
  If $u$ and $v$ agree 
  along the loop at $s=0$,
  then $u=v$.
\item[\rm\bfseries(B)]
  Let $u$ and $v$
  be smooth solutions
  of the heat equation~(\ref{eq:heat})
  defined on the backward halfcylinder
  $(-\infty,0]\times S^1$.
  Assume further that
  $$
     \sup_{s\in(-\infty,0]}
     \Ss_\Vv\bigl(u(s,\cdot)\bigr)\le c_0,\qquad
     \sup_{s\in(-\infty,0]}
     \Ss_\Vv\bigl(v(s,\cdot)\bigr)\le c_0,
  $$
  for some constant $c_0>0$.
  Then the following is true.
  If $u$ and $v$ agree 
  along the loop at $s=0$,
  then $u=v$.
\end{enumerate}
\end{theorem}

\begin{proof}
The idea is the same
as in the proof
of theorem~\ref{thm:UC-NL},
namely to decompose
the halfcylinder into
small cylinders of width
$\delta$ and then
show $u=v$ on each piece
(by the method
developed in the first step
of the proof of theorem~\ref{thm:UC-NL}).
The only additional
problem is noncompactness
of the domain.
One way to deal with this 
is to choose the same 
width for each piece
(in order to arrive at any given
time $s$ in finitely many steps).
Here we need uniform
bounds for $\abs{\p_su}$
and $\abs{\p_sv}$.
Once we have these
we can define $\delta$
again by~(\ref{eq:delta}).
Check the proof of
theorem~\ref{thm:UC-NL}
to see that the only further ingredients
in proving $u=v$ on each small cylinder
are uniform bounds for
the first two $t$-derivatives
of $u$ and of $v$.
Hence to complete the proof
it remains to show that
$$
     \Norm{\p_su}_\infty
     +\Norm{\p_tu}_\infty
     +\Norm{\Nabla{t}\p_tu}_\infty
     +\Norm{\p_sv}_\infty
     +\Norm{\p_tv}_\infty
     +\Norm{\Nabla{t}\p_tv}_\infty
     \le C
$$
for some constant $C>0$.

{\bf ad (F)} 
Let $C_0$ be the constant
in axiom~(V0) and observe that
$\Ss_\Vv\ge -C_0$.
Now by theorem~\ref{thm:gradient}
with constant $C_1$,
more precisely, by checking its proof
\begin{equation*}
\begin{split}
     \Abs{\p_su(s,t)}^2
    &\le C_1 E_{[s-1,s]}(u)\\
    &=C_1\left(\Ss_\Vv(u_{s-1})
     -\Ss_\Vv(u_s)\right)\\
    &\le C_1\left(\Ss_\Vv(u_0)+C_0\right)
\end{split}
\end{equation*}
for $(s,t)\in [1,\infty)\times S^1$.
In the second and the last step
we used that $u$ is a negative gradient
flow line and the action
decreases along $u$.
Note that the proof of
theorem~\ref{thm:gradient}
shows that
the estimate at a point depends
on its past.
This is why we get
the above estimate only on
$[1,\infty)\times S^1$.
However, the missing part
$[0,1]\times S^1$ is compact
and $u$ is smooth.
Hence $\norm{\p_su}_\infty\le C$
and
$$
     \Norm{\Nabla{t}\p_tu}_\infty
     \le \Norm{\p_su}_\infty
     +\Norm{\grad \Vv(u)}_\infty
     \le C+C_0.
$$
Here we used the 
heat equation~(\ref{eq:heat})
and axiom~(V0)
with constant $C_0$.
It follows similarly
by (checking the proof of)
theorem~\ref{thm:apriori-t}
that $\abs{\p_tu(s,t)}$ is
uniformly bounded
on $[1,\infty)\times S^1$.
The corresponding estimates
for $v$ are analoguous.

{\bf ad (B)} The proof
of the $L^\infty$ estimates
follows the same steps
as in~(F).
We even get
all estimates
right away on the whole
backward halfcylinder,
because this halfcylinder
contains the past of
each of its points.
\end{proof}

\section{Transversality}
\label{sec:transversality}

In section~\ref{subsec:perturbations-banach}
we construct a separable Banach space $Y$ of
abstract perturbations satisfying
axioms~{\rm (V0)--(V3)}.
In section~\ref{subsec:transversality-proof}
we fix a perturbation
$\Vv$ such that~{\rm (V0)--(V3)} hold
and $\Ss_\Vv$ is Morse.
We choose a closed $L^2$ neighborhood $U$
of the critical points of the function $\Ss_\Vv$
and define the subspace $Y(\Vv,U)\subset Y$
of those perturbations supported away from $U$.
Then, given a regular value $a$ of $\Ss_\Vv$,
we define a separable Banach manifold
$\Oo^a=\Oo^a(\Vv,U)$ of admissible perturbations.
In fact $\Oo^a$ is the open ball 
about zero in the Banach space
$Y(\Vv,U)$ for some sufficiently small radius $r^a$.
For any admissible perturbation $v$ it holds that
$\Pp^a(\Vv)=\Pp^a(\Vv+v)$ -- in particular
$a$ is also a regular value of $\Ss_{\Vv+v}$ --
and the sublevel sets $\{\Ss_\Vv\le a\}$ and 
$\{\Ss_{\Vv+v}\le a\}$  are homologically equivalent.
For such a triple $(\Vv,U,a)$
we prove in section~\ref{subsec:onto-universal}
that there is a residual
subset $\Oo^a_{reg}\subset\Oo^a$
of regular perturbations $v$.
These, in addition, have
the property that
the perturbed functional
$\Ss_{\Vv+v}$ is Morse--Smale
below level $a$.
The crucial step is to prove
proposition~\ref{prop:onto-universal}
on surjectivity of the universal
section $\Ff$.
Here unique continuation
for the linear heat equation enters.
A further key ingredient
in the 'no return' part of the proof
is the (negative) gradient flow
property which implies that the functional is
strictly decreasing along nonconstant
heat flow solutions.

\subsection{The universal Banach space of perturbations}
\label{subsec:perturbations-banach}

We fix, once and for all, the 
following data.
\begin{enumerate}
\item[\rm\bfseries a)]
     A dense sequence 
     $\bigl( x_i\bigr)_{i\in\N}$
     in $\Ll M=C^\infty(S^1,M)$.
\item[\rm\bfseries b)]
     For every $x_i$ a dense sequence
     $\bigl( \eta^{ij}\bigr)_{j\in\N}$
     in $C^\infty(S^1,x_i^*TM)$.
\item[\rm\bfseries c)]
     A smooth cutoff function
     $\rho:\R\to[0,1]$ such that $\rho=1$
     on $[-1,1]$ and $\rho=0$
     outside $[-4,4]$
     and such that 
     $\norm{\rho^\prime}_\infty<1$.
     Then set
     $\rho_{1/k}(r)=\rho(rk^2)$
     for $k\in\N$
     (figure~\ref{fig:fig-rhok}).
\end{enumerate}
Moreover, let $\iota>0$
denote the injectivity radius
of the closed Riemannian 
manifold $M$ and fix
a smooth cutoff function $\beta$
such that $\beta=1$
on $[-(\iota/2)^2,(\iota/2)^2]$
and $\beta=1$ outside
$[-\iota^2,\iota^2]$ 
(figure~\ref{fig:fig-cutb}).
\begin{figure}[h]
\begin{minipage}[b]{.47\linewidth}
  \centering
  \epsfig{figure=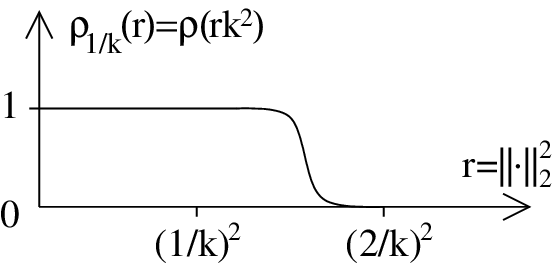}
  \caption{The cutoff function $\rho_{1/k}$} 
  \label{fig:fig-rhok}
\end{minipage}
\hfill
\begin{minipage}[b]{.47\linewidth}
  \centering
  \epsfig{figure=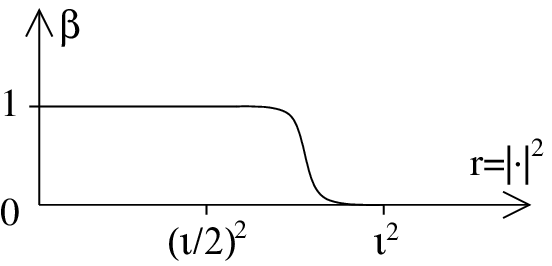}
  \caption{The cutoff function $\beta$} 
  \label{fig:fig-cutb}
\end{minipage}
\hfill
\end{figure}
\\
Then for any choice
of $i,j,k\in\N$
there is a smooth
function on the loop space 
given by
\begin{equation}\label{eq:perturbation}
     \Vv_\ell(x)
     =\Vv_{ijk}(x)
     =\rho_{1/k}\left(\Norm{x-x_i}_{L^2}^2\right)
     \int_0^1 V^{ij}(t,x(t))\, dt,
\end{equation}
where $V^{ij}$ is
the smooth function
on $S^1\times M$
defined by
\begin{equation*}
     V^{ij}(t,q):=
     \begin{cases}
       \beta\bigl(\abs{\xi_q^i(t)}^2\bigr)
       \;\big\langle\xi_q^i(t),
       \eta^{ij}(t)\big\rangle
       &\text{, $\abs{\xi_q^i(t)}<\iota$,}
       \\
       0
       &\text{, else.}
     \end{cases}
\end{equation*}
Here the vector $\xi_q^i(t)$
is determined by the identity 
$$
     q=\exp_{x_i(t)} \xi_q^i(t)
$$
whenever the Riemannian distance
between $q$ and $x_i(t)$
is less than $\iota$.
To simplify notation
we fixed a bijection
$\ell:\N^3\to\N_0$.
Observe that 
the support of $\Vv_{ijk}$
is contained in the $L^2$ ball
of radius $2/k$ about $x_i$.
Each function $\Vv_\ell:\Ll M\to\R$
is uniformly continuous with respect
to the $C^0$ topology
and satisfies~{\rm (V0)--(V3)}.
This follows by compactness of $M$,
smoothness of the potential $V$,
and by the identity
\begin{equation*}
\begin{split}
     \left\langle\grad\Vv(u),\p_su\right\rangle_{L^2}
    &=\frac{d}{ds} \Vv(u)\\
    &=2\rho^\prime\left(\Norm{u-x_0}_2^2\right)
     \left(\int_0^1 V_t(u(s,t))\, dt\right)
     \left\langle u-x_0,\p_su\right\rangle_{L^2}\\
    &\quad
     +\rho\left(\Norm{u-x_0}_2^2\right)
     \left\langle \nabla V(u),\p_su\right\rangle_{L^2}
\end{split}
\end{equation*}
which determines $\grad\Vv$.
Here $\R\to\Ll M:s\mapsto u(s,\cdot)$
is any smooth map.

Given $\Vv_\ell$,
we fix a constant $C_\ell^0\ge 1$
which is greater than its
constant of uniform continuity
and for which~(V0) holds true.
Then we fix a constant
$C_\ell^1\ge C_\ell^0$
for which both estimates in~(V1)
hold true and a constant
$C_\ell^2\ge C_\ell^1$
to cover the three estimates
of~(V2).
Furthermore, for every
integer $i\ge3$, we choose a constant
$C_\ell^i\ge C_\ell^{i-1}$
that covers all estimates
in~(V3) with $k^\prime+\ell^\prime=i$
(here $k^\prime$ and $\ell^\prime$ denote
the integers $k$ and $\ell$
that appear in~(V3)).
To summarize, for each
integer $\ell\ge 0$ we have 
fixed a sequence of constants
\begin{equation}\label{eq:constants}
     1
     \le C_\ell^0
     \le C_\ell^1\le...
     \le C_\ell^\ell\le...\qquad
     \forall \ell\in\N_0.
\end{equation}
The 
\emph{universal space of perturbations}
is the normed linear space
\begin{equation}\label{eq:Y}
     Y
     =\left\{
     v_\lambda
     :=\sum_{\ell=0}^\infty
     \lambda_\ell\Vv_\ell
     \left.\frac{}{}\right|\,
     \text{
     $\lambda=\left(\lambda_\ell
     \right)
     \subset\R$
     and
     $\Norm{v_\lambda}:=
     \sum_{\ell=0}^\infty\abs{\lambda_\ell}
     C_\ell^\ell <\infty$}
    \right\}.
\end{equation}

\begin{prop}\label{prop:Y}
The universal space $Y$ of perturbations
is a separable Banach space
and every $v_\lambda\in Y$
satisfies the axioms~{\rm (V0)--(V3)}.
\end{prop}

\begin{proof}
The map
$v_\lambda\mapsto
(\lambda_\ell C_\ell^\ell)_{\ell\in\N_0}$
provides an isomorphism
from $Y$ to
the separable
Banach space $\ell^1$
of absolutely summable
real sequences.
This proves that $Y$ is
a separable Banach space.
That every element
$v_\lambda=\sum\lambda_\ell\Vv_\ell$
of $Y$ satisfies~{\rm (V0)--(V3)}
follows readily from
the corresponding property of
the generators $\Vv_\ell$.
To explain the idea
we give the proof 
of the second estimate in~(V2), namely
\begin{equation*}
\begin{split}
     \Abs{\Nabla{t}\Nabla{s}\grad v_\lambda(u)}
    &\le\sum_{\ell=0}^\infty
     \Abs{\lambda_\ell}\cdot
     \Abs{\Nabla{t}\Nabla{s}\grad\Vv_\ell(u)}\\
    &\le\left(
     \Abs{\lambda_0} C_0^2
     +\Abs{\lambda_1} C_1^2
     +\sum_{\ell=2}^\infty
     \Abs{\lambda_\ell} C_\ell^2
     \right) f(u)\\
    &\le\left(
     \Abs{\lambda_0} C_0^2
     +\Abs{\lambda_1} C_1^2
     +\Norm{v_\lambda}\right) f(u)
\end{split}
\end{equation*}
for every smooth map $\R\to\Ll M:s\mapsto u(s,\cdot)$ 
and every $(s,t)\in\R\times S^1$.
We abbreviated
$f(u)=(\abs{\Nabla{t}\p_su} 
+(1+\abs{\p_tu})(\abs{\p_su}
+\norm{\p_su}_{L^1}))$.
Step two uses the second
estimate in~(V2) for each $\Vv_\ell$
with constant $C_\ell^2$.
Step three follows from
$C_\ell^k\le C_\ell^\ell$ whenever
$\ell\ge k$, see~(\ref{eq:constants}).
The remaining estimates in~{\rm (V0)--(V3)}
follow by the same argument.
Continuity of $v_\lambda$ with respect 
to the $C^0$ topology follows similarly
using uniform continuity of the functions $\Vv_\ell$.
\end{proof}

\subsection{Admissible perturbations}
\label{subsec:transversality-proof}

Throughout we fix a perturbation
$\Vv$ that satisfies~{\rm (V0)--(V3)} 
and such that
$\Ss_{\Vv}:\Ll M\to\R$ is Morse.
Denote the critical
values $c_i$ of $\Ss_\Vv$ by
$$
     c_0<c_1<c_2<\ldots<c_k<a<c_{k+1}<\ldots
$$
and recall that there is no accumulation point,
because $\Ss_\Vv$ admits only finitely many
critical points on each sublevel set.
Now fix a regular value $a>c_0$
(otherwise $\{\Ss_\Vv\le a\}=\emptyset$ 
and we are done) and let $c_k$
be the largest critical value smaller than $a$.
If there are critical values
larger than $a$ let $c_{k+1}$ be the smallest such,
otherwise set $c_{k+1}$ at the same distance
above $a$ as $c_k$ sits below $a$, that
is $c_{k+1}:=a+(a-c_k)$.
The idea to prove the transversality
theorem~\ref{thm:transversality}
is to perturb
$\Ss_{\Vv}$ outside
some $L^2$ neigborhood $U$
of its critical points
in such a way that
no new critical points arise
on the sublevel set $\{\Ss_\Vv<c_{k+1}\}$.
To achieve this we fix
for every critical
point $x$
a closed $L^2$ neighborhood
$U_x$ such that
$U_x\cap U_y=\emptyset$
whenever $x\not= y$.
This is possible, because
on any sublevel set there are
only finitely many critical points
($\Ss_{\Vv}$ is Morse and
satisfies the Palais-Smale condition;
see e.g.~\cite[app.~A]{Joa-INDEX}).
Set
\begin{equation}\label{eq:U}
     U=U(\Vv)
     :=\bigcup_{x\in\Pp(\Vv)}U_x
\end{equation}
and consider
the Banach space 
of perturbations
$Y$ given by~(\ref{eq:Y}).
We are interested
in the subset of those perturbations
supported away from $U$, namely
\begin{equation*}
     Y(\Vv,U)
     :=\left\{
     v_\lambda
     =\sum_{\ell=0}^\infty
     \lambda_\ell\Vv_\ell
     \in Y
     \left.\frac{}{}\right|\,
      \supp \Vv_\ell \cap U\not= \emptyset
      \;\;\Rightarrow\;\;
      \lambda_\ell=0
      \right\}.
\end{equation*}

\begin{lemma}\label{le:Y-U}
$Y(\Vv,U)$ is a closed
subspace of the separable Banach space $Y$.
\end{lemma}

\begin{proof}
Let $\alpha,\beta\in\R$
and let $v_\lambda$ and
$v_\mu$
be elements of $Y(\Vv,U)$.
By definition of $Y(\Vv,U)$
the following is true
for every $\ell\in\N_0$.
If $\supp \Vv_\ell \cap U\not= \emptyset$,
then $\lambda_\ell=0$
and $\mu_\ell=0$.
Hence $\alpha\lambda_\ell
+\beta\mu_\ell=0$
and therefore
$\alpha v_\lambda+
\beta v_\mu\in Y(\Vv,U)$.
To see that the
subspace $Y(\Vv,U)$ is closed
let $v_\lambda^i=\sum\lambda_\ell^i\Vv_\ell$
be a sequence in $Y(\Vv,U)$
which converges to some element
$v_\lambda=\sum\lambda_\ell\Vv_\ell$
of $Y$.
This means that
$\lambda_\ell^i\to\lambda_\ell$
as $i\to\infty$, for every $\ell$.
Now assume
$\supp \Vv_\ell \cap U\not= \emptyset$.
It follows that $\lambda_\ell^i=0$,
because $v_\lambda^i\in Y(\Vv,U)$,
and this is true for all $i$.
Hence the limit $\lambda_\ell$
is zero and therefore
$v_\lambda\in Y(\Vv,U)$.
\end{proof}

For $c_k<a<c_{k+1}$ as above set
\begin{equation}\label{eq:delta^a}
     \delta^a=\delta^a(\Vv)
     :=\frac{1}{2}\min\{a-c_k,c_{k+1}-a\}>0,
     \qquad
     a_\pm:=a\pm\delta^a.
\end{equation}
Hence the distance between
any two of the five reals
$$
     c_k<a_-<a<a_+<c_{k+1}
$$
is at least $\delta^a$.

\begin{lemma}\label{le:sublevel}
Fix a perturbation $\Vv$
that satisfies~{\rm (V0--V3)}
and assume $\Ss_\Vv$ is Morse.
Let $U$ be given by~(\ref{eq:U}).
Fix a regular value $a$ of $\Ss_\Vv$
and consider the reals $c_k$, $c_{k+1}$, 
$a_\pm$, and $\delta^a$,
defined above. Then the following is true.
If $v_\lambda\in Y(\Vv,U)$
and $\norm{v_\lambda}<\delta^a$, then
there are inclusions
\begin{equation*}
\begin{gathered}
     \left\{\Ss_\Vv\le c_k\right\}
     \subset \left\{\Ss_{\Vv+v_\lambda}\le a_-\right\}
     \subset \left\{\Ss_\Vv\le a\right\}
     \subset \left\{\Ss_{\Vv+v_\lambda}\le a_+\right\}
     \subset \left\{\Ss_\Vv< c_{k+1}\right\}
     \\
     \left\{\Ss_\Vv\le a_-\right\}
     \subset \left\{\Ss_{\Vv+v_\lambda}\le a\right\}
     \subset \left\{\Ss_\Vv\le a_+\right\}.
\end{gathered}
\end{equation*}
\end{lemma}

\begin{proof}
Fix $v_\lambda\in Y(\Vv,U)$
with $\norm{v_\lambda}<\delta^a$.
Observe that for each $\gamma\in\Ll M$ 
$$
     \Abs{v_\lambda(\gamma)}
     \le\sum_{\ell=0}^\infty
     \Abs{\lambda_\ell \Vv_\ell(\gamma)}
     \le\sum_{\ell=0}^\infty\Abs{\lambda_\ell} C_\ell^0
     \le\sum_{\ell=0}^\infty\Abs{\lambda_\ell} C_\ell^\ell
     =\Norm{v_\lambda}
     <\delta^a.
$$
Here we used that $v_\lambda$ is of
the form $\sum\lambda_\ell\Vv_\ell$,
axiom~(V0) with constant $C_\ell^0$
for $\Vv_\ell$, the fact that
$C_\ell^0\le C_\ell^\ell$
by~(\ref{eq:constants}),
and definition~(\ref{eq:Y})
of the norm on $Y$.
Observe further that
$$
     \Ss_{\Vv+v_\lambda}
     =\Ss_\Vv-v_\lambda.
$$
The proofs of the asserted inclusions
all follow the same pattern.
We only provide details for the last two
inclusions in the first line of the assertion
of the lemma.
Assume $\Ss_\Vv(\gamma)\le a$, then
$\Ss_{\Vv+v_\lambda}(\gamma)
=\Ss_\Vv(\gamma)-v_\lambda(\gamma)
<a+\delta^a=a_+$
where the last step is by
definition of $a_+$.
Now assume $\Ss_{\Vv+v_\lambda}(\gamma)\le a_+$,
then $\Ss_\Vv(\gamma)
\le a_++v_\lambda(\gamma)<a+2\delta^a
\le c_{k+1}$ again by definition of $a_+$.
The last step is by definition of $\delta^a$.
\end{proof}

Consider the positive constants given by
$$
     \kappa^a=\kappa^a(\Vv,U)
     :=\inf_{\gamma\in\{\Ss_\Vv<c_{k+1}\}\setminus U}
     \Norm{\grad \Ss_{\Vv}(\gamma)}_2>0
$$
and
\begin{equation}\label{eq:r^a}
     r^a=r^a(\Vv,U):=\frac12 \min\{\delta^a,\kappa^a\}>0.
\end{equation}
To prove the strict inequality $\kappa^a>0$
assume by contradiction that $\kappa^a=0$.
Then by Palais-Smale 
there exists a sequence
$(\gamma_k)\subset \{\Ss_\Vv<c_{k+1}\}\setminus U$ 
converging in the $W^{1,2}$ topology
to a critical point $x$.
It follows that $x\in U$, because
$U$ contains all critical points.
Since $W^{1,2}$ convergence implies $L^2$ convergence
and $U$ is a $L^2$ neighborhood
of the critical points,
we arrive at a contradiction
to $\gamma_k\notin U$
for every $k\in\N$.

\begin{proposition}\label{prop:sublevel}
Fix a perturbation $\Vv$
that satisfies~{\rm (V0--V3)}
and assume $\Ss_\Vv$ is Morse
and $a$ is a regular value.
Then the following is true.
If $v_\lambda\in Y(\Vv,U)$
and $\norm{v_\lambda}\le r^a$,
then
$$
     \Pp^a(\Vv)=\Pp^a(\Vv+v_\lambda),\qquad
     {\rm H}_*\left(\left\{\Ss_\Vv\le a\right\}\right)
     \cong
     {\rm H}_*\left(\left\{
     \Ss_{\Vv+v_\lambda}\le a\right\}\right).
$$
\end{proposition}

\begin{proof}
Fix $v_\lambda\in Y(\Vv,U)$
with $\norm{v_\lambda}\le\frac12\min\{\delta^a,\kappa^a\}$.
Define $a_+$ by~(\ref{eq:delta^a}).

1) We prove that $\Pp^{a_+}(\Vv)=\Pp^{a_+}(\Vv+v_\lambda)$
and this immediately implies the 
first assertion of the proposition.
On $U$ both functionals
$\Ss_\Vv$ and $\Ss_{\Vv+v_\lambda}$
coincide, because $\Ss_{\Vv+v_\lambda}=\Ss_\Vv-v_\lambda$
and $v_\lambda$ is not supported on $U$.
Now $\Ss_\Vv$ does not admit any critical point
on $\{\Ss_{\Vv+v_\lambda}<c_{k+1}\}\setminus U$
by definition of $U$.
Assume the same holds true for $\Ss_{\Vv+v_\lambda}$.
Then, since $\{\Ss_{\Vv+v_\lambda}\le a_+\}
\subset\{\Ss_\Vv< c_{k+1}\}$
by lemma~\ref{le:sublevel},
it follows that all critical point of
$\Ss_{\Vv+v_\lambda}$ below level $a_+$
are contained in $U$. But there it coincides with $\Ss_\Vv$.
Hence $\Pp^{a_+}(\Vv+v_\lambda)=\Pp^{a_+}(\Vv)$.
\\
It remains to prove the assumption. Suppose by
contradiction that there is 
a critical point $x$ of $\Ss_{\Vv+v_\lambda}$ on
$\{\Ss_{\Vv+v_\lambda}<c_{k+1}\}\setminus U$. 
Hence
$$
     0=\grad\,\Ss_{\Vv+v_\lambda}(x)
      =\grad\,\Ss_\Vv(x)
      -\grad\, v_\lambda(x)
$$
and therefore $\norm{\grad\, v_\lambda(x)}_2=
\norm{\grad\,\Ss_\Vv(x)}_2\ge\kappa^a$
by definition of $\kappa^a$.
On the other hand, 
since $v_\lambda$ is of the form
$\sum\lambda_\ell\Vv_\ell$
it follows that
\begin{equation*}
\begin{split}
     \Norm{\grad\, v_\lambda(x)}_2
    &\le\sum_{\ell=0}^\infty\Abs{\lambda_\ell}\cdot
     \Norm{\grad\,\Vv_\ell(x)}_\infty\\
    &\le\sum_{\ell=0}^\infty\Abs{\lambda_\ell}
     C_\ell^0\\
    &\le\Norm{v_\lambda}\\
    &\le \frac12 \kappa^a.
\end{split}
\end{equation*}
Here we used the inequality
$\norm{\cdot}_2\le\norm{\cdot}_\infty$,
axiom~(V0) with constant $C_\ell^0$
for $\Vv_\ell$ and the fact that
$C_\ell^0\le C_\ell^\ell$
by~(\ref{eq:constants}).
The last two lines are by definition~(\ref{eq:Y})
of the norm on $Y$
and the assumption on $\norm{v_\lambda}$.

2) We prove that 
${\rm H}_*\left(\left\{\Ss_{\Vv+v_\lambda}\le a\right\}\right)
\cong {\rm H}_*\left(\left\{
\Ss_{\Vv+v_\lambda}\le a\right\}\right)$.
Observe that all elements of the intervall $[a_-,a_+]$
are regular values of $\Ss_{\Vv+v_\lambda}$
by step~1).
Hence classical Morse theory 
for the negative $W^{1,2}$ gradient flow on the loop space
shows that
$$
     {\rm H}_*\left(\left\{
     \Ss_{\Vv+v_\lambda}\le a_-\right\}\right)
     \cong
     {\rm H}_*\left(\left\{
     \Ss_{\Vv+v_\lambda}\le a_+\right\}\right).
$$
On the other hand, using the inclusions 
provided by lemma~\ref{le:sublevel}
this isomorphism factors through
the inclusion induced homomorphisms
$$
     {\rm H}_*\left(\left\{
     \Ss_{\Vv+v_\lambda}\le a_-\right\}\right)
     \to
     {\rm H}_*\left(\left\{\Ss_\Vv\le a\right\}\right)
     \to
     {\rm H}_*\left(\left\{
     \Ss_{\Vv+v_\lambda}\le a_+\right\}\right).
$$
Therefore the first homomorphism is injective and
the second one surjective. 
Since $a$ lies in the interval of 
regular values of $\Ss_{\Vv+v_\lambda}$,
the first one leads to
an injective homomorphism
$
     {\rm H}_*\left(\left\{
     \Ss_{\Vv+v_\lambda}\le a\right\}\right)
     \to
     {\rm H}_*\left(\left\{\Ss_\Vv\le a\right\}\right)
$.
By construction the
intervall $[a_-,a_+]$
consists of regular values of $\Ss_\Vv$.
Hence the same argument using again lemma~\ref{le:sublevel}
to obtain the inclusion induced homomorphisms
$$
     {\rm H}_*\left(\left\{
     \Ss_\Vv\le a_-\right\}\right)
     \to
     {\rm H}_*\left(\left\{\Ss_{\Vv+v_\lambda}\le a\right\}\right)
     \to
     {\rm H}_*\left(\left\{
     \Ss_\Vv\le a_+\right\}\right)
$$
provides a surjection
$
     {\rm H}_*\left(\left\{
     \Ss_{\Vv+v_\lambda}\le a\right\}\right)
     \to
     {\rm H}_*\left(\left\{\Ss_\Vv\le a\right\}\right)
$.
\end{proof}

By definition the set of
\emph{admissible perturbations}
is given by the open ball 
in the Banach space $Y(\Vv,U)$
of radius $r^a$ defined in~(\ref{eq:r^a}).
We denote this set by
\begin{equation}\label{eq:OOa}
     \Oo^a=\Oo^a(\Vv,U)
     :=\left\{v_\lambda\in Y(\Vv,U):
     \Norm{v_\lambda}\le r^a
     \right\}.
\end{equation}
Since $Y(\Vv,U)$ is a separable Banach space
by lemma~\ref{le:Y-U},
the closed subset $\Oo^a$ inherits the structure
of a complete metric space.
Proposition~\ref{prop:sublevel}
then concludes the proof of the first part
of theorem~\ref{thm:transversality}.
Namely, if $v_\lambda\in\Oo^a$, then
$\Ss_\Vv$ and $\Ss_{\Vv+v_\lambda}$
have homologically equivalent sublevel sets
with respect to $a$
and the same critical points 
when restricted to these sublevel sets.

\begin{remark}\label{rmk:Oo}
If $a<b$ are regular values of $\Ss_\Vv$
and $v\in\Oo^b$ satisfies $\norm{v}\le \delta^a/2$,
then $v\in\Oo^a$. To see this note that
$\kappa^b\le \kappa^a$ and therefore
$\norm{v}\le r^b\le\kappa^b/2\le\kappa^a/2$.
Hence $\norm{v}\le\frac12\min\{\delta^a,\kappa^a\}=:r^a$.
\end{remark}

\begin{remark}
Since we chose to cut off
our abstract perturbations 
in section~\ref{sec:perturbations}
with respect to the $L^2$ norm, 
we cannot naturally control the support
of $v\in\Oo^a$ in terms
of sublevel sets of $\Ss_\Vv$.
This would be possible if we had cut off
with respect to the $W^{1,2}$ norm, because
the action functional $\Ss_\Vv$ 
is continuous with respect 
to the $W^{1,2}$ topology.
\end{remark}

\subsection{Surjectivity}
\label{subsec:onto-universal}

\begin{proof}[Proof of theorem~\ref{thm:transversality}]
Assume that the perturbation
$\Vv$ satisfies~{\rm (V0)--(V3)} 
and the function $\Ss_{\Vv}:\Ll M\to\R$ is Morse.
Consider the neighborhood $U$ of the critical points
of $\Ss_{\Vv}$ defined by~(\ref{eq:U})
in the previous section
and fix a regular value $a$ of $\Ss_{\Vv}$.
For $\Oo^a=\Oo^a(\Vv,U)$
given by~(\ref{eq:OOa}) the first 
part of theorem~\ref{thm:transversality}
is true by proposition~\ref{prop:sublevel}.
To prove the second part
fix in addition a constant $p>2$
and two critical points $x,y\in\Pp^a(\Vv)$.
We denote by $\Bb^{1,p}_{x,y}$
the smooth Banach manifold
of cylinders between $x$ and $y$
defined by~(\ref{eq:B}) in section~\ref{sec:IFT}.
This manifold is separable and admits
a countable atlas.
Now consider the smooth Banach space bundle
$$
     \Ee^p\to\Bb^{1,p}_{x,y}\times\Oo^a
$$
whose fibre over $(u,v_\lambda)$
are the $L^p$ vector fields along $u$.
The formula
\begin{equation}\label{eq:section-F}
     \Ff(u,v_\lambda)
     =\p_su-\Nabla{t}\p_tu
     -\grad \bigl(\Vv+v_\lambda\bigr)(u)
\end{equation}
defines a smooth section of
this bundle. Its zero set
$$
     \Zz=\Zz(x,y;\Vv,U,a)=\Ff^{-1}(0)
$$
is called the
{\bf universal moduli space}.
It does not
depend on $p>2$, since
all solutions of the
heat equation~(\ref{eq:heat})
are smooth by
theorem~\ref{thm:regularity}.
Now zero is a
{\bf regular value} of $\Ff$.
By definition this means that \emph{either}
there is no zero of $\Ff$ at all
\emph{or} $d\Ff(u,v_\lambda)$ is onto
and $\ker d\Ff(u,v_\lambda)$ 
admits a topological
complement,
whenever $\Ff(u,v_\lambda)=0$.
In the first case it is natural to
set $\Oo^a_{reg}(x,y)=\Oo^a$.
The second case naturally decomposes 
into two classes. 

The first class consists of constant
solutions $u$ and transversality
holds true automatically, since $\Ss_\Vv$ is Morse.
More precisely, if $x=y$, then $u(s,\cdot):=x(\cdot)$
is a zero of $\Ff$. 
In fact it solves~(\ref{eq:heat}) since
each $v_\lambda\in\Oo^a$ is supported away from
the elements of $\Pp(\Vv)$.
Now the linearization $\Dd_u$ of~(\ref{eq:heat})
reduces to the covariant Hessian $A_x$ of $\Ss_\Vv$
given by~(\ref{eq:Hessian}).
This Hessian is injective by 
the Morse assumption on $\Ss_\Vv$.
It is also surjective, because the cokernel of $A_x$
coincides with the kernel of its formal adjoint
operator with respect to the $L^2$ inner product.
But by symmetry of $A_x$ this kernel 
is equal to $\ker A_x=\{0\}$.
Hence $\Dd_u$, and therefore $d\Ff(u,v_\lambda)$,
is automatically surjective
at constant solutions. Hence
$\Oo^a_{reg}(x,x)=\Oo^a$.

The second class consists of zeroes
$(u,v_\lambda)$ of~(\ref{eq:section-F}) 
where $u$ depends on $s$.
In this case surjectivity of $d\Ff(u,v_\lambda)$
is the content
of proposition~\ref{prop:onto-universal}
below
and existence of a
topological complement
follows (see e.g.~\cite[prop.~3.3]{Joa-INDEX})
from surjectivity
and the fact that
by theorem~\ref{thm:fredholm}
and theorem~\ref{thm:par-exp-decay}
the operator 
\begin{equation}\label{eq:D_u-v}
     \Dd_u\xi
     =\Nabla{s}\xi -\Nabla{t}\Nabla{t}\xi
     -R(\xi,\p_tu)\p_tu-\Hh_{\Vv+v_\lambda}(u)\xi
\end{equation}
is Fredholm. (Note that $\Ss_{\Vv+v_\lambda}$
is Morse below level $a$
by proposition~\ref{prop:sublevel}
and the fact that $v_\lambda$ is 
not supported near
the critical points.)
Hence $\Zz$ is a 
smooth Banach manifold
by the implicit function theorem.
Consider the projection onto
the second factor
$$
     \pi:\Zz\to\Oo^a.
$$
By standard Thom-Smale
transversality theory (see
e.g.~\cite[lemma~A.3.6]{MS})
$\pi$ is a smooth Fredholm map
whose index is given by
the Fredholm index of $\Dd_u$.
This index is equal to the difference of the
Morse indices of $x$ and $y$,
again by theorem~\ref{thm:fredholm}.
Since $\Zz$ is separable and
admits a countable atlas,
we can apply the Sard-Smale
theorem~\cite{Sm73}
to countably many coordinate
representatives of $\pi$.
It follows that
the set of regular
values of $\pi$ is residual in $\Oo^a$.
We denote this set by
$\Oo^a_{reg}(x,y)$
and observe that
$$
     \Oo^a_{reg}(x,y)
     =\{v_\lambda\in\Oo^a\mid
     \text{$\Dd_u$ onto $\forall u\in
     \Mm(x,y;\Vv+v_\lambda)$}
     \}
$$
again by standard
transversality theory; see
e.g.~\cite[prop.~3.4]{Joa-INDEX}.

We define the set of {\bf regular perturbations} by
$$
     \Oo^a_{reg}
     :=\bigcap_{x,y\in\Pp^a(\Vv)}
     \Oo^a_{reg}(x,y).
$$
It is a residual subset
of $\Oo^a$, since
it consists of a finite intersection
of residual subsets.
This proves
theorem~\ref{thm:transversality}
up to proposition~\ref{prop:onto-universal}.
\end{proof}

\begin{proposition}[Surjectivity]
\label{prop:onto-universal}
Fix a perturbation $\Vv$ that satisfies~{\rm (V0)--(V3)} 
and assume $\Ss_{\Vv}$ is Morse.
Fix a regular value $a$,
critical points $x,y\in\Pp^a(\Vv)$, 
and a constant $p>2$.
Let $U$ be defined by~(\ref{eq:U})
and consider the section $\Ff$
given by~(\ref{eq:section-F}).
Then the following is true.
The linearization
$$
     d\Ff(u,v_\lambda):\Ww^{1,p}_u
     \times Y(\Vv,U)
     \to\Ll^p_u
$$
is onto at every zero 
$(u,v_\lambda)\in \Bb^{1,p}_{x,y}\times \Oo^a(\Vv,U)$ 
of the section $\Ff$.
\end{proposition}

\begin{proof}
Assume that $(u,v_\lambda)$ is a zero of $\Ff$.
The case of constant $u$ has been treated
in the proof of theorem~\ref{thm:transversality}.
Hence we assume that $u$ depends on $s\in\R$.
Since the action $\Ss_{\Vv+v_\lambda}$
decreases strictly along nonconstant zeroes
of~(\ref{eq:section-F}), it follows that
\begin{equation}\label{eq:u_s}
     c_k\ge\Ss_\Vv(x)=\Ss_{\Vv+v_\lambda}(x)
     >\Ss_{\Vv+v_\lambda}(u_s)
     >\Ss_{\Vv+v_\lambda}(y)=\Ss_\Vv(y).
\end{equation}
Here the two identities
are due to the fact that
$v_\lambda$ is not supported
near $x$ and $y$. In particular,
this shows that $x\not= y$.
Now define $1<q<2$ by $1/p+1/q=1$.
By the regularity theorem~\ref{thm:regularity}
the map $u$ is smooth
and by theorem~\ref{thm:par-exp-decay}
on exponential decay
all derivatives of $\p_su$
are bounded.
The linearization of $\Ff$ at the zero
$(u,v_\lambda)$ is given by
\begin{equation*}
\begin{split}
     d\Ff(u,v_\lambda)\;(\xi,\hat\Vv)
    &=d\Ff_{v_\lambda}(u)\;\xi
     +d\Ff_u(v_\lambda)\;\hat\Vv\\
    &=\Dd_u\xi-\grad \hat\Vv(u)
\end{split}
\end{equation*}
where 
$\Ff_{v_\lambda}(u):=\Ff(u,v_\lambda)=:\Ff_u(v_\lambda)$
and $\Dd_u$ is given by~(\ref{eq:D_u-v}).
Recall that $\Ss_{\Vv+v_\lambda}$
is Morse below level $a$
by proposition~\ref{prop:sublevel}
and the fact that $v_\lambda$ is 
not supported near
the critical points.
Hence by theorem~\ref{thm:par-exp-decay}
the Fredholm theorem~\ref{thm:fredholm}
shows that the operator $\Dd_u$
is Fredholm. 
Moreover, the second operator
$$
     Y(\Vv,U)\to\Ll^p_u\;:\; 
     \hat\Vv\mapsto -\grad \hat\Vv(u)
$$
is bounded.
To see this observe that,
since the support of $\hat\Vv$ is disjoint
to the neighborhood $U$ of $x$ and $y$,
there is a constant $T=T(u)>0$
such that $\grad \hat\Vv(u_s)=0$
whenever $\abs{s}>T$.
Now $\hat\Vv$ is of the form
$\sum_{\ell=0}^\infty \mu_\ell\Vv_\ell$.
Hence
\begin{equation*}
\begin{split}
     \bigl\|\grad \hat\Vv(u)\bigr\|_{L^p(\R\times S^1)}
    &=\left(\int_{-T}^T\Norm{\grad\hat\Vv(u_s)}_p^p ds
     \right)^{1/p}\\
    &\le \left(2T\right)^{1/p} 
     \sum_{\ell=0}^\infty \abs{\mu_\ell}\cdot
     \Norm{\grad\Vv_\ell(u_s)}_\infty\\
    &\le \left(2T\right)^{1/p} 
     \sum_{\ell=0}^\infty \abs{\mu_\ell} C_\ell^0\\
    &\le \left(2T\right)^{1/p} \bigl\|\hat\Vv\bigr\|
\end{split}
\end{equation*}
where for each $\Vv_\ell$
we used the last condition
in~(V0) with constant
$C_\ell^0\le C_\ell^\ell$.
The last step uses the 
definition~(\ref{eq:Y})
of the norm in $Y$.

Hence the range of
$d\Ff(u,v_\lambda)$ is closed
by standard arguments;
see e.g.~\cite[proposition~3.3]{Joa-INDEX}.
Therefore it suffices to prove that
it is dense. We use that density of the range
is equivalent to \emph{triviality
of its annihilator}: By definition
this means that, given $\eta\in\Ll_u^q$, then
\begin{equation}\label{eq:dense-1}
     \langle\eta,\Dd_u\xi
     \rangle=0,
     \qquad \forall\xi\in\Ww^{1,p}_u,
\end{equation}
and
\begin{equation}\label{eq:dense-2}
     \langle\eta,\grad\hat\Vv(u)
     \rangle=0,
     \qquad \forall\hat\Vv\in Y(\Vv,U),
\end{equation}
imply that $\eta=0$.

Assume by contradiction
that $\eta\in\Ll_u^q$
satisfies~(\ref{eq:dense-1})
and $\eta\not=0$.
In five steps we
derive a contradiction
to~(\ref{eq:dense-2}).
Steps~1--3 are preparatory,
in step~4 we construct
a model perturbation
$\Vv_\eps$
violating~(\ref{eq:dense-2})
and in step~5
we approximate $\Vv_\eps$
by the fundamental perturbations $\Vv_{ijk}$
of the form~(\ref{eq:perturbation}).
To start with observe
that $\eta$ is smooth
by~(\ref{eq:dense-1})
and theorem~\ref{thm:REG-L}.
Furthermore, integrating~(\ref{eq:dense-1})
by parts for 
$\xi\in C^\infty_0(\R\times S^1,u^*TM)$
shows that $\Dd_u^*\eta=0$ pointwise,
where the operator $\Dd_u^*$ arises
by replacing $\Nabla{s}$ by 
$-\Nabla{s}$ in~(\ref{eq:D_u-v}).
Throughout we use the notation
$\eta_s(t)=\eta(s,t)$. Hence
$\eta_s$ is a smooth vector field
along the loop $u_s$.

\vspace{.1cm}
\noindent
{\bf Step 1.} (Unique Continuation)
{\it $\eta_s\not=0$ and
$\p_su_s\not=0$ for every $s\in\R$.
}

\vspace{.1cm}
\noindent
Because $\eta$ is smooth, nonzero,
and $\Dd_u^*\eta=0$,
proposition~\ref{prop:UC-L} 
on unique continuation
shows that $\eta_s\not=0$
for every $s\in\R$.
Next observe that
$\p_su$ is smooth, because
$u$ is smooth,
and that
$0=\frac{d}{ds}\Ff_{v_\lambda}(u)=\Dd_u\p_su$.
Since $u$ connects different
critical points,
the derivative $\p_su$
cannot vanish identically on $\R\times S^1$.
Hence $\p_su_s\not=0$
for every $s\in\R$ by
proposition~\ref{prop:UC-L}
for $\xi(s):=\p_su_s$.

\vspace{.1cm}
\noindent
{\bf Step 2.} (Slicewise Orthogonal)
{\it $\langle\eta_s,\p_su_s\rangle=0$
for every $s\in\R$.
}

\vspace{.1cm}
\noindent
Note that here $\langle\eta_s,\p_su_s\rangle$
denotes the $L^2(S^1)$ inner product.
Now
\begin{equation*}
\begin{split}
     \frac{d}{ds}
     \langle\eta_s,\p_su_s\rangle
    &=\langle\Nabla{s}\eta_s,\p_su_s\rangle
     +\langle\eta_s,\Nabla{s}\p_su_s\rangle\\
    &=\langle-\Nabla{t}\Nabla{t}\eta_s
     -R(\eta_s,\p_tu_s)\p_tu_s
     -\Hh_{\Vv+v_\lambda}(u_s)\eta_s,\p_su_s\rangle\\
    &\quad +\langle\eta_s,
     \Nabla{t}\Nabla{t}\p_su_s
     -R(\p_su_s,\p_tu_s)\p_tu_s
     -\Hh_{\Vv+v_\lambda}(u_s)\p_su_s\rangle\\
    &=0
\end{split}
\end{equation*}
by straightforward calculation.
In the second equality
we replaced $\Nabla{s}\eta_s$
according to
the identity $\Dd_u^*\eta=0$
and $\Nabla{s}\p_su_s$
according to $\Dd_u\p_su=0$;
see~(\ref{eq:D_u-v}).
The last step is by
integration by parts,
symmetry of the Hessian $\Hh$,
and the first Bianchi identity
for the curvature operator $R$.
It follows that
$\langle\eta_s,\p_su_s\rangle$
is constant in $s$.
Now this constant, say $c$,
must be zero, because
$$
     \int_{-\infty}^\infty c \; ds
    =\int_{-\infty}^\infty
    \langle\eta_s,\p_su_s\rangle \; ds
    =\langle\eta,\p_su\rangle
$$
and the right hand side is finite,
because $\eta\in\Ll_u^q$
and $\p_su\in\Ll_u^p$
with $\frac{1}{p}+\frac{1}{q}=1$.
This proves step~2.
Note that $\eta_s$ and $\p_su_s$
are linearly independent for
every $s\in\R$ 
as a consequence of step~1 and step~2.

\vspace{.1cm}
\noindent
{\bf Step 3.} (No Return)
{\it Assume the loop
$u_{s_0}$ is different
from the asymptotic
limits $x$ and $y$
and let $\delta>0$.
Then there exists $\eps>0$
such that for every $s\in\R$
$$
     \Norm{u_s-u_{s_0}}_2<3\eps
     \quad \Longrightarrow \quad
     s\in(s_0-\delta,s_0+\delta).
$$
In words,
once $s$ leaves a given $\delta$-interval
about $s_0$ the loops
$u_s$ cannot return 
to some $L^2$ $\eps$-neighborhood
of $u_{s_0}$.
}

\vspace{.1cm}
\noindent
Key ingredients
in the proof are
smoothness of $u$,
existence of asymptotic limits,
and the gradient flow property.
Recall the footnote in remark~\ref{rmk:pert}
concerning the difference of loops $u_s-u_{s_0}$.
Now assume by contradiction
that there is
a sequence of
positive reals $\eps_i\to 0$
and a sequence of
reals $s_i$ which satisfy
$\norm{u_{s_i}-u_{s_0}}_2<3\eps_i$
and
$s_i\notin(s_0-\delta,s_0+\delta)$.
In particular, it follows that
\begin{equation}\label{eq:L2-limit}
      u_{s_i}
      \stackrel{L^2}{\longrightarrow}
      u_{s_0}\quad
      \text{as $i\to\infty$.}
\end{equation}
Assume first that the
sequence $s_i$ is unbounded.
Hence we can choose
a subsequence, without changing
notation, such that
$s_i$ converges to
$+\infty$ or $-\infty$.
In either case $u_{s_i}$ converges
to one of the critical points
$x$ or $y$ and the convergence
is in $C^0(S^1)$ by
theorem~\ref{thm:par-exp-decay}.
By~(\ref{eq:L2-limit})
and uniqueness of limits
it follows that
$u_{s_0}$ is equal to
one of the critical points $x$ or $y$
contradicting our assumption.
\\
Assume now that the
sequence $s_i$ is bounded.
Then we can choose
a subsequence, without
changing notation,
such that $s_i$
converges to some element
$s_1\notin (s_0-\delta,s_0+\delta)$.
Since $u$ is smooth,
it follows that
$u_{s_i}$ converges
to $u_{s_1}$ in $C^0(S^1)$.
Again by uniqueness of limits
$u_{s_1}=u_{s_0}$.
On the other hand,
the action functional
is strictly decreasing
along nonconstant
negative gradient flow lines.
Therefore $s_1=s_0$
and this contradiction
concludes the proof of step~3.

\vspace{.1cm}
\noindent
{\bf Step 4.}
{\it There is a time $s_0\in\R$
such that the loop $u_{s_0}$
is not element of $U$.
Moreover, there is a constant $\eps>0$
and a smooth function $\Vv_0 :\Ll M\to\R$ 
supported in the $L^2$ ball of radius 
$2\eps$ about $u_{s_0}$ such that
$$
     \Vv_0 (u_{s_0})=0,\qquad
     d\Vv_0 (u_{s_0})\eta_{s_0}=
     \Norm{\eta_{s_0}}_2^2,\qquad
     \langle\grad\Vv_0 (u),
     \eta\rangle
     \not=0.
$$
}

\noindent
The first assertion follows from
$x\not= y$ and the fact that the
closed sets
$U_z$, where $z\in\Pp(\Vv)$, 
are pairwise disjoint.
Now observe that the graph
$t\mapsto(t,u_{s_0}(t))$
of the loop $u_{s_0}$
is embedded in $S^1\times M$.
We define a smooth function $V$
on $S^1\times M$
supported near this graph 
as follows.
Denote by $\iota>0$ the injectivity
radius of the closed
Riemannian manifold $M$.
Pick a smooth cutoff function
$\beta:\R\to[0,1]$
such that $\beta=1$ on
$[-(\iota/2)^2,(\iota/2)^2]$ and
$\beta=0$ outside
$[-{\iota}^2,{\iota}^2]$;
see figure~\ref{fig:fig-cutb}.
Then define
\begin{equation}\label{eq:V}
     V_t(q):=V(t,q):=
     \begin{cases}
       \beta\bigl(\abs{\xi_q(t)}^2\bigr)\;
       \bigl\langle\xi_q(t),
       \eta_{s_0}(t)\bigr\rangle
       &\text{, $\abs{\xi_q(t)}<\iota$,}
       \\
       0
       &\text{, else,}
     \end{cases}
\end{equation}
where the vector $\xi_q(t)$ is
determined by the identity
$
     q=\exp_{u_{s_0}(t)}\xi_q(t)
$
whenever the Riemannian distance
between $q$ and $u_{s_0}(t)$
is less than $\iota$.
Note that the function $V$
vanishes on the graph of
the loop $u_{s_0}$.

Since all maps involved are
smooth, we can choose
a constant $\delta>0$ 
sufficiently small
such that for every
$s\in(s_0-\delta,s_0+\delta)$
the following is true
\begin{enumerate}
\item[i)]
     $d_{C^0}(u_s,u_{s_0})
     =\Norm{\xi_s}_\infty
     <\frac12\iota$, where
     the vector field $\xi_s$ 
     along the loop $u_{s_0}$
     is uniquely determined by
     the pointwise identity
     $
        u_s=\exp_{u_{s_0}}\xi_s
     $,
\item[ii)]
     $\langle
     E_2(u_{s_0},\xi_s)^{-1}
     \eta_s,\eta_{s_0}
     \rangle
     \ge\frac12\mu_0$,
     where $\mu_0:=\Norm{\eta_{s_0}}_2^2>0$,
\item[iii)]
     $\frac{1}{2} \mu_1
     \le
     \frac{\Norm{u_s-u_{s_0}}_2}{\abs{s-s_0}}
     \le
     \frac{3}{2} \mu_1$,
     where $\mu_1:=\Norm{\p_su_{s_0}}_2>0$.
\end{enumerate}
Recall the definition~(\ref{eq:exponential-identity}) 
of $E_2$ and the 
identities~(\ref{eq:E_ij(0)}).
For $s\in(s_0-\delta,s_0+\delta)$,
we obtain that
\begin{equation}\label{eq:dV}
\begin{split}
     dV_t(u_s)\,\eta_s
    &={\textstyle \left.\frac{d}{dr}\right|_{r=0}}
     V_t(\exp_{u_s}r\eta_s)\\
    &=2\beta^\prime(\abs{\xi_s}^2)\;
     \langle\xi_s,
     E_2(u_{s_0},\xi_s)^{-1}
     \eta_s\rangle\cdot
     \langle\xi_s,
     \eta_{s_0}\rangle\\
    &\quad+\beta(\abs{\xi_s}^2)\;
     \langle
     E_2(u_{s_0},\xi_s)^{-1}
     \eta_s,\eta_{s_0}
     \rangle\\
    &=\langle
     E_2(u_{s_0},\xi_s)^{-1}
     \eta_s,\eta_{s_0}
     \rangle
\end{split}
\end{equation}
pointwise for every $t\in S^1$.
The final step
uses~{\rm i)} and the
definition of $\beta$.
Note that
$dV_t(u_{s_0})\,\eta_{s_0}
=\abs{\eta_{s_0}}^2$ pointwise.

Integrating $V$ along a loop
defines a smooth function
on the loop space which vanishes
on $u_{s_0}$. 
To cut this function off
with respect to the
$L^2$ distance
fix a smooth
cutoff function
$\rho:\R\to[0,1]$
such that $\rho=1$
on $[-1,1]$, $\rho=0$
outside $[-4,4]$, and
$\norm{\rho^\prime}_\infty<1$.
Then, for the constant
$\delta$ fixed above,
choose $\eps>0$
according to step~3 (No Return) and
set $\rho_\eps(r)=\rho(r/\eps^2)$;
see figure~\ref{fig:fig-rhok}
for $\eps=\frac{1}{k}$.
Note that
$\norm{\rho_\eps^\prime}_\infty<\eps^{-2}$.
Observe that we can choose
$\eps>0$ smaller and the assertion
of step~3 remains true.
Now define
a smooth function
on $\Ll M$ by
\begin{equation*}
     \Vv_0 (x)
     :=\rho_\eps\left(
     \Norm{x-u_{s_0}}_2^2\right)
     \int_0^1 V(t,x(t))\, dt
\end{equation*}
where $V$ is given
by~(\ref{eq:V}).
The function $\Vv_0 $
vanishes on the loop $u_{s_0}$
and satisfies
\begin{equation*}
\begin{split}
     d\Vv_0 (u_s)\,\eta_s
    &={\textstyle \left.\frac{d}{dr}\right|_{r=0}}
     \Vv_0 (\exp_{u_s}r\eta_s)\\
    &=2\rho_\eps^\prime\bigl(
     \Norm{u_s-u_{s_0}}_2^2\bigr)\;
     \langle u_s-u_{s_0},
     \eta_s\rangle
     \int_0^1 V_t(u_s(t))\, dt
     \\
    &\quad
     +\rho_\eps\bigl(
     \Norm{u_s-u_{s_0}}_2^2\bigr)\;
     \int_0^1 dV_t(u_s(t))\,\eta_s(t)\, dt.
\end{split}
\end{equation*}
Hence
$d\Vv_0 (u_{s_0})\eta_{s_0}
=\norm{\eta_{s_0}}_2^2$
and this proves another 
assertion  of step~4.

To prove the final assertion 
of step~4 observe that
$s\notin(s_0-\delta,s_0+\delta)$
implies that
$\norm{u_s-u_{s_0}}_2\ge3\eps$,
by step~3,
and therefore
$u_s\notin\supp\,\Vv_0 $.
It follows that
\begin{equation}\label{eq:grad-V-est}
\begin{split}
     \langle\grad\,\Vv_0 (u),\eta\rangle
    &=\int_{s_0-\delta}^{s_0+\delta}
     d\Vv_0 (u_s)\eta_s\, ds\\
    &=\int_{s_0-\delta}^{s_0+\delta}
     2\rho_\eps^\prime\bigl(
     \Norm{u_s-u_{s_0}}_2^2\bigr)
     \langle u_s-u_{s_0},
     \eta_s\rangle
     \langle\xi_s,
     \eta_{s_0}\rangle\, ds\\
    &\quad
     +\int_{s_0-\delta}^{s_0+\delta}
     \rho_\eps\bigl(
     \Norm{u_s-u_{s_0}}_2^2\bigr)
     \langle
     E_2(u_{s_0},\xi_s)^{-1}
     \eta_s,\eta_{s_0}
     \rangle\, ds.
\end{split}
\end{equation}
We shall estimate
the two terms in the
sum separately.
Let $s_2>s_0$ be such that
$\norm{u_{s_2}-u_{s_0}}_2=\eps$
and
$\norm{u_s-u_{s_0}}_2<\eps$
whenever $s\in(s_0,s_2)$.
This means that
$s_2$ is the \emph{forward}
exit time of $u_s$
with respect to the
$L^2$ ball of radius $\eps$
about $u_{s_0}$.
Let $s_1<s_0$ be the corresponding
\emph{backward} exit time;
see figure~\ref{fig:fig-exit}.
Then, by~{\rm ii)}
and $\rho_\eps\ge0$, it holds that
\begin{equation*}
\begin{split}
    &\int_{s_0-\delta}^{s_0+\delta}
     \rho_\eps\bigl(
     \Norm{u_s-u_{s_0}}_2^2\bigr)
     \langle
     E_2(u_{s_0},\xi_s)^{-1}
     \eta_s,\eta_{s_0}
     \rangle\, ds\\
    &\ge\int_{s_1}^{s_2}
     1\cdot\frac{\mu_0}{2}\, ds
     =\frac{\mu_0}{2}
     \left( s_2-s_0+s_0-s_1\right)\\
    &\ge\frac{\mu_0}{3\mu_1}\left(
     \Norm{u_{s_2}-u_{s_0}}_2
     +\Norm{u_{s_0}-u_{s_1}}_2\right)
     =\frac{2\mu_0}{3\mu_1}\, \eps.
\end{split}
\end{equation*}
Here the second 
inequality uses~{\rm iii)}.
To estimate the
other term in~(\ref{eq:grad-V-est})
let $\sigma_1$
be the time of first entry
into the $L^2$ ball of radius
$2\eps$ starting from $s_0-\delta$
and let $\sigma_2$ be
the corresponding time
when time runs backwards
and we start from $s_0+\delta$;
see figure~\ref{fig:fig-exit}.
\begin{figure}
  \centering
  \epsfig{figure=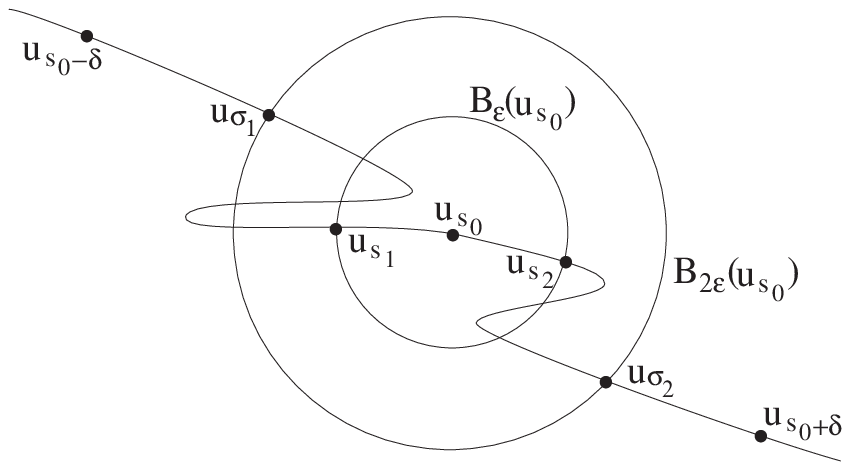}
  \caption{Exit times $s_1,s_2$ and 
  entry times $\sigma_1,\sigma_2$} 
  \label{fig:fig-exit}
\end{figure}
Then it follows that
\begin{equation*}
\begin{split}
    &\int_{s_0-\delta}^{s_0+\delta}
     2\rho_\eps^\prime\bigl(
     \Norm{u_s-u_{s_0}}_2^2\bigr)
     \langle u_s-u_{s_0},
     \eta_s\rangle
     \langle\xi_s,
     \eta_{s_0}\rangle\, ds\\
    &\ge-2\int_{\sigma_1}^{\sigma_2}
     \Norm{\rho_\eps^\prime}_\infty
     \Abs{\langle u_s-u_{s_0},
     \eta_s\rangle}\cdot
     \abs{\langle\xi_s,
     \eta_{s_0}\rangle}\, ds\\
    &\ge-2c_1c_2\eps^{-2}\int_{\sigma_1}^{\sigma_2}
     (s-s_0)^4\, ds\\
    &=-\frac{2c_1c_2}{5\eps^2}\left(
     \sigma_2-s_0+s_0-\sigma_1\right)^5
     \ge-\frac{2c_1c_28^5}{5\mu_1^5}\eps^3.
\end{split}
\end{equation*}
It remains to explain the
second and the final inequality.
In the final one we use that
by~{\rm iii)} there is the
estimate
$\sigma_2-s_0\le
2\norm{u_{\sigma_2}-u_{s_0}}_2/\mu_1
=4\eps/\mu_1$
and similarly
for $s_0-\sigma_1$.
The second inequality
is based on the geometric fact
that $\p_su$ and $\eta$
are slicewise orthogonal
by step~2. Namely,
let $f(s)=\langle u_s-u_{s_0},
\eta_s\rangle$ and
$h(s)=\langle\xi_s,
\eta_{s_0}\rangle$, then
$f(s_0)=h(s_0)=0$ and
\begin{equation*}
\begin{split}
     f^\prime(s)
    &=\langle\p_su_s,\eta_s\rangle
     +\langle u_s-u_{s_0},
     \Nabla{s}\eta_s\rangle
     =\langle u_s-u_{s_0},
     \Nabla{s}\eta_s\rangle\\
     h^\prime(s)
    &=\langle E_2(u_{s_0},\xi_s)^{-1}\p_su_s,
     \eta_{s_0}\rangle.
\end{split}
\end{equation*}
Hence $f^\prime(s_0)=h^\prime(s_0)=0$
and so there exist constants
$c_1=c_1(f)>0$ and $c_2=c_2(h)>0$
depending continuously
on $\delta$
such that for every 
$s\in(s_0-\delta,s_0+\delta)$
$$
     \Abs{f(s)}
     \le c_1(s-s_0)^2,\qquad
     \Abs{h(s)}
     \le c_2(s-s_0)^2.
$$
This proves the second inequality.
Now choose $\eps>0$
sufficiently small
such that
$\eps^2<\mu_0\mu_1^4/c_1c_2$.
This implies that
$\langle\grad\,\Vv_0 (u),
\eta\rangle>0$ and proves step~4.

Recall that $u_{s_0}\notin U$.
Now we choose $\eps>0$
again smaller such that
the $L^2$ ball of radius $3\eps$ 
about $u_{s_0}$ is disjoint from 
the $L^2$ closed set $U$,
that $3\eps$ is smaller
than the injectivity radius
$\iota$ of $M$,
and that $\eps=1/k$ for
some integer $k$.

\vspace{.1cm}
\noindent
{\bf Step 5.}
{\it Given $k=1/\eps$ as in the paragraph above,
there exist integers $i,j>0$ such that
the function $\hat\Vv:=\Vv_{ijk}$
given by~(\ref{eq:perturbation})
lies in $Y(\Vv,U)$ and satisfies
$$
     \langle\grad\,\Vv_{ijk}(u),
     \eta\rangle>0.
$$
This contradicts~(\ref{eq:dense-2})
and thereby proves
proposition~\ref{prop:onto-universal}.
}

\vspace{.1cm}
\noindent
Let $u_{s_0}$ be as in step~4.
In section~\ref{subsec:perturbations-banach}
we fixed a dense
sequence $(x_i)$ in $C^\infty(S^1,M)$
and for each $i$
a dense sequence $(\eta^{ij})$
in $C^\infty(S^1,x_i^*TM)$.
Choose a subsequence, still
denoted by $(x_i)$,
such that
\begin{equation*}
     x_i
     \to u_{s_0}\qquad
     \text{as $i\to\infty$}.
\end{equation*}
Now we may assume
without loss of generality
that every $x_i$
lies in $B_\eps(u_{s_0})$
the $L^2$ ball of radius
$\eps$ about $u_{s_0}$.
Hence
$B_{2\eps}(x_i)\subset
B_{3\eps}(u_{s_0})$.
Let $\xi_{s_0}^i$ be defined
by the identity
$u_{s_0}=\exp_{x_i} \xi_{s_0}^i$
pointwise for every $t\in S^1$.
Choose a diagonal
subsequence, denoted for simplicity
by $(\eta^{ii})$, such that
$$
     \Phi_{x_i}(\xi_{s_0}^i)\eta^{ii}
     \to \eta_{s_0}\qquad
     \text{as $i\to\infty$}.
$$
Here $\Phi_x(\xi)$ is
parallel transport from $x$
to $\exp_x \xi$ along
$\tau\mapsto\exp_x\tau \xi$
pointwise for every $t\in S^1$.
Let $(\Vv_{iik})_{i\in\N}$
be the corresponding sequence of
functions where each $\Vv_{iik}$
is given by~(\ref{eq:perturbation}).
Now observe that
$$
     \supp\Vv_{iik}
     \subset B_{2/k}(x_i)
     =B_{2\eps}(x_i)
     \subset B_{3\eps}(u_{s_0}).
$$
But $B_{3\eps}(u_{s_0})\cap U=\emptyset$
by the choice of $\eps$
in the paragraph prior to step~4,
and therefore $\Vv_{iik}\in Y(\Vv,U)$.
Next recall that the constant $\delta>0$
has been chosen in the proof of step~4
in order to exclude any
return of the trajectory
$s\mapsto u_s$ to
the ball $B_{3\eps}(u_{s_0})$
once $s$ has left the
interval $(s_0-\delta,s_0+\delta)$.
Since $\supp\Vv_{iik}\subset B_{3\eps}(u_{s_0})$,
this shows that $\Vv_{iik}(u_s)=0$
whenever $s\notin(s_0-\delta,s_0+\delta)$.
Hence
\begin{equation*}
\begin{split}
     \langle\grad\,\Vv_{iik}(u),\eta\rangle
    &=\int_{s_0-\delta}^{s_0+\delta}
     2\rho_{1/k}^\prime\bigl(
     \Norm{u_s-x_i}_2^2\bigr)
     \langle u_s-x_i,
     \eta_s\rangle
     \langle\xi_s^i,
     \eta^{ii}\rangle\, ds\\
    &\quad
     +\int_{s_0-\delta}^{s_0+\delta}
     \rho_{1/k}\bigl(
     \Norm{u_s-x_i}_2^2\bigr)
     \langle
     E_2(x_i,\xi_s^i)^{-1}
     \eta_s,\eta^{ii}
     \rangle\, ds
\end{split}
\end{equation*}
where $\xi_s^i$ is
determined
by $u_s=\exp_{x_i}\xi_s^i$.
Now the right hand
side converges as $i\to\infty$
to the right hand side
of~(\ref{eq:grad-V-est}),
which equals
$\langle\grad\,\Vv_0 (u),
\eta\rangle>0$.
This proves step~5 and 
proposition~\ref{prop:onto-universal}.
\end{proof}

\section{Heat flow homology}
\label{sec:Morse-homology}

In section~\ref{subsec:unstable}
we define the unstable manifold
of a critical point $x$
of the action functional
$\Ss_\Vv:\Ll M\to\R$
as the set of endpoints at time zero
of all backward halfcylinders
solving the heat equation~(\ref{eq:heat})
and emanating from $x$ at $-\infty$.
The main result is 
theorem~\ref{thm:unstable-mf} saying that
if $x$ is nondegenerate,
then this is a submanifold
of the loop space and its dimension
equals the Morse index of $x$.

Section~\ref{subsec:Morse-complex}
puts together the results proved so far
to construct the Morse complex
for the negative $L^2$ gradient
of the action functional on the loop space.

\subsection{The unstable manifold theorem}
\label{subsec:unstable}

Fix a perturbation
$\Vv:\Ll M\to\R$
that satisfies~{\rm (V0)--(V3)}
and let $Z^-$ be
the {\bf backward halfcylinder}
$(-\infty,0]\times S^1$.
Given a critical point $x$
of the action
functional $\Ss_\Vv$
the moduli space
\begin{equation}\label{eq:M-}
     \Mm^-(x;\Vv)
\end{equation}
is, by definition,
the set of all smooth solutions
$u^-:Z^-\to M$
of the heat 
equation~(\ref{eq:heat})
such that
$u^-(s,t)\to x(t)$
as $s\to-\infty$,
uniformly in $t\in S^1$.
Note that the moduli space
is not empty, since
it contains
the stationary solution
$u^-(s,t)=x(t)$.
The {\bf unstable
manifold of \boldmath$x$}
is defined by
$$
     W^u(x;\Vv)
     =\{{u^-}(0,\cdot)\mid
     {u^-}\in\Mm^-(x;\Vv)\}.
$$
\begin{theorem}
\label{thm:unstable-mf}
Let $\Vv:\Ll M\to\R$
be a perturbation
that satisfies~{\rm (V0)--(V3)}.
If $x$ is a nondegenerate critical point
of the action functional $\Ss_\Vv$,
then the unstable manifold $W^u(x;\Vv)$ 
is a smooth contractible
embedded submanifold
of the loop space and its dimension
is equal to the Morse index of $x$.
\end{theorem}

The idea to prove theorem~\ref{thm:unstable-mf} 
is to first show
in proposition~\ref{prop:unstable-modspace}
that nondegeneracy of $x$
implies that
the moduli space $\Mm^-(x;\Vv)$
is a smooth manifold
of the desired dimension.
A crucial ingredient
is proposition~\ref{prop:onto}
on surjectivity of the operator
$\Dd_{u^-}:\Ww^{1,p}\to\Ll^p$
whenever $u^-\in\Mm^-(x;\Vv)$
and $p\ge2$. Here the operator
$\Dd_{u^-}$ given by~(\ref{eq:lin-op})
arises by linearizing the
heat equation at the backward
trajectory $u^-$. 
A further key result
to prove theorem~\ref{thm:unstable-mf} 
is unique continuation
for the linear and
the nonlinear heat equation,
proposition~\ref{prop:UC-L}
and theorem~\ref{thm:UC-NL-noncompact}.
Namely, unique continuation
implies that the evaluation map
$$
     ev_0:\Mm^-(x;\Vv)\to\Ll M,\qquad
     {u^-}\mapsto {u^-}(0,\cdot)
$$
is an injective immersion.
It is even an embedding
by the gradient flow property.

\begin{proposition}[Moduli space]
\label{prop:unstable-modspace}
Let $\Vv:\Ll M\to\R$
be a perturbation
satisfying~{\rm (V0)--(V3)}
and suppose that
$x$ is a nondegenerate
critical point of $\Ss_\Vv$.
Then the moduli space
$\Mm^-(x;\Vv)$ is a smooth contractible
manifold of dimension
$\IND_\Vv(x)$.
Its tangent space
at ${u^-}$ is equal to
the vector space $X^-$
given by~(\ref{eq:kernel}).
\end{proposition}

\begin{proposition}[Surjectivity]
\label{prop:onto}
Fix a constant $p>2$,
a perturbation $\Vv$
that satisfies~{\rm (V0)--(V3)},
and a nondegenerate
critical point $x$ of $\Ss_\Vv$.
If ${u^-}\in\Mm^-(x;\Vv)$,
then the operator
$
     \Dd_{u^-}:\Ww^{1,p}\to\Ll^p
$
is onto and its kernel is given by
\begin{equation}\label{eq:kernel}
\begin{split}
     X^-
    &:=\Bigl\{
     \xi\in C^\infty(Z^-,{u^-}^*TM)
     \mid\text{
     $\Dd_{u^-}\xi=0$,
     $\exists c,\delta>0\;
     \forall s\le 0:$}
     \\
    &\qquad
     \text{
     $
       \Norm{\xi_s}_\infty
       +\Norm{\Nabla{t}\xi_s}_\infty
       +\Norm{\Nabla{t}\Nabla{t}\xi_s}_\infty
       +\Norm{\Nabla{s}\xi_s}_\infty
       \le ce^{\delta s}
     $}
     \Bigr\}.
\end{split}
\end{equation}
Moreover, the dimension of
$X^-$ is equal to
the Morse index of $x$.
\end{proposition}

Proposition~\ref{prop:onto}
is in fact a corollary
of theorem~\ref{thm:onto}
below which asserts
surjectivity
in the special case
of a stationary solution
${u^-}(s,t)=x(t)$,
where $x$ is a nondegenerate
critical point of $\Ss_\Vv$.
The idea is that
if a solution
${u^-}$ is nearby the
stationary solution $x$
in the $\Ww^{1,p}$ topology,
then the corresponding 
linearizations $\Dd_{u^-}$ and $\Dd_x$
are close in the operator norm topology.
But surjectivity is an open condition
with respect to the norm topology.
The case of a general solution 
reduces to the nearby case
by shifting
the $s$-variable.

\begin{remark}\label{rmk:A}
Abbreviate 
$H=L^2(S^1,\R^n)$
and $W=W^{2,2}(S^1,\R^n)$ and
consider the operator
$$
     A_S=-\frac{d^2}{dt^2}-S:H\to H
$$
with dense domain $W$.
Here we assume that $S:W\to H$
is a symmetric and compact
linear operator.
Under these assumptions
it is well known (see~(ii) in 
section~\ref{sec:linearized})
that $A_S$ is self-adjoint
and that its
Morse index $\IND(A_S)$,
that is the dimension
of the negative eigenspace 
$E^-$ of $A_S$,
is finite.
\end{remark}

\begin{theorem}
\label{thm:onto}
Let $S$ and $A_S$ be as
in remark~\ref{rmk:A}.
Fix $p\ge 2$ and assume
that the
linear operator
$S:W^{1,p}(S^1,\R^n)\to L^p(S^1,\R^n)$
is bounded with bound $c_S$.
Then the following is true.
If $A_S$ is injective, then
the operator
$$
     D=\p_s-\p_t\p_t-S:
     \Ww^{1,p}(Z^-,\R^n)\to 
     L^p(Z^-,\R^n)
$$
is onto.
In the case $p=2$ the map
$E^-\to\ker D$, $v\mapsto e^{-sA_S}v$
is an isomorpism.
\end{theorem}

\begin{proof}[Proof of theorem~\ref{thm:onto}]
The proof takes four steps.
Step~1 proves the theorem for $p=2$.
The proof by Salamon~\cite[lemma~2.4 step~1]{Sa-FLOER}
of the corresponding
result in Floer theory carries
over with minor but important
modifications. These are
due to the fact that our domain $Z^-$
does have a boundary.
The proof uses the theory of
semigroups. We recall the details
for convenience of the reader.
The generalization of surjectivity
in step~4 to $p>2$
follows an argument
due to Donaldson~\cite{Don-FLOER}.
It uses the case $p=2$ and the
estimates provided by step~2
and step~3.
Again we follow the presentation
in~\cite[lemma~2.4 steps~2--4]{Sa-FLOER}
up to minor but subtle modifications.
One subtlety is related to the parabolic
estimate of step~2. Here 
in contrast to the elliptic case 
the domain
needs to be increased only
towards the \emph{past}.
Hence the estimates
of step~3 work precisely
for the
\emph{backward} halfcylinder.
Throughout the proof,
unless indicated differently,
the domain of
all spaces 
is the backward halfcylinder
$Z^-$ and the target is $\R^n$.

\vspace{.1cm}
\noindent
{\bf Step~1.}
{\it The theorem is true for $p=2$.
}

\vspace{.1cm}
\noindent
The operator $A_S$
is unbounded and self-adjoint
on the Hilbert space
$H$ with dense domain $W$.
Denote the negative and positive
eigenspaces of $A_S$
by $E^-$ and $E^+$, respectively.
Note that $\dim E^-<\infty$ 
by remark~\ref{rmk:A}.
By assumption
$A_S$ is injective,
hence zero is not an eigenvalue
and there is a splitting
$H=E^-\oplus E^+$.
Denote by
$P^\pm:H\to E^\pm$
the orthogonal projections
and set $A^\pm=A_S|_{E^\pm}$.
The self-adjoint 
negative semidefinite
operators $A^-$ and $-A^+$ 
generate contraction
semigroups on $E^-$ and $E^+$,
respectively, by the
Hille-Yosida theorem; see
e.g.~\cite[sec.~X.8 ex.~1]{ReSi-II}.
We denote them by
$s\mapsto e^{sA^-}$ and
$s\mapsto e^{-sA^+}$, respectively.
Both are defined for $s\ge0$.
Define the map $K:\R\to\Ll(H)$ by
$$
     K(s)
     =
     \begin{cases}
       -e^{-sA^-}P^-,
       &\text{for $s\le0$,}
       \\
       e^{-sA^+}P^+,
       &\text{for $s>0$.}
     \end{cases}
$$
This function is strongly
continuous for $s\not=0$
and satisfies
\begin{equation}\label{eq:K}
     \Norm{K(s)}_{\Ll(H)}
     \le e^{-\delta\abs{s}}
\end{equation}
where
$\delta=\min\{-\lambda^-,\lambda^+\}>0$.
Here $\lambda^-$ denotes
the largest eigenvalue of $A^-$
and $\lambda^+$
the smallest eigenvalue of $A^+$.
Abbreviate $\R^-=(-\infty,0]$.
For $\eta\in L^2(\R^-,H)$
consider the operator 
$$
     \left(Q\eta\right)(s)
     :=\int_{-\infty}^0
     K(s-\sigma)\eta(\sigma)\; d\sigma.
$$
Now the operator
$Q$ maps $L^2(\R^-,H)$ to
the intersection of Banach spaces
$W^{1,2}(\R^-,H)\cap L^2(\R^-,W)$
and it is a right inverse of $D$.
To prove the latter set $\xi:=Q\eta$.
Then $\xi=\xi^-+\xi^+$,
where
$$
     \xi^+(s)
     =\int_{-\infty}^s
     e^{-(s-\sigma)A^+}
     P^+\eta(\sigma)\; 
     d\sigma,
     \qquad
     \xi^-(s)
     =-\int_s^0
     e^{-(s-\sigma)A^-}
     P^-\eta(\sigma)\; 
     d\sigma.
$$
Calculation shows that
$D\xi^\pm=P^\pm\eta$
pointwise for every $s\in\R^-$.
It follows that
$$
     DQ\eta=D\xi=D\xi^-+D\xi^+
     =P^-\eta+P^+\eta=\eta.
$$
Since the space
$W^{1,2}(\R^-,H)\cap L^2(\R^-,W)$
agrees with $\Ww^{1,2}$,
this proves that $Q$ is a right inverse
of $D$. Hence $Q$ is injective
and $D$ is onto.
To calculate the
kernel of $D$
fix $\xi\in\Ww^{1,2}$
and set $\eta:=D\xi$.
Then by straightforward calculation
\begin{equation*}
\begin{split}
     \left(QD\xi\right)(s)
    &=\left(Q\eta\right)(s)
     =\xi^+(s) +\xi^-(s)\\
    &=\int_{-\infty}^s\frac{d}{d\sigma}\left(
     e^{-(s-\sigma)A^+}
     P^+\xi(\sigma)\right) d\sigma
     -\int_s^0\frac{d}{d\sigma}\left(
     e^{-(s-\sigma)A^-}
     P^-\xi(\sigma)\right) d\sigma\\
    &=P^+\xi(s)
     -e^{-sA^-}P^-\xi(0)+P^-\xi(s)\\
    &=\xi(s)-e^{-sA^-}P^-\xi(0).
\end{split}
\end{equation*}
To obtain the third identity
replace
$\eta(\sigma)$ in 
$\xi^\pm(s)$ by
$\xi^\prime(\sigma)+A_S\xi(\sigma)$
and use the fact that 
$A^\pm P^\pm=P^\pm A_S$.
Now observe that
$\xi\in\ker D$ is equivalent to
$D\xi\in\ker Q$,
because $Q$ is injective.
But $QD\xi=0$ means that
$
     \xi(s)=e^{-sA^-}P^-\xi(0)
$
for every $s\in\R^-$.
This shows that the map
\begin{equation}\label{eq:iso-kernel}
     E^-\to\ker \big[D:\Ww^{1,2}\to L^2\big]:
     v_k\mapsto e^{-s\lambda_k}v_k
\end{equation}
induces an isomorphism.
Here $v_1,\ldots,v_N$
is an orthonormal
basis of $E^-$
consisting of eigenvectors
of $A_S$ 
with eigenvalues 
$\lambda_1,\dots,\lambda_N$
and where $N=\IND(A_S)$.

\vspace{.1cm}
\noindent
{\bf Step~2.}
{\it Fix a constant $p\ge 2$.
Then there is a constant
$c_1=c_1(p,c_S)$ such that
$$
     \Norm{\xi}_{\Ww^{1,p}([-1,0]\times S^1)}
     \le c_1 \left(
     \Norm{D\xi}_{L^p([-3,0]\times S^1)}
     +\Norm{\xi}_{L^2([-3,0]\times S^1)}
     \right)
$$
for
$\xi\in C^\infty([-3,0]\times S^1)$.
Moreover, if 
$\xi\in \Ww^{1,2}$ and
$D\xi\in L^p_{loc}$,
then $\xi\in \Ww^{1,p}_{loc}$.
}

\vspace{.1cm}
\noindent
Choose a smooth compactly supported
cutoff function $\rho:(-2,0]\to [0,1]$
such that $\rho=1$ on $[-1,0]$
and $\norm{\p_s\rho}_\infty\le 2$.
Now apply proposition~\ref{prop:par-linear}
for the backward halfcylinder $Z^-$,
Euclidean space $\R^n$,
covariant derivatives
replaced by partial derivatives,
and with constant $c$
to the function $\rho\xi$
to obtain that
\begin{equation*}
     \Norm{\xi}_{\Ww^{1,p}([-1,0]\times S^1)}
     \le c \left(
     2\Norm{(\p_s-\p_t\p_t)\xi}
     _{L^p([-2,0]\times S^1)}
     +\Norm{\xi}_{L^p([-2,0]\times S^1)}
     \right)
\end{equation*}
for every
$\xi\in C^\infty([-2,0]\times S^1)$.
To obtain the first estimate
in step~3 for the backward 
\emph{half}cylinder
it will be crucial that
the domain on the right hand
side does not extend to the future.
Now write $\p_s-\p_t\p_t=D+S$
and use that the operator
$S:W^{1,p}(S^1)\to L^p(S^1)$
is bounded to obtain that
\begin{equation*}
\begin{split}
     \Norm{\xi}_{\Ww^{1,p}([-1,0]\times S^1)}
    &\le c \Bigl(
     \Norm{D\xi}_{L^p([-2,0]\times S^1)}
     +(1+c_S)\Norm{\xi}_{L^p([-2,0]\times S^1)}\\
    &\quad
     +c_S\Norm{\p_t\xi}_{L^p([-2,0]\times S^1)}
     \Bigr)\\
\end{split}
\end{equation*}
for every
$\xi\in C^\infty([-2,0]\times S^1)$
and some constant $\tilde{c}=\tilde{c}(p,c,c_S)$.
Now integrate the estimate in
lemma~\ref{le:plus-minus}
over $s\in[-2,0]$
and chose $\delta>0$
sufficiently small
in order to throw
the arising term $\p_t\p_t\xi$ 
to the left hand side.
It follows that
\begin{equation}\label{eq:parabolic-basic-2}
     \Norm{\xi}_{\Ww^{1,p}([-1,0]\times S^1)}
     \le \tilde{c}\Bigl(
     \Norm{D\xi}_{L^p([-2,0]\times S^1)}
     +\Norm{\xi}_{L^p([-2,0]\times S^1)}\Bigr)
\end{equation}
for every
$\xi\in C^\infty([-2,0]\times S^1)$
and some constant $\tilde{c}=\tilde{c}(p,c,c_S)$.
It remains to replace
the $L^p$ norm of $\xi$
by the $L^2$ norm.
Since $p\ge2$, there
is the Sobolev inequality
$
     \norm{\xi}_{L^p}\le 
     c_p\norm{\xi}_{W^{1,2}}
$
for $\xi\in W^{1,2}$; see
e.g.~\cite[theorem~8.5~(ii)]{LIEB-LOSS}
for the domain $\R^2$.
The first step is
to replace the last term
in~(\ref{eq:parabolic-basic-2})
according to the Sobolev
inequality.
Then use~(\ref{eq:parabolic-basic-2})
with $p=2$ and on \emph{increased} domains
to complete the proof
of the estimate in step~2
(use H\"older's inequality
to estimate the $L^2$ norm 
of $D\xi$ by the $L^p$ norm).

To conclude
the proof of step~2
assume $\xi\in\Ww^{1,2}$,
then of course $\xi\in L^2$
and $D\xi\in L^2$.
If in addition
$D\xi$ is locally $L^p$ integrable,
then the estimate of step~2
which we just proved shows
that $\xi\in\Ww^{1,p}_{loc}$.

\vspace{.1cm}
\noindent
{\bf Step~3.}
{\it Fix a constant $p\ge2$
and consider the norm
$$
     \Norm{\xi}_{2;p}
     =\left(\int_{-\infty}^0
     \Norm{\xi(s,\cdot)}^p_{L^2(S^1)} ds
     \right)^{1/p}.
$$
Then there exist constants
$c_2$ and $c_3$ both depending
on $p$ and $c_S$ such that
the following is true. If
$\xi\in\Ww^{1,2}$ and $D\xi\in L^p$,
then $\xi\in\Ww^{1,p}$ and
$$
     \Norm{\xi}_{\Ww^{1,p}}
     \le c_2 \left(
     \Norm{D\xi}_{L^p}+\Norm{\xi}_{2;p}
     \right)
     ,\qquad
     \Norm{QD\xi}_{2;p}
     \le c_3
     \Norm{D\xi}_{L^p}.
$$
}

\vspace{.1cm}
\noindent
Fix $\xi\in\Ww^{1,2}$
such that $D\xi\in L^p$.
Then $\xi\in\Ww^{1,p}_{loc}$
by step~2.
Moreover,
the estimate of step~2
implies that
$$
     \Norm{\xi}_{\Ww^{1,p}([k,k+1]\times S^1)}^p
     \le 3^{p/2-1}2^p{c_1}^p
     \int_{k-2}^{k+1}\left(
     \Norm{D\xi}_{L^p(S^1)}^p
     +\Norm{\xi}_{L^2(S^1)}^p\right) ds
$$
for every integer $k<0$;
see~\cite[lemma~2.4 step~3]{Sa-FLOER}
for details.
Now take the
sum over all such $k$
to obtain the 
first estimate of step~3.

Next observe that
$\eta:=D\xi$
lies in $L^2(\R^-,H)$
and in $L^p(\R^-,H)$.
Here $H=L^2(S^1)$ and
we used that
by H\"older's inequality
\begin{equation}\label{eq:2-p-est}
     \Norm{\cdot}_{L^2(S^1)}
     \le
     \Norm{\cdot}_{L^p(S^1)}.
\end{equation}
Since $\eta$ is in the domain
$L^2(\R^-,H)$
of the operator $Q$
from step~1,
we obtain
$$
     QD\xi=Q\eta=K*\eta.
$$
Now Young's inequality
applies to $K*\eta$,
because
$\eta\in L^p(\R^-,H)$.
Hence
\begin{equation}\label{eq:y_2}
     \Norm{K*\eta}_{2;p}
     \le\Norm{K}_{L^1(\R^-,\Ll(H))}
     \Norm{\eta}_{L^p(\R^-,H)}
     \le C
     \Norm{D\xi}_{L^p}
\end{equation}
where $C$ depends
on the constant
$\delta$ in
estimate~(\ref{eq:K})
for the norm of $K$;
see~\cite{Sa-FLOER}.
The last step
uses~(\ref{eq:2-p-est}) again.
This proves the second
estimate of step~3.

It remains to prove 
that $\xi\in\Ww^{1,p}$.
The two estimates
of step~3
imply that
$$
     \Norm{\xi}_{\Ww^{1,p}}
     \le c_2 \left((1+c_3)
     \Norm{D\xi}_{L^p}
     +\Norm{\xi-QD\xi}_{2;p}
     \right).
$$
To see that
the right hand side
is finite recall
that $D\xi\in L^p$
by assumption
and $\xi-QD\xi$
lies in the kernel of
$D:\Ww^{1,2}\to L^2$
by (the proof of) step~1.
Moreover,
by~(\ref{eq:iso-kernel})
every element
of this kernel is
a finite sum of functions
of the form 
$\xi_k=e^{-s\lambda_k}v_k$
and  $\norm{\xi_k}_{2;p}<\infty$
by calculation.

\vspace{.1cm}
\noindent
{\bf Step~4.}
{\it The theorem is true for $p>2$.
}

\vspace{.1cm}
\noindent
Fix $p>2$ and set
$X^-:=\ker[D:\Ww^{1,2}\to L^2]$.
Then the linear operator
$$
     \pi:
     \Ww^{1,p}\to
     \left(X^-,\norm{\cdot}_{2;p}\right),\qquad
     \xi\mapsto\xi-QD\xi,
$$
is well defined, bounded and of
finite rank, hence compact.
To prove this
observe that $\pi$ is
well defined on the
dense subset $C_0^\infty(Z^-)$ 
of $\Ww^{1,p}$. Since
$C_0^\infty(Z^-)$ is
also dense in $\Ww^{1,2}$,
step~1 shows that
$\xi-QD\xi\in X^-$.
To see that $\pi$
is bounded on $C_0^\infty(Z^-)$
let $\xi\in C_0^\infty(Z^-)$.
Then
$$
     \Norm{\pi\xi}_{2;p}
     =\Norm{\xi-QD\xi}_{2;p}
     \le\Norm{\xi}_p
     +c_3\Norm{D\xi}_p
     \le(1+c_3c_4)\Norm{\xi}_{\Ww^{1,p}}
$$
by definition of $\pi$,
the triangle inequality,
the estimate~(\ref{eq:2-p-est}),
and the second estimate of
step~3.
The last inequality
follows from the estimate
$$
     \norm{D\xi}_{L^p}
     \le \norm{\p_s\xi}_{L^p}
     +\norm{\p_t\p_t\xi}_{L^p}
     +\norm{S\xi}_{L^p}
     \le c_4\Norm{\xi}_{\Ww^{1,p}}
$$
with suitable constant $c_4=c_4(p,c_S)$.
Here we used that
$\norm{S}_p\le c_S(\norm{\xi}_p+\norm{\p_t\xi}_p)$
by boundedness of $S$.
Now being bounded on a dense
subset the operator $\pi$ extends
to a bounded linear operator
on $\Ww^{1,p}$.
The rank of
$\pi$ is finite,
because the dimension
of its target $X^-$
is equal to the Morse index
of $A_S$ by step~1.

To prove that
$D:\Ww^{1,p}\to L^p$ is onto
we show first that the range is closed
and then that it is dense.
By the two estimates
of step~3 we have that
$$
     \Norm{\xi}_{\Ww^{1,p}}
     \le c_2 \left((1+c_3)
     \Norm{D\xi}_{L^p}+\Norm{\pi\xi}_{2;p}
     \right)   
$$
for every $\xi\in C_0^\infty$,
hence for every 
$\xi\in\Ww^{1,p}$
by density.
Since $\pi$ is compact,
the range of $D$ is closed
by the abstract closed range lemma.
To prove density of the range
fix
$
     \eta\in
     L^p\cap L^2
$
and note that the subset
$L^p\cap L^2$ is dense in $L^p$,
because it contains
the dense subset
$C_0^\infty$ of $L^p$.
Now by surjectivity
of $D$ in the case $p=2$ (step~1)
and since $\eta\in L^2$,
there exists an element
$\xi\in\Ww^{1,2}$
such that $D\xi=\eta$.
But then
$\xi\in\Ww^{1,p}$
by step~3,
because $D\xi=\eta\in L^p$
by the choice of $\eta$.
Hence
$\eta$ is in the range of
$D:\Ww^{1,p}\to L^p$.
This proves
theorem~\ref{thm:onto}.
\end{proof}

\begin{proof}[Proof of proposition~\ref{prop:onto}]
The arguments in the proof of 
proposition~\ref{prop:kernel-smooth}
show that the kernel of
$\Dd_{u^-}:\Ww^{1,p}\to\Ll^p$ 
is equal to $X^-$
and $X^-$ does not depend on $p$.
On the other hand, for $p=2$ 
the dimension of the kernel
is equal to the Morse index of $x$
by theorem~\ref{thm:onto}.
Surjectivity of $\Dd_{u^-}$
follows in three stages.

\vspace{.1cm}
\noindent
{\sc The stationary case.}
Consider the stationary
solution $u^-(s,t)=x(t)$,
then $\Dd_x$ is onto
by theorem~\ref{thm:onto}.
To see this represent
$\Dd_x$ with respect
to an orthonormal frame
along $x$; see
section~\ref{sec:linearized}.

\vspace{.1cm}
\noindent
{\sc The nearby case.}
Surjectivity
is preserved under small
perturbations with respect
to the operator norm.
Moreover, the operator
family $\Dd_{u^-}$ depends
continuously
on $u^-$ with respect to the
$\Ww^{1,p}$ topology
(here we use $p>2$).
Hence, if $u^-\in\Mm^-(x;\Vv)$
satisfies $u^-=\exp_x(\eta)$
and $\norm{\eta}_{\Ww^{1,p}}$
is sufficiently small,
it follows that
$\Dd_{u^-}$ is onto.

\vspace{.1cm}
\noindent
{\sc The general case.}
Given $u\in\Mm^-(x;\Vv)$
and $\sigma<0$,
consider the shifted solution
$u^\sigma(s,t):=u(s+\sigma,t)$.
Then
$\left(\Dd_{u}\xi\right)^\sigma
=\Dd_{u^\sigma}\xi^\sigma$
by shift invariance
of the linear heat equation.
This means that
surjectivity of $\Dd_{u}$
is equivalent
to surjectivity 
of $\Dd_{u^\sigma}$.
But the latter is true by the
nearby case above, because
$u^\sigma$ converges to
$x$ in the $\Ww^{1,p}$ topology
as $\sigma\to-\infty$.
To see this apply
theorem~\ref{thm:exp-decay}~(B)
on exponential decay to $u$
and note that 
$u^\sigma(0,t)=u(\sigma,t)$.
\end{proof}

\begin{proof}
[Proof of proposition~\ref{prop:unstable-modspace}]
The proof follows the same (standard)
pattern as the proof
of theorem~\ref{thm:IFT};
see also the introduction to
section~\ref{sec:IFT}.
The first key step is the definition
of a Banach manifold
$\Bb=\Bb^{1,p}_x$
of backward halfcylinders
emanating from $x$
such that $\Bb$
contains the moduli space
$\Mm^-(x;\Vv)$ whenever $p>2$.
The second key step
is to define
a smooth map
$\Ff_{u^-}$ between Banach spaces
as in~(\ref{eq:Ff_u}).
Its significance lies in the
fact that
its zeroes correspond precisely
to the elements of the moduli space
near $u^-$
and that $d\Ff_{u^-}(0)=\Dd_{u^-}$.
By proposition~\ref{prop:onto}
this operator
is surjective and the
dimension of its
kernel is equal to the 
Morse index of $x$.
Hence $\Mm^-(x;\Vv)$
is locally near $u^-$
modeled on $\ker\Dd_{u^-}$
by the implicit function
theorem for Banach spaces.
To see that the moduli space
is a contractible manifold
observe that backward time shift
provides a contraction
\begin{equation*}
\begin{split}
     h:\Mm^-(x;\Vv)\times[0,1]
    &\to\Mm^-(x;\Vv)
     \\
     (u,r)
    &\mapsto 
     u(\cdot -\sqrt{r/(1-r)},\cdot)
\end{split}
\end{equation*}
onto the stationary solution $x$, that is
$h$ is continuous and
satisfies $h(u,0)=u$ and $h(u,1)=x$
for every $u\in\Mm^-(x;\Vv)$.
\end{proof}

\begin{proof}[Proof of
theorem~\ref{thm:unstable-mf}]
We abbreviate
$\Mm^-=\Mm^-(x;\Vv)$
and $W^u=W^u(x;\Vv)$.
Recall that the moduli space $\Mm^-$
is a smooth manifold of
dimension equal to
$\IND_\Vv(x)$ by
proposition~\ref{prop:unstable-modspace}
and, furthermore, by definition 
the unstable manifold $W^u$
is equal to the image of the
evaluation map $ev_0:\Mm^-\to\Ll M$.
We use the notation $ev_0(u)=:u_0$,
hence $u_0(t)=u(0,t)$.
It remains
to prove that $ev_0$
and its linearization
are injective and that $ev_0$
is a homeomorphism onto $W^u$.

To prove that
$ev_0$ is injective
let $u,v\in\Mm^-$
and assume that $ev_0(u)=ev_0(v)$,
that is $u_0=v_0$.
Hence $u=v$ by 
theorem~\ref{thm:UC-NL-noncompact}
on backward unique
continuation.

We prove that the linearization
$d(ev_0)_u$ of $ev_0$ at $u\in\Mm^-$
is injective.
Let $\xi,\eta\in T_u\Mm^-$.
Hence $\Dd_u\xi=0=\Dd_u\eta$ by
proposition~\ref{prop:unstable-modspace}.
Now assume that
$d(ev_0)_u\xi=d(ev_0)_u\eta$.
This means that $\xi_0=\eta_0$.
Therefore $\xi=\eta$ by application of
proposition~\ref{prop:UC-L}~{\rm (a)}
on linear unique continuation
to the vector field $\xi-\eta$.

We prove that $ev_0:\Mm^-\to\Ll M$ 
is a homeomorphism onto its image.
Fix $u\in\Mm^-$ and 
recall that every immersion is locally an
embedding.
Hence there is an open disk 
$D$ in $\Mm^-$ containing $u$
such that $ev_0|_D:D\to \Ll M$
is an embedding.
It remains to prove that there is
an open neighborhood $U$ of $u_0=ev_0(u)$
in $\Ll M$ such that
\begin{equation}\label{eq:UU}
     U\cap W^u=U\cap ev_0(D).
\end{equation}
Now there are two cases.
In case one $u$ is
constant in $s$ and therefore $u\equiv x$.
Here we exploit the (negative) gradient flow property
that the restricted function $\Ss_\Vv|_{W^u}$
takes on its maximum precisely at the critical point $x$.
Case two is the complementary case
in which $u$ depends on $s$.
Here we use a convergence argument
based on the compactness 
theorem~\ref{thm:compactness-gradbound}.

\vspace{.1cm}
\noindent
{\bf Case 1.} ($u\equiv x$)
Set $c=\Ss_\Vv(x)$, then
a set $U$ having the desired 
property~(\ref{eq:UU})
is given by
$$
     U:=\{c- \eps 
     < \Ss_\Vv < c+ \eps\},
$$
where
$$
     2\eps:= \min_{u\in\cl D\setminus D}
      \left(\Ss_\Vv(x)-\Ss_\Vv(u_0)\right).
$$
Here the compact set $\cl D\setminus D$ is the
topological boundary of the open disc $D$.
Note that the elements of
$W^u\setminus ev_0(D)$ 
have action at most $c-2\eps$.

\vspace{.1cm}
\noindent
{\bf Case 2.} ($u\not\equiv x$)
Assume by contradiction that
there is no $U$ which satisfies~(\ref{eq:UU}).
Then there is a sequence
$\gamma^\nu\in W^u\setminus ev_0(D)$
that converges to $u_0$ in $\Ll M$
as $\nu\to\infty$.
Note that $\gamma^\nu=ev_0(u^\nu)$
where $u^\nu\in\Mm^-\setminus D$.
In particular, each heat trajectory $u^\nu$
converges in backward time asymptotically to $x$.
Thus we obtain that
$$
     \sup_{s\in(-\infty,0]}\Ss_\Vv(u^\nu_s)
     \le \Ss_\Vv(x)=:c
$$
for every $\nu$.
Together with the energy identity
this implies that
\begin{equation*}
\begin{split}
     E(u^\nu)
    &=\Ss_\Vv(x)-\Ss_\Vv(u^\nu_0)\\
    &=c-\tfrac12 \Norm{\p_tu^\nu_0}_{L^2(S^1)}^2
     +\Vv(u^\nu_0)\\
    &\le c+C_0
\end{split}
\end{equation*}
where $C_0>1$ is the constant in axiom~(V0).
Adapting the proofs
of the apriori theorem~\ref{thm:apriori-t}
and the gradient bound theorem~\ref{thm:gradient}
to cover the case of 
backward \emph{half}cylinders
it follows that there is 
a constant $C=C(c,\Vv)>0$ such that
$$
     \Norm{\p_tu^\nu}_\infty
     \le C,
$$
and
$$
     \Norm{\p_su^\nu}_\infty
     \le C \sqrt{E(u^\nu)}
     \le C(c+C_0)
$$
for every $\nu$.
Here the norms are
taken on the domain $(-\infty,0]\times S^1$.
Adapting also the proof
of the compactness 
theorem~\ref{thm:compactness-gradbound}
we obtain -- in view of
the uniform apriori $L^\infty$ bounds 
for $\p_tu^\nu$ and $\p_su^\nu$ just derived --the existence of a smooth heat flow
solution $v:(-\infty,0]\times S^1\to M$
and a subsequence, still denoted by $u^\nu$,
such that $u^\nu$ converges to $v$
in $C^\infty_{loc}$.
In particular, this implies that
$u_0=v_0$ and that
$\p_tu^\nu_s$ converges to $\p_tv_s$, 
as $\nu\to\infty$,
uniformly with all derivatives on $S^1$
and for each $s$.
This and our earlier uniform action
bound for $u^\nu_s$ show that
$$
     \Ss_\Vv(v_s)
     =\lim_{\nu\to\infty}\Ss_\Vv(u^\nu_s)
     \le c
$$
for every $s$.
To summarize,we have two backward flow lines $u$ and $v$
defined on $(-\infty,0]\times S^1$
along which the action is bounded from above by $c$
and which coincide along the loop $u_0=v_0$.
Hence theorem~\ref{thm:UC-NL-noncompact}~(B)
on backward unique continuation
asserts that $u=v$.
Because $u^\nu$ converges to $v$
in $C^\infty_{loc}$,
this means that $u^\nu\in D$
whenever $\nu$ is sufficiently large.
For such $\nu$ we arrive at the
contradiction 
$\gamma^\nu=ev_0(u^\nu)\in ev_0(D)$
and this proves theorem~\ref{thm:unstable-mf}.
\end{proof}

\subsection{The Morse complex}
\label{subsec:Morse-complex}

Assume that the action
$\Ss_V$ is a Morse function on the loop space.
This is true for a generic
potential $V\in C^\infty(S^1\times M)$
by~\cite{Joa-INDEX}.
Fix a regular value $a$ of $\Ss_V$
and, furthermore, for each critical point
$x\in\Pp^a(V)$ fix an {\bf orientation}
$\langle x\rangle$
of the tangent space at $x$
to the (finite dimensional)
unstable manifold $W^u(x;V)$.
By $\nu=\nu(V,a)$ we
denote a {\bf choice of orientations}
for all $x\in\Pp^a(V)$.
The {\bf Morse chain groups}
are the $\Z$-modules
$$
     \CM_k^a=\CM^a_k(V,\nu)
     :=\bigoplus_{\scriptstyle x\in \Pp^a(V)\atop 
                    \scriptstyle \IND_V(x)=k} \Z\,\langle x\rangle,
     \qquad k\in \Z.
$$
These modules are finitely generated
and graded by the Morse index.
We set $C_k^a=\{0$\} whenever the direct sum is taken 
over the empty set. We define 
$$
     \CM^a_*
     :=\bigoplus_{k=0}^N \CM^a_k
$$
where $N$ is the largest Morse index
of an element of the finite set $\Pp^a(V)$.

Set $\Vv(x)=\int_0^1 V_t(x(t))\, dt$
and note that $\Vv$
satisfies~{\rm (V0)--(V3)}.
Now consider the associated
set of admissible perturbations 
$\Oo^a$ of $\Vv$
defined by~(\ref{eq:OOa})
and the dense subset $\Oo^a_{reg}$
of regular perturbations
provided by theorem~\ref{thm:transversality}.
(The ambient Banach space $Y$ given by~(\ref{eq:Y})
provides the metric on $\Oo^a$.)
Now for any $v\in\Oo^a_{reg}$
we have the following key facts:
The functionals $\Ss_V$ and 
$\Ss_{V+v}$ coincide near their critical points
and have the same sublevel set
with respect to $a$.
Moreover, the perturbed
functional $\Ss_{V+v}$
is Morse-Smale below level $a$.
Here and throughout
we sometimes denote $\Vv+v$ 
in abuse of notation by $V+v$
to emphasize that we are actually
perturbing a \emph{geometric} potential.

To define the Morse boundary operator 
$\p$ on $\CM^a_*$
it suffices to define
it on the set of generators $\Pp^a(V)$
and then extend linearly.
Fix a regular perturbation $v\in\Oo^a_{reg}$.
Note that each chosen orientation 
$\langle x\rangle$
orients the \emph{perturbed}
unstable manifold $W^u(x;V+v)$.
This is because the tangent spaces
at $x$ to $W^u(x;V)$ and $W^u(x;V+v)$
coincide ($v$ is not supported near $x$)
and unstable manifolds are finite dimensional and
contractible, hence orientable, 
by theorem~\ref{thm:unstable-mf}.
Now given two critical points $x^\pm$
of action less than $a$,
consider the heat moduli space $\Mm(x^-,x^+;V+v)$
of solutions $u$ of the heat equation~(\ref{eq:heat})
with $\Vv$ replaced by $\Vv+v$ and subject to the boundary
condition~(\ref{eq:limit}).
Jointly with D.~Salamon 
we proved in~\cite[ch.~11]{SaJoa-LOOP}
that a choice of orientations
for all unstable manifolds
determines a system of
{\bf coherent orientations}
on the heat moduli spaces
in the sense of 
Floer--Hofer~\cite{FLOER-HOFER-ORIENTATION}.

From now on we assume that $x^\pm$
are of \emph{Morse index difference one}.
In this case $\Mm(x^-,x^+;V+v)$
is a smooth 1-dimensional manifold
by theorem~\ref{thm:IFT}
and its quotient
$\Mm(x^-,x^+;V+v)/\R$ 
by the (free) time shift action
consists of finitely many points
by proposition~\ref{prop:finite-set}.
For $[u]\in\Mm(x^-,x^+;V+v)/\R$
time shift naturally induces an orientation
of the corresponding component of $\Mm(x^-,x^+;V+v)$;
compare~\cite{SaJoa-LOOP} and note 
that $\p_su\in\ker\Dd_u=\det(\Dd_u)$.
We set $n_u=+1$, if the time shift
orientation coincides with the
coherent orientation, 
and we set $n_u=-1$ otherwise. One calls 
$n_u$ the {\bf characteristic sign} of 
the heat trajectory $u$.
It depends on the orientations
$\langle x^-\rangle$ and $\langle x^+\rangle$.
Consider the (finite) sum
of characteristic signs
corresponding to all heat trajectories
from $x^-$ to $x^+$, namely
$$
     n(x^-,x^+)
     :=\sum_{[u]\in \Mm(x^-,x^+;V+v)/\R} n_u.
$$
If the sum runs over the
empty set, we set $n(x^-,x^+)=0$.
For $x\in\Pp^a(V)$
define the {\bf Morse boundary operator} 
{\boldmath $\p=\p(V,a,\nu,v)$}
by the (finite) sum
$$
     \p x
     :=\sum_{\scriptstyle y\in\Pp^a(V)
               \atop 
               \scriptstyle \IND_V(x)-\IND_V(y)=1}
     n(x,y)\, y.
$$
Set $\p x=0$, 
if the sum runs over the empty set.

\begin{proof}[Proof of theorem~\ref{thm:d^2=0}.]
The main result of~\cite{SaJoa-LOOP}
is that for each heat flow line
$u$ between critical points of
Morse index difference one
there is precisely one
Floer trajectory in the loop
space of the cotangent bundle
between corresponding critical points
of the symplectic action functional;
see~\cite[cor.~10.4~(ii)]{SaJoa-LOOP}.
Moreover, we proved that
the characteristic sign of $u$ 
coincides with the characteristic sign 
of the corresponding Floer trajectory.
In other words, both chain complexes
are \emph{equal} (up to natural identification).
Hence $\p\circ\p=0$
follows immediately from the
well known analogue
for the Floer boundary operator;
see e.g.~\cite{FLOER5,Sa-FLOER}.
(The required, but in case of 
our nongeometric potentials 
$\Vv$ slightly nonstandard
apriori $C^0$ estimate 
is provided by~\cite[thm.~5.1]{SaJoa-LOOP}
with $\eps=1$.)

The fact that heat flow
homology is independent
of the choice of orientations
$\nu$ and the regular perturbation $v$
follows from the homotopy
argument which is standard in Floer theory;
see again e.g.~\cite{FLOER5,Sa-FLOER}.
Here it is crucial to observe
that our admissible perturbations $v\in\Oo^a$
are supported away from the 
level set $\{\Ss_V=a\}$
on which the $L^2$ gradient of $\Ss_V$
(hence of $\Ss_{V+v}$) is nonvanishing and
inward pointing with respect to $\Ll^a M$.
Likewise independence
follows by theorem~\ref{thm:isomorphism}.
\end{proof}

\begin{appendix}
\section{Parabolic regularity}
\label{sec:REG}

Proofs of all results collected in this appendix
are given in~\cite{Joa-PARABOLI}, unless
specified differently.
By $\HH^-$ we denote
the closed lower half plane,
that is the set of pairs of reals $(s,t)$
with $s\le 0$.
In this section,
unless specified differently,
all maps are real-valued and the domains of 
the various Banach spaces which appear
are understood to be
either open subsets
$\Omega$ of $\R^2$ or $\HH^-$
or (cylindrical subsets of) the cylinder $Z=\R\times S^1$.
To deal with the heat equation
it is useful to consider the
anisotropic Sobolev spaces $W_p^{k,2k}$.
We call them \emph{parabolic Sobolev spaces}
and denote them by $\Ww^{k,p}$.
For constants $p\ge 1$ and integers $k\ge 0$
these spaces are defined as follows.
Set $\Ww^{0,p}=L^p$
and denote by $\Ww^{1,p}$
the set of all $u\in L^p$
which admit weak derivatives
$\p_su$, $\p_tu$, and $\p_t\p_tu$ in $L^p$.
For $k\ge2$ define
$$
     \Ww^{k,p}=\{ u\in \Ww^{1,p}\mid
     \p_su,\p_tu,\p_t\p_tu\in\Ww^{k-1,p}\}
$$
where the derivatives are again meant in
the weak sense.
The norm
\begin{equation}\label{eq:parabolic-Wk}
     \Norm{u}_{\Ww^{k,p}}
     :=\left(
     \int\int \sum_{2\nu+\mu\le 2k}
     \Abs{\p_s^\nu\p_t^\mu u(s,t)}^p\, dt ds
     \right)^{1/p}
\end{equation}
gives $\Ww^{k,p}$ the structure
of a Banach space.
Here $\nu$ and $\mu$
are nonnegative integers.
For $k=1$ we obtain that
$$
     \Norm{u}_{\Ww^{1,p}}^p
     =\Norm{u}_p^p
     +\Norm{\p_su}_p^p
     +\Norm{\p_tu}_p^p
     +\Norm{\p_t\p_tu}_p^p
$$
and occasionally we abbreviate $\Ww=\Ww^{1,p}$.
Note the difference to
(standard) Sobolev space $W^{k,p}$
where the norm is given by
$$
     \Norm{u}_{k,p}^p
     =\sum_{\nu+\mu\le k}
     \Norm{\p_s^\nu\p_t^\mu u}_p^p.
$$
A \emph{rectangular domain} 
is a set of the form $I\times J$
where $I$ and $J$ are bounded intervals.
For rectangular
(or more generally Lipschitz) domains $\Omega$
the parabolic Sobolev spaces $\Ww^{k,p}$
can be identified with the closure
of $C^\infty(\overline{\Omega})$
with respect to the
$\Ww^{k,p}$ norm;
see e.g.~\cite[appendix~B.1]{MS}.
Similarly, we define the $\Cc^k$ 
(or $\Ww^{k,\infty}$) norm by
\begin{equation}\label{eq:parabolic-Ck}
     \Norm{u}_{\Cc^k}
     :=\sum_{2\nu+\mu\le 2k}
     \Norm{\p_s^\nu \p_t^\mu u}_{\infty}.
\end{equation}

The following parabolic analogue of the Calderon-Zygmund
inequality is used to prove theorem~\ref{thm:local-regularity}
on local regularity.

\begin{theorem}[Fundamental $L^p$ estimate, 
{\cite{SaJoa-LOOP}}]
\label{thm:kerD-CalZyg}
For every $p>1$, there is a constant 
$c=c(p)$ such that
$$
     \norm{\p_sv}_p
     +\norm{\p_t\p_tv}_p
     \le c
     \norm{\p_sv-\p_t\p_tv}_p
$$
for every $v\in C_0^\infty(\R^2)$.
The same statement is true for
the domain $\HH^-$.
\end{theorem}

\begin{theorem}[Local regularity]
\label{thm:local-regularity}
Fix a constant $1<q<\infty$,
an integer $k\ge0$, and an open
subset $\Omega\subset\HH^-$.
Then the following is true.
\begin{itemize}

\item[\rm a)]
If $u\in L^1_{loc}(\Omega)$ and
$f\in\Ww^{k,q}_{loc}(\Omega)$ satisfy
\begin{equation}\label{eq:weak}
     \int_{\Omega}
     u\left(-\p_s\phi-\p_t\p_t\phi\right)
     =\int_{\Omega} f\phi
\end{equation}
for every
$\phi\in C_0^\infty(\INT\,\Omega)$,
then $u\in\Ww^{k+1,q}_{loc}(\Omega)$.

\item[\rm b)]
If $u\in L^1_{loc}(\Omega)$ and
$f,h\in\Ww^{k,q}_{loc}(\Omega)$ satisfy
\begin{equation}\label{eq:weak-divergence}
     \int_{\Omega}
     u\left(-\p_s\phi-\p_t\p_t\phi\right)
     =\int_{\Omega} f\phi
     -\int_{\Omega} h\,\p_t\phi
\end{equation}
for every $\phi\in C_0^\infty(\INT\,\Omega)$,
then $u$ and $\p_tu$ are in
$\Ww^{k,q}_{loc}(\Omega)$.
\end{itemize}
\end{theorem}

\begin{lemma}[{\cite[lemma~D.4]{SaJoa-LOOP}}]
\label{le:plus-minus}
Let $x\in C^\infty(S^1,M)$
and $p>1$. Then
$$
     \Norm{\Nabla{t}\xi}_p
     \le \kappa_p\left(
     \delta^{-1}\Norm{\xi}_p
     +\delta\Norm{\Nabla{t}\Nabla{t}\xi}_p\right)
$$ 
for $\delta>0$ and smooth
vector fields
$\xi$ along $x$.
Here $\kappa_p$ equals 
$p/(p-1)$ for $p\le 2$
and it equals $p$ for $p\ge 2$.
\end{lemma}

\begin{proposition}
\label{prop:par-linear}
Assume $u:\R\times S^1\to M$
is a smooth map such that
$\norm{\p_s u}_\infty$, $\norm{\p_t u}_\infty$,
and $\norm{\Nabla{t}\p_t u}_\infty$
are finite and
$
     \lim_{s\to\pm\infty}u(s,t)
$
exists, uniformly in $t$.
Then, for every $p>1$,
there is a constant 
$c=c(p,u,M)$ such that
\begin{equation}\label{eq:parabolic-linear}
     \Norm{\Nabla{s}\xi}_p 
     +\Norm{\Nabla{t}\xi}_p
     +\Norm{\Nabla{t}\Nabla{t}\xi}_p
    \le c \left(\Norm{\Nabla{s}\xi-\Nabla{t}\Nabla{t}\xi}_p 
     + \Norm{\xi}_p
     \right)
\end{equation}
for every smooth compactly 
supported vector field
$\xi$ along $u$.
Estimate~(\ref{eq:parabolic-linear}) 
remains valid
for $-\Nabla{s}$ replacing $\Nabla{s}$.
Estimate~(\ref{eq:parabolic-linear}) 
also remains valid if $u$ is
defined on the backward
halfcylinder $(-\infty,0]\times S^1$.
\end{proposition}

\begin{proof}
The proof of~(\ref{eq:parabolic-linear})
for $\R\times S^1$ and 
$(-\infty,0]\times S^1$ is based on
theorem~\ref{thm:kerD-CalZyg} for $\R^2$
and $\HH^-$, respectively,
using a covering argument.
Full details in the case $\R\times S^1$
are provided by~\cite[prop.~D.2]{SaJoa-LOOP}.
Lemma~\ref{le:plus-minus}
allows to add the term $\Nabla{t}\xi$
to the left hand side of~(\ref{eq:parabolic-linear}).
The underlying reason is periodicity
in the $t$ variable.
The statement for $-\Nabla{s}$ follows
by reflection $s\mapsto -s$.
\end{proof}

Applications of proposition~\ref{prop:par-linear} include
closedness of the range of the linearized
operator, proposition~\ref{prop:closed-range},
estimate~(\ref{eq:half-cylinder})
in the proof of the
exponential decay theorem~\ref{thm:par-exp-decay},
and step~2 in the proof of
theorem~\ref{thm:onto}.

\begin{lemma}[Product estimate]
\label{le:product-derivatives}
Let $N$ be a Riemannian manifold
with Levi-Civita connection $\nabla$
and Riemannian curvature tensor $R$.
Fix constants $2\le p <\infty$
and $c_0>0$.
Then there is a constant
$C=C(p,c_0,\norm{R}_\infty)$
such that
the following holds.
If $u:(a,b]\times S^1\to N$
is a smooth map such that
$$
     \Norm{\p_su}_\infty+\Norm{\p_tu}_\infty
     \le c_0,
$$
then
$$
     \left(\int_a^b\int_0^1
     \left(\Abs{\Nabla{t}\xi}
     \Abs{\Nabla{t}X}\right)^p\, dtds 
     \right)^{1/p}
     \le C\Norm{\xi}_{\Ww^{1,p}}
     \left(\Norm{\Nabla{t} X}_p
     +\Norm{\Nabla{t}\Nabla{t} X}_p\right)
$$
for all smooth
compactly supported
vector fields
$\xi$ and $X$ along $u$.
\end{lemma}

\begin{remark}\label{rmk:product-estimate}
Lemma~\ref{le:product-derivatives}
continues to hold for smooth
maps $u$ that are defined on
the whole cylinder $\R\times S^1$.
In this case
the (compact) supports of
$\xi$ and $X$
are contained in an
interval of the
form $(a,b]$.
\end{remark}

Now we fix
a closed smooth submanifold
$M\hookrightarrow\R^N$
and a smooth family of
vector-valued symmetric bilinear forms
$\Gamma:M\to\R^{N\times N\times N}$.
Abbreviate
$\Ww^{k,p}(Z)=\Ww^{k,p}(Z,\R^N)$.
Moreover, for $T>T^\prime>0$
we abbreviate
$$
     Z=Z_T=(-T,0]\times S^1,\qquad
     Z^\prime=Z_{T^\prime}=(-T^\prime,0]\times S^1.
$$

\begin{proposition}[Parabolic regularity]
\label{prop:bootstrap}
Fix constants
$p>2$, $\mu_0>1$, and $T>0$.
Fix a map $F:Z\to\R^N$ such that
$F$ and $\p_tF$ are of class $L^p$.
Assume that $u:Z\to\R^N$ is a $\Ww^{1,p}$
map taking values in $M$
with $\norm{u}_{\Ww^{1,p}}\le\mu_0$ and
such that the perturbed heat equation
\begin{equation}\label{eq:heat-local-F}
     \p_su-\p_t\p_tu
     =\Gamma(u)\left(\p_tu,\p_tu\right)
     +F
\end{equation}
is satisfied almost everywhere.
Then the following is true
for every integer $k\ge1$ such that
$F,\p_tF\in\Ww^{k-1,p}(Z)$
and every $T^\prime\in(0,T)$.
\begin{enumerate}

\item[\rm(i)]
There is a constant
$a_k$ depending on $p$, $\mu_0$,
$T$, $T^\prime$, $\norm{\Gamma}_{C^{2k+2}}$, and
the $\Ww^{k-1,p}(Z)$ norms of
$F$ and $\p_tF$ such that
$$
     \Norm{\p_tu}_{\Ww^{k,p}(Z^\prime)}
     \le a_k.
$$

\item[\rm(ii)]
If $\p_sF\in\Ww^{k-1,p}(Z)$
then there is a constant
$b_k$ depending on $p$, $\mu_0$,
$T$, $T^\prime$,
$\norm{\Gamma}_{C^{2k+2}}$,
and the $\Ww^{k-1,p}(Z)$ norms
of $F$, $\p_tF$, and
$\p_sF$ such that
\begin{equation*}
     \Norm{\p_su}_{\Ww^{k,p}(Z^\prime)}
     \le b_k.
\end{equation*}

\item[\rm(iii)]
If $\p_t\p_tF\in\Ww^{k-1,p}(Z)$
then there is a constant
$c_k$ depending on $p$, $\mu_0$,
$T$, $T^\prime$,
$\norm{\Gamma}_{C^{2k+2}}$, and
the $\Ww^{k-1,p}(Z)$ norms
of $F$, $\p_tF$, and
$\p_t\p_tF$ such that
\begin{equation*}
     \Norm{\p_t\p_tu}_{\Ww^{k,p}(Z^\prime)}
     \le c_k.
\end{equation*}

\end{enumerate}
\end{proposition}

Note that by the Sobolev embedding
theorem the assumption $p>2$ guarantees
that every $\Ww^{1,p}$ map $u$
is continuous. Hence it makes sense
to specify that $u$ takes values in
the submanifold $M$ of $\R^N$.

\begin{corollary}
\label{cor:bootstrap}
Under the assumptions of proposition~\ref{prop:bootstrap}
the following is true.
For every integer $k\ge 1$
such that $F\in\Ww^{k,p}(Z_T)$
and every $T^\prime\in(0,T)$
there is a constant 
$c_k=c_k(k,p,\mu_0,T-T^\prime,
\norm{\Gamma}_{C^{2k+2}(M)},
\norm{F}_{\Ww^{k,p}(Z_T)})$
such that
$$
     \Norm{u}_{\Ww^{k+1,p}(Z_{T^\prime})}
     \le c_k.
$$
\end{corollary}

\begin{proof}
The $\Ww^{k+1,p}$ norm of $u$
is equivalent to the 
sum of the $\Ww^{k,p}$ norms
of $u$, $\p_tu$, $\p_su$, and $\p_t\p_tu$.
Apply
proposition~\ref{prop:bootstrap}~(i--iii).
\end{proof}

\begin{lemma}[{\cite{Joa-PARABOLI}}]
\label{le:product-Sobolev}
Fix a constant $p>2$ and
a bounded open subset
$\Omega\subset\R^2$ with area $\abs{\Omega}$.
Then for every 
integer $k\ge 1$
there is a constant $c=c(k,\abs{\Omega})$
such that
$$
     \Norm{\p_tu\cdot v}_{\Ww^{k,p}}
     \le c \left(\Norm{\p_tu}_{\Ww^{k,p}}
     \Norm{v}_\infty
     +\Norm{u}_{\Cc^k}
     \Norm{v}_{\Ww^{k,p}}\right)
$$
for all functions $u,v\in
C^\infty(\overline{\Omega})$.
\end{lemma}

\end{appendix}


\end{document}